\documentclass{amsart}

\usepackage{hyperref}

\usepackage{esint}
\usepackage{amssymb}
\usepackage{tikz}
\usepackage{graphicx}
\usepackage{amsrefs,enumitem,pgfkeys}

\usepackage{fancyhdr}
\usepackage{time}
 
\newtheorem{theorem}{Theorem}
\newtheorem{lemma}[theorem]{Lemma}
\newtheorem{corollary}[theorem]{Corollary}
\newtheorem{proposition}[theorem]{Proposition}
\theoremstyle{definition}

\newtheorem{remark}[theorem]{Remark}
\newtheorem{property}[theorem]{Property}

\newtheorem{definition}[theorem]{Definition}

\numberwithin{theorem}{section}

\renewcommand{\Subset}{\subset\subset}

\newcommand{\eref}[1]{(\ref{e.#1})}
\newcommand{\tref}[1]{Theorem \ref{t.#1}}
\newcommand{\lref}[1]{Lemma \ref{l.#1}}
\newcommand{\pref}[1]{Proposition \ref{p.#1}}
\newcommand{\cref}[1]{Corollary \ref{c.#1}}
\newcommand{\fref}[1]{Figure \ref{f.#1}}
\newcommand{\sref}[1]{Section \ref{s.#1}}
\newcommand{\aref}[1]{Appendix \ref{s.#1}}

\newcommand{\rref}[1]{Remark \ref{r.#1}}

\newcommand{\Z}{\mathbb{Z}}
\newcommand{\R}{\mathbb{R}}

\newcommand{\E}{\mathbb{E}}
\newcommand{\T}{\mathbb{T}}

\newcommand{\ind}{\raisebox{.45ex}{$\chi$}}

\def\XXint#1#2#3{{\setbox0=\hbox{$#1{#2#3}{\int}$ }
\vcenter{\hbox{$#2#3$ }}\kern-.6\wd0}}

\newcommand{\ep}{\varepsilon}

\renewcommand{\div}{\operatorname{div}}

\begin{document}

\title[Gradient discontinuities and zero mobility]{The occurrence of surface tension gradient discontinuities and zero mobility for Allen-Cahn and curvature Flows in periodic media}

\author{William M Feldman}
\address{Department of Mathematics, University of Utah, Salt Lake City, USA}
\email{feldman@math.utah.edu}
\author{Peter S Morfe}
\address{Department of Mathematics, The University of Chicago, Chicago, USA}
\email{pmorfe@math.uchicago.edu}

\maketitle

\begin{abstract}
We construct several examples related to the scaling limits of energy minimizers and gradient flows of surface energy functionals in heterogeneous media.  These include both sharp and diffuse interface models.  The focus is on two separate but related issues, the regularity of effective surface tensions and the occurrence of zero mobility in the associated gradient flows.  On regularity we build on the theory of Chambolle, Goldman and Novaga \cite{GCN} to show that gradient discontinuities in the surface tension are generic for sharp interface models.   In the diffuse interface case we only show that the laminations by plane-like solutions satisfying the strong Birkhoff property generically are not foliations and do have gaps.  On mobility we construct examples in both the sharp and diffuse interface case where the homogenization scaling limit of the $L^2$ gradient flow is trivial, i.e. there is pinning at every direction.  In the sharp interface case, these are related to examples previously constructed by Novaga and Valdinoci \cite{NovagaValdinoci2011} for forced mean curvature flow.  
\end{abstract}

\section{Introduction}

\subsection{Sharp Interface Models}  The primary focus of the paper is on the analysis of sharp and diffuse heterogeneous surface energy functionals.  We start the exposition introducing the sharp interface energy on subsets $S \subset \R^d$
\begin{equation}\label{e.surfaceenergy}
 E_{a}(S;U) = \int_{\partial S \cap U} a(x)\mathcal{H}^{d-1}(dx),
 \end{equation}
where $U \subset \R^d$ is an open bounded domain and $a$ is $\Z^d$ periodic and $1 \leq a(x) \leq \Lambda$.  The energy is well defined on locally finite perimeter subsets of $\R^d$, and can also be made sense of on a closed set with finite perimeter.  The quantity of interest in this case is the effective surface tension which can be defined for each normal $n$ by
\[ \bar{\sigma}(n,a) = \lim_{T \to \infty} \frac{1}{\omega_{d-1}T^{d-1}} \inf \{ E_{a}(S;B_T(0)): \ S \cap \partial B_T(0) = \{x \cdot n \leq 0\} \cap \partial B_T(0)\}.\]
There is an equivalent definition using cubes aligned with the direction $n$. 

A basic question of interest is the regularity of the $\bar{\sigma}$.   Chambolle, Goldman and Novaga \cite{GCN} proved that, consistently with other models in Aubry-Mather theory, the differentiability properties of $\bar{\sigma} : S^{d-1} \to [1,\Lambda]$ are largely determined by the geometry of the so-called plane-like minimizers of \eqref{e.surfaceenergy}.
This leads to the observation that even restricting to smooth coefficients $a$, there are examples for which $\bar{\sigma} : S^{d-1} \to [1,\Lambda]$ has gradient discontinuities at all lattice (rational) directions.

In this note we show that the presence of gradient discontinuities at all rational directions appears not only in specially constructed examples, but is a generic feature in the topological sense.

\begin{theorem}\label{t.generic}
For each $n \in S^{d-1}$, the sets of coefficients $\mathcal{A}_n \subset C^\infty(\T^d ; [1,\Lambda])$ for which the associated surface tension $\bar{\sigma}$ has a gradient discontinuity at $n$ is open and dense in the topologies induced by $C(\T^d ; [1,\Lambda])$ and $W^{1,p}(\mathbb{T}^{d};[1,\Lambda])$ for $d < p < +\infty$.  Furthermore, $\cap_{n \in S^{d-1}} \mathcal{A}_{n}$ is also dense in those topologies.
\end{theorem}

This result will be proved below in \sref{stgaps}. Note that the coefficients $a$ in the theorem are necessarily not laminar.  Indeed, \cite{GCN} shows that the surface tension is necessarily $C^{1}$ in some non-empty open set of $S^{d-1}$ when the underlying medium is laminar.  From that point of view, the theorem shows that the general non-laminar case can be much less regular than the laminar setting.

Associated with this energy is the heterogeneous curvature flow of an evolving set $S_t$
\begin{equation}\label{e.mcfa}
 V_n = -a(x)\kappa - D a(x) \cdot n
 \end{equation}
where $n$ is the outward normal to $\partial S_t$, $V_n$ is the outward normal velocity of $\partial S_t$, and $\kappa$ is the mean curvature oriented so that convex $S$ has $\kappa\geq 0$. 

In this case we are interested in the limiting behavior of the rescaled curvature flow of $S^\ep_t$ starting from some compact initial data $S_0$, 
\begin{equation}\label{e.siepeqn} V_n = -a(\tfrac{x}{\ep})\kappa - \tfrac{1}{\ep} Da(\tfrac{x}{\ep}) \cdot n.\end{equation}
By analogy with what has been proved in related models (cf.\ \cite{barles souganidis}, \cite{BBP}, \cite{variational einstein}), one would expect that the limiting equation would be
\[ V_n = -\mu(n)\div_{\partial S}(\bar{\sigma}(n))\]
where the additional anisotropic term $\mu: S^{d-1} \to [0,\infty)$ is the mobility, the infinitesimal velocity of the system induced by additive forcing (see below for more details).  At the very best, based on the examples of $\bar{\sigma}$, we are looking at something like a crystalline curvature flow.  However the situation is more delicate than even this.

To start with, the construction of \tref{generic} also implies a certain kind of pathology at the level of the gradient flow, which we state next using the language of level set PDE for convenience:

\begin{corollary} \label{c.level set bubbling}  There is a family of coefficients $\mathcal{F} \subset C^{\infty}(\mathbb{T}^{d}; [1,\Lambda])$, which is dense in $C(\mathbb{T}^{d}; [1,\Lambda])$, such that if $a \in \mathcal{F}$, then the following statement holds: for each $u_{0} \in UC(\mathbb{R}^{d})$, if $(u^{\ep})_{\ep > 0}$ are the solutions of the level set PDE
	\begin{equation}\label{e.levelsetepeqn}
		\left\{ \begin{array}{r l}
			u^{\ep}_{t} = a(\tfrac{x}{\ep}) \textup{tr} \left( \left( \textup{Id} - \widehat{Du^{\ep}}\otimes \widehat{Du^{\ep}} \right) D^{2} u^{\ep} \right) +  \ep^{-1} Da(\tfrac{x}{\ep})\cdot Du^{\ep}   & \textup{in} \, \, \mathbb{R}^{d} \times (0,\infty), \\
			u^{\ep} = u_{0} & \textup{on} \, \, \mathbb{R}^{d} \times \{0\},
		\end{array} \right.
	\end{equation}
	then 
		\begin{equation*}
			\limsup \nolimits^{*} u^{\ep} \geq u_{0} \geq \liminf \nolimits_{*} u^{\ep}.
		\end{equation*}
	\end{corollary}  
See \sref{stgaps} for the proof. The point is that wherever the front recedes some areas of the positive phase are left behind, and wherever the front advances some areas of the negative phase remain. We refer to this phenomenon as \emph{bubbling}.  Bubbling is well known to occur in these kinds of interface motions, see Cardaliaguet, Lions, and Souganidis \cite{CLS2009}, the only part which is possibly new about \cref{level set bubbling} is that the set of coefficients for which it holds is dense.
Despite this somewhat pathological behavior, it is still conceivable that the ``bulk" of the front moves by a limiting curvature flow.  Simply instead of being a transition between the $+1$ and $-1$ phases it is actually a transition between some more complicated $+$ and $-$ phases that include the bubbles left behind by the bulk of the moving front.

	\begin{remark}
	Note that we do not prove topological genericity in \cref{level set bubbling}, only density.  The set of coefficients $\mathcal{F}$ which we construct is dense but not open in $C(\T^d;[1,\Lambda])$ and open but not dense in $C^1(\T^d;[1,\Lambda])$.  It would be interesting to know whether the occurrence of gaps in the lamination by strong Birkhoff plane-like solutions at every lattice direction, which occurs generically in the uniform topology by \tref{generic}, directly implies bubbling as in \cref{level set bubbling}.
	\end{remark}

\subsection{Interface pinning}	In fact, it can happen that the effective dynamics are in a sense ``worse" than this, the entire front can be pinned not only some compact bubbles.  Through the construction of a specific class of examples in dimension $d = 2$, we show that it is possible that the mobility $\mu(n) \equiv 0$, meaning the scaling limit is trivial.  This is exactly the phenomenon known as pinning, which occurs ubiquitously in problems involving interface motion in heterogeneous media.

Discussing formally we explain the so-called Einstein relation \cite{Kubo_1966,https://doi.org/10.1002/andp.19053220806, spohn interface dynamics}, which identifies the friction term in the effective diffusivity as the mobility, the infinitesimal response of the system to an external volume forcing.  Consider the solution $S_t(n)$ of the forced mean curvature flow for a constant large scale forcing $F \in \R$ with planar initial data
\begin{equation}\label{e.mcfaplanarforcing}
 V_n = -a(x)\kappa - D a(x) \cdot n + F \ \hbox{ with } \ S_0 = \{ x \cdot n \leq 0\}.
 \end{equation}
Then associated with each propagation direction $n$ there are minimal and maximal asymptotic speeds
 \[ c_*(n,F) = \lim_{ t \to \infty} \frac{1}{t}\inf_{x \in \partial S_t} x \cdot n\ \hbox{ and } \ c^*(n,F) = \lim_{n \to \infty} \frac{1}{t}\sup_{x \in \partial S_t} x \cdot n\] 
 which may not be the same.  It is not difficult to check that both are monotone nondecreasing in $F$ and $c_*(n,0) = c^*(n,0) = 0$.  Ignoring, for now, the possibility that the two asymptotic propagation speeds do not agree in some cases we define $c(n,F) = c_*(n,F) = c^*(n,F)$ and then define the mobility $\mu(n)$ by 
\begin{equation} \label{e.mobility} 
\mu(n) = \left.\frac{d}{dF}\right|_{F=0} c(n,F).
\end{equation}
   Although both quantities are strictly monotone in the set where they are non-zero, there can occur a nontrivial pinning interval $[-F_*(n),F^*(n)]\ni 0$ so that
 \[ c^*(n,F) = 0 \ \hbox{ for } \ F \leq F^*(n) \ \hbox{ and } \ c_*(n,F) = 0 \ \hbox{ for } \ F \geq F_*(n).\]
 In this case $\mu(n) = 0$, possibly at \emph{every} direction $n \in S^{d-1}$.  This is a well-known phenomenon which occurs in many related models of interface propagation in heterogeneous media. See for example \cite{FeldmanSmart,Feldman2021,NovagaValdinoci2011,DirrYip,DirrKaraliYip}.

Physical intuition suggests that the existence of a nontrivial pinning interval at every direction is, in some sense, a generic feature.  Again, special assumptions (e.g.\ laminarity) on the medium may produce positive mobility at some directions.  Here we prove that there exists a medium in dimension $d = 2$ with nontrivial pinning interval at every direction, actually we make an even stronger statement.

\begin{theorem} \label{t.level set pinning}  There is a medium $a \in C^{\infty}(\mathbb{T}^{2}; [1,\Lambda])$ and an $F_{a} > 0$ such that, for each $u_{0} \in UC(\mathbb{R}^{2})$, if $(u^{\ep})_{\ep > 0}$ are the solutions of the level set PDE associated with \eref{siepeqn} and the initial datum $u_{0}$ and forcing $F \in (-F_{a},F_{a})$, that is, if they are the viscosity solutions of the equations
		\begin{equation*}
			\left\{ \begin{array}{l l}
				u^{\ep}_{t} = a(\tfrac{x}{\ep}) \textup{tr} \left( \left(\textup{Id} - \widehat{Du^{\ep}} \otimes \widehat{Du^{\ep}} \right) D^{2} u^{\ep} \right) + \tfrac{1}{\ep}  Da(\tfrac{x}{\ep})\cdot Du^{\ep} + \tfrac{1}{\ep} F \|Du^{\ep}\|  & \textup{in} \, \, \mathbb{R}^{2} \times (0,\infty) \\
				u^{\ep} = u_{0} & \textup{on} \, \, \mathbb{R}^{2} \times \{0\}
			\end{array} \right.
		\end{equation*}
	then 
		\begin{equation*}
			\lim_{\ep \to 0^{+}} u^{\ep} = u_{0} \quad \textup{locally uniformly in} \, \, \mathbb{R}^{2} \times [0,\infty).
		\end{equation*}
	Furthermore, given any $\zeta > 0$, we can choose $a$ so that $\|a - 1\|_{L^{\infty}(\mathbb{T}^{2})} \leq \zeta$.  
\end{theorem}  

Taking $u_{0}$ to be a linear function in the previous theorem, we see that $a$ has a non-trivial pinning interval:

\begin{corollary}\label{c.lspc1}  If $a$ and $F_{a}$ are as in Theorem \ref{t.level set pinning}, then $c^{*}(n,F)$ and $c_{*}(n,F)$ as defined above are well-defined for each $F \in (-F_{a},F_{a})$ and $n \in S^{1}$, and $c^{*}(n,F) = c_{*}(n,F) = 0$.  \end{corollary}

The proof of Theorem \ref{t.level set pinning} is based on the construction of a medium $a$ such that stationary \emph{strict} supersolutions are plentiful.  The following lemma is the main technical result used in its proof:

\begin{lemma} \label{l.stationary solutions}
There is a medium $a \in C^{\infty}(\mathbb{T}^{2};[1,\Lambda])$, and a  (numerical) constant $C > 0$ such that if $K \subset \mathbb{R}^{2}$ satisfies an interior and exterior ball condition with large enough radius, then there is a strict stationary supersolution $S$ of \eref{mcfa} such that
	\begin{gather}
		\{x \in K \, \mid \, \textup{dist}(x,\partial K) \geq C\} \subset S, \quad S \subset K + B_{C}(0), \quad d_{H}(\partial K, \partial S) \leq C.
	\end{gather}
\end{lemma}

In terms of the evolution equation \eref{siepeqn} on regions of $\R^2$, the lemma immediately implies a quantitative pinning result for $C^{2}$ compact sets in the homogenization limit:

\begin{corollary} \label{c.pinning corollary}
If $a$ is the medium of Lemma \ref{l.stationary solutions}, $K$ is any compact subset of $\R^2$ with $C^2$ boundary, and ${S^\ep_t}$ is any solution of 
\begin{equation}\label{e.siepeqnforcing} V_n = -a(\tfrac{x}{\ep})\kappa - \tfrac{1}{\ep}Da(\tfrac{x}{\ep}) \cdot n + \tfrac{1}{\ep} F.\end{equation}
with initial data ${K}$ and forcing $F \in (-F_a,F_a)$, then there is an $\ep_{0}(K) > 0$ such that, for each $\ep \in (0,\ep_{0}(K))$,
\begin{gather*}
	\{x \in K \, \mid \, \textup{dist}(x,\partial K) \geq C \ep \} \subset \bigcap_{t \geq 0} S^{\ep}_{t}, \quad \bigcup_{t \geq 0} S^{\ep}_{t} \subset K + B_{C \ep}(0), \\
	\sup_{t \geq 0} d_{H}(\partial S^{\ep}_{t}, \partial K) \leq C \ep.
\end{gather*}
\end{corollary}

All of the above sequence of results will be proved in \sref{mobility}.

\begin{remark}
The construction in Lemma \ref{l.stationary solutions} is stable with respect to uniform norm perturbations of $a$ and $Da$ so we can also conclude that $\mu(\cdot) \equiv 0$ on an open subset of $a \in C^1(\T^2 ;[1,\Lambda])$. It is possible that there is also an open subset of $C^1(\T^2 ;[1,\Lambda])$ on which there is no pinning at any direction, although we have no evidence to suggest such a set of data exists.  The only explicitly understood case is that of laminar media of the form $a(x) = \tilde{a}(x \cdot n)$ for some $n \in S^{1}$, where the mobility is always zero at the laminar direction unless the medium is homogeneous, that property is again stable with respect to small perturbations in the $C^1(\T^2 ;[1,\Lambda])$ norm.
\end{remark}

\subsection{Diffuse Interface Models}  In the diffuse interface setting, we obtain results analogous to those in the sharp interface model by exploiting the close connection between the two models as the diffuse interface width vanishes.  To be more precise, we consider diffuse interface functionals $\mathcal{AC}^{\delta}_{\theta}$ defined on configurations $u : \mathbb{R}^{d} \to \mathbb{R}$ of the form
	\begin{equation}\label{e.diffuseenergy}
		\mathcal{AC}^{\delta}_{\theta}(u; U) = \int_{U} \left(\frac{\delta}{2} \|Du(x)\|^{2} + \delta^{-1} \theta(x) W(u(x)) \right) \, dx,
	\end{equation}
where $\theta$ is a $\mathbb{Z}^{d}$-periodic function satisfying $1 \leq \theta \leq \Lambda^{2}$, $W : \mathbb{R} \to [0,\infty)$ is a double-well potential with $\{W = 0\} = \{-1,1\}$ satisfying standard assumptions (see \eqref{A: smooth and zeros}, \eqref{A: shape of W}, \eqref{A: nondegenerate zeros} below), and $\delta > 0$ is a parameter that, roughly speaking, encodes the typical width of a (minimizing) diffuse interface.

In summary, the next results show that when $\delta$ is small enough, relative to the $C^1(\T^d)$ norm of $\theta$, $\mathcal{F}^{\delta}_{\theta}$ and its {$L^2$} gradient flow exhibit the same large scale behavior as $E_{\sqrt{\theta}}$ and its flow.

At equilibrium, the large scale behavior of $\mathcal{F}^{\delta}_{\theta}$ is described by a homogenized surface tension, as in the sharp interface case.  The following formula for the effective surface tension can be derived
\[ \overline{\sigma}_{AC}(n,\delta,\theta) = \lim_{T \to \infty} \frac{1}{\omega_{d-1}T^{d-1}}\inf\{ \mathcal{AC}^{\delta}_{\theta} (u;B_{T}(0)) : u \in \mathcal{T}(n,B_T(0))\}\]
where the admissible class is defined
\[ \mathcal{T}(n,U) = \{ u \in H^1(U): u = -\tanh(n \cdot x)  \ \hbox{ on } \ \partial U\}.\]
Note that the boundary data is just the plane separation data $\ind_{\{n \cdot x \leq 0\}} - \ind_{\{n \cdot x > 0\}}$ but smoothed out at the unit scale to avoid technical difficulties associated with the discontinuous boundary condition.

We expect the sub-differential of the effective surface tension to be characterized by the geometry of the plane-like minimizers of $\mathcal{AC}^{\delta}_{\theta}$, just as in the sharp interface case.  However, this has not yet been proved.  Toward that end, we expect the next result will be of interest.

\begin{theorem}\label{t.genericdiffuse}  For any dimension $d \geq 2$, there is a dense $G_{\delta}$ set $\mathcal{G} \subset C(\mathbb{T}^{d}; [1,\Lambda^{2}])$ such that if $\theta \in \mathcal{G}$, then there is an open set $I(\theta) \subset (0,1)$ with $0 \in \overline{I(\theta)}$ such that if $\delta \in I(\theta)$, then the following statements hold:
		\item[(i)] For each $n \in S^{d-1}$, the family $\mathcal{M}^{\delta}_{\theta}(n)$ of strongly Birkhoff plane-like minimizers of $\mathcal{AC}^{\delta}_{\theta}$ in the $n$ direction has gaps.
		\item[(ii)] Given $u_{0} \in UC(\mathbb{R}^{d}; [-3,3])$, if $(u^{\ep})_{\ep > 0}$ are the solutions of the Cauchy problem
			\begin{equation*}
				\left\{ \begin{array}{r l}
						\delta (u^{\ep}_{t} - \Delta u^{\ep}) + \ep^{-2} \delta^{-1} \theta(\ep^{-1} x)W'(u^{\ep}) = 0 & \textup{in} \, \, \mathbb{R}^{d} \times (0,\infty), \\
						u^{\ep} = u_{0} & \textup{on} \, \, \mathbb{R}^{d} \times \{0\},
				\end{array} \right.
			\end{equation*}
		then
			\begin{equation*}
				\limsup \nolimits^{*} u^{\ep} = 1 \quad \textup{in} \, \, \{u_{0} > 0\}, \quad \liminf \nolimits_{*} u^{\ep} = -1 \quad \textup{in} \, \, \{u_{0} < 0\}.
			\end{equation*}
	Furthermore, the subset of $\mathcal{G}$ consisting of $\theta$ for which $I(\theta) \supset (0,\delta_{\theta})$ for some $\delta_{\theta} > 0$ is dense in $C(\mathbb{T}^{d}; [1,\Lambda^{2}])$.
	
The statements above remain true if $C(\mathbb{T}^{d};[1,\Lambda])$ is replaced by $W^{1,p}(\mathbb{T}^{d};[1,\Lambda])$ for $p \in (d,\infty)$.
\end{theorem}

See \sref{stgaps-diffuse} for the proof. 

Although the parameter $\delta$ is useful for the statements of our theorems, it is cumbersome for the following informal discussion, so we set $\delta = 1$ for the next paragraph.  The $L^2$ gradient flow of the Allen-Cahn energy functional is well known to be
\begin{equation}\label{e.aceqn1}
 u_{t} - \Delta u  +  \theta(x) W'(u)  = 0 \  \textup{in} \, \,U \times (0,\infty).
 \end{equation}
 Considering the long time behavior of \eref{aceqn1} in the parabolic scaling leads to the equation
\begin{equation}\label{e.aceqnep}
u^{\ep}_{t} - \Delta u^{\ep} + \ep^{-2} \theta(\tfrac{x}{\ep})W'(u^{\ep}) = 0 \  \textup{in} \, \,U \times (0,\infty).
 \end{equation}
By analogy with what has been proved in related models (cf.\ \cite{barles souganidis}, \cite{BBP}, \cite{variationaleinstein}), one would expect that, in the limit $\ep \to 0$, the interface between the positive and negative phases evolves by a curvature flow
\[  V_n = -\mu_{AC}(n)\div_{\partial S}(\overline{\sigma}_{AC}(n)) \]
with the mobility $\mu_{AC}: S^{d-1} \to [0,\infty)$, as before, being the infinitesimal response to additive forcing.  We will show (see \tref{diffuse interface pinning} below) that, on an open set of coefficients in $C(\T^2)$, this homogenization limit holds but results in a trivial flow $\mu_{AC}(n) \equiv 0$.

In fact, as in the sharp interface case, we give examples of coefficients $\theta$ for which the pinning interval associated to the gradient flow dynamics is uniformly bounded below with respect to the direction, i.e. the mobility is zero at every direction, and the homogenization scaling limit \eref{aceqnep} results in a trivial flow even when an external force is added.

\begin{theorem} \label{t.diffuse interface pinning} There is an open set $\mathcal{O} \subset C(\mathbb{T}^{2}; [1,\Lambda^{2}])$ such that if $\theta \in \mathcal{O}$ and $\delta > 0$ is small enough, then there is an $F_0 > 0$ (independent of $\delta$) such that, for each $F\in (-F_0,F_0)$ and each $u_{0} \in UC(\mathbb{R}^{2}; [-3,3])$, if $u^{-}(\alpha) < u^{0}(\alpha) < u^{+}(\alpha)$ are the critical points of the potential $W(u) - \alpha u$ and $(u^{\ep})_{\ep > 0}$ are the solutions of the forced equation
	\begin{equation} \label{e.forced allen cahn}
		\left\{ \begin{array}{r l}
			\delta (u^{\ep}_{t} - \Delta u^{\ep} ) + \ep^{-2} \theta(\tfrac{x}{\ep}) \left( \delta^{-1}W'(u^{\ep}) - F \right) = 0 & \textup{in} \, \, \mathbb{R}^{2} \times (0,\infty) \\
			u^{\ep} = u_{0} & \textup{on} \, \, \mathbb{R}^{2} \times \{t=0\}
		\end{array} \right.
	\end{equation}
then, as $\ep \to 0^{+}$,
	\begin{align*}
		u^{\ep} &\to u^{+}(F\delta) \quad \textup{locally uniformly in} \, \, \{u_{0} > u^{0}(F\delta)\} \times (0,\infty), \\
		u^{\ep} &\to u^{-}(F \delta) \quad \textup{locally uniformly in} \, \, \{u_{0} < u^{0}(F \delta)\} \times (0,\infty).
	\end{align*} 
Moreover, the constant function $\theta \equiv 1$ is an accumulation point of $\mathcal{O}$ in $C(\mathbb{T}^{2};[1,\Lambda^{2}])$.
\end{theorem}  
See \sref{diffusepinning} for the proof.
\begin{remark}
It is not hard to check that this result is stable with respect to small perturbations of the coefficients in the uniform norm.  Thus we can say that this pinning phenomenon is not non-generic in the coefficient space $C(\T^2;[1,\Lambda^2])$.
\end{remark}

\begin{remark}
In the theorem, the reaction term $\theta(x)(W'(u) - F)$ appearing in \eref{forced allen cahn} satisfies 
	\begin{equation*}
		\left|\int_{u^{-}(F\delta)}^{u^{+}(F\delta)} \theta(x) (\delta^{-1} W'(u) - F) \,du\right|     \geq c|F|   \quad \textup{for each} \, \, F \in (-F_{0},F_{0}).
	\end{equation*}
Thus, the result is a manifestation of the front-blocking phenomenon in the study of bistable reaction diffusion equations (cf. Lewis and Keener \cite{MR1776397}).  In fact, we show below that, for $\theta \in \mathcal{O}$, \eref{forced allen cahn} has stationary, plane-like solutions for all forcing values $F$ in this interval (see Remark \ref{R: plane like}).

\begin{remark} \label{R: other energies} Our arguments also apply to diffuse interface energies where the heterogeneity appears on the gradient rather than the potential term, as well as those where it multiplies both.  That is, the same results apply to energies of the form
	\begin{gather}
		\int_{U} \left(\frac{\delta}{2} \theta(x) \|Du(x)\|^{2} + \delta^{-1} W(u(x)) \right) \, dx, \label{E: gradient model} \\
		\int_{U} \left( \frac{\delta}{2} \|Du(x)\|^{2} + \delta^{-1} W(u(x)) \right) \sqrt{\theta(x)} \, dx. \label{E: weight model}
	\end{gather}
For more details, see Remarks \ref{R: other energies construction} and \ref{R: other energies generic} below.
\end{remark}  

\end{remark}
\subsection{Literature}  

The study of variational models in periodic media falls under the broad umbrella of Aubry-Mather theory, named after the fundamental contributions of Aubry and LeDaeron \cite{aubry_le-daeron} and Mather \cite{mather}, who investigated the (discrete) Frenkel-Kontorova model and twist maps.  In Aubry-Mather theory, one of the main questions is the existence and structure of ``plane-like" minimizers and its relation to the large-scale (or homogenized) behavior of the energy itself.  The investigation of continuum models via PDE methods was initiated by Moser \cite{Moser} with the fundamental structural theorems contributed by Bangert \cite{Bangert1,Bangert2}.  

The results of Moser and Bangert concern graphical energies modeled on the Dirichlet energy.  In more recent years, variational problems with more of a geometrical flavor have been shown to possess the same basic structure.  Caffarelli and de la Llave \cite{CdlL2001} extended the basic existence results of Aubry-Mather theory to surface energies like those considered here.  There has also been considerable interest in diffuse interface energies, including contributions by many authors.  For references and connections to the work of Moser and Bangert, see the book of Rabinowitz and Stredulinsky \cite{RSbook} and the expository paper by Junginger-Gestrich and Valdinoci \cite{JGV09}.

Chambolle, Goldman, Novaga \cite{GCN} studied the effective energy for the sharp interface model, giving a precise characterization of the differentiability properties in terms of the existence (or not) of gaps in the corresponding laminations by plane-like solutions.  They also gave specific examples where the effective surface tension has gradient discontinuities at every direction satisfying a rational relation.  Ruggiero \cite{Ruggiero04} and Pacheco and Ruggiero \cite{PachecoRuggiero} showed that media with gaps in the lamination at every direction are residual (i.e.\ they form a dense $G_{\delta}$ set) in two dimensions in the $C^1$ and $C^{1,\beta}$ norms, respectively.  Our result \tref{generic} shows that the gap phenomenon is residual in higher dimensions as well, at least for the rational directions, but only in the uniform norm.  
It would be interesting to obtain a similar result for all directions satisfying a rational relation or, even better, all directions, and in stronger topologies.

An analogous connection between gaps in the laminations by plane-like solutions and surface tension regularity has not yet been established in the case of diffuse interfaces, see \cite{variationaleinstein,cflv} for partial results in this direction.

 The front bubbling phenomenon, as discussed in \cref{level set bubbling}, has also been known for some time, examples were constructed for forced mean curvature flow by Dirr, Karali and Yip \cite{DirrKaraliYip} and by Cardaliaguet, Lions and Souganidis \cite{CLS2009}.  Novaga and Valdinoci \cite{NovagaValdinoci2011}, in the setting of the forced mean curvature flow with homogeneous perimeter, have shown that bubbling as in \cref{level set bubbling} occurs generically with respect to the $L^1$ distance on the coefficient field in dimension $2$.  Note that this type of genericity is quite similar to our result because their forcing term corresponds to $Da$ in our setting.
 
Front pinning is another well known and fundamental feature of interface propagation in heterogeneous media, and has been studied for many related models in both periodic and random media \cite{DirrYip,Feldman2021,BCN2011,MR2581044,MR3719956,MR2846018,kardar,courte2021proof}.  In the reaction diffusion literature this is referred to as wave blocking  \cite{MR988617,MR1776397}.  Examples of front pinning have been constructed for various models in both periodic and random media: for the forced quenched Edwards-Wilkinson equation, Dirr and Yip \cite{DirrYip} have shown that pinning is generic and they have also constructed pinning examples for forced Allen-Cahn (homogeneous energy, but heterogeneous volume forcing) in one dimension.  The first author \cite{Feldman2021} gave examples of front pinning at every direction for the Bernoulli free boundary problem in heterogeneous media. Novaga and Valdinoci \cite{NovagaValdinoci2011}, which showed a similar result to our \cref{level set bubbling} in the context of forced mean curvature flow (homogeneous surface energy, but heterogeneous volume forcing), does not explicitly give an example of pinning of the entire interface (as in our \cref{pinning corollary}), we believe that a small modification of their ideas would also yield an example of pinning in $2$-d. We also were recently made aware of a paper by Courte, Dondl and Ortiz \cite{courte2021proof} which considers a curvature driven motion with dry friction in random media with sparse obstacles.  They show the occurrence of additional pinning by the Poissonian obstacles and establish the precise scaling exponent of the additional pinning in the sparse obstacle limit.  Their fundamental barrier construction bears significant similarity to ours, patching together barrier pieces near concentrated obstacles (the example of \cite[Section 5.2]{Feldman2021} is also similar, but patching barriers is easier due to the particular nature of that problem).

In the context of bistable reaction diffusion equations in one-dimensional periodic media, Xin \cite{MR1251222} and Ding, Hamel, and Zhao \cite{MR3412383} have constructed unbalanced reaction terms for which one nonetheless finds plane-like stationary solutions, a phenomenon associated with pinning.  Our results provide (non-laminar) examples of this in two dimensions.

 \subsection{Organization of the paper}
 We begin in \sref{preliminaries} with some background on viscosity solution theory of interface motions and related concepts.  Next, in \sref{mobility}, we construct a special class of ``highly pinning" media and use the construction to prove \tref{level set pinning}, \cref{lspc1}, \lref{stationary solutions}, and \cref{pinning corollary}.  In \sref{diffusepinning}, we prove \tref{diffuse interface pinning}, obtaining analogous interface pinning results for diffuse interface models.  In \sref{stgaps}, we prove \tref{generic} and \cref{level set bubbling} on the genericity of surface tension gradient discontinuities for sharp interface models. In \sref{stgaps-diffuse}, we prove \tref{genericdiffuse}, the analogous result on genericity of surface tension singularities for diffuse interface models.

\subsection{Acknowledgments}  Both authors would like to thank Takis Souganidis for helpful discussions and comments throughout preparation of this paper.  William Feldman would like to acknowledge the support of NSF grant DMS-2009286.  Peter Morfe was partially supported by P.E. Souganidis's NSF grants DMS-1600129 and DMS-1900599.


\section{Preliminaries}\label{s.preliminaries}

\subsection{Viscosity Solutions and Level-Set Formulation} Given a positive periodic function $a \in C^{1}(\mathbb{T}^{d})$ and a force $F \in \mathbb{R}$, we are interested in sets moving with normal velocity given by 
	\begin{equation} \label{E: flow definition}
		V_{n} = - a(x) \kappa - Da(x) \cdot n + F.
	\end{equation}
We will use the level-set formulation, which amounts to studying the following nonlinear PDE:
	\begin{equation} \label{E: level set definition}
		u_{t} - a(x) \text{tr} \left( \left( \text{Id} - \widehat{Du} \otimes \widehat{Du} \right) D^{2} u \right) - Da(x) \cdot Du - F\|Du\| = 0.
	\end{equation}
Above $\widehat{Du}$ is a shorthand for the normal vector $\widehat{Du} = \|Du\|^{-1}Du$.  

The level-set formulation allows one to define (weak) solutions of \eqref{E: flow definition} using viscosity solutions of \eqref{E: level set definition}.  Very roughly speaking, we say that a family of sets $(S_{t})_{t \geq 0}$ is a viscosity solution of \eqref{E: flow definition} if the characteristic function $u(x,t) = \chi_{S_{t}}(x)$ \footnote{Here and in what follows $\chi_{S}(x) = 1$ if $x \in S$ and $\chi_{S}(x) = 0$, otherwise.} determines a viscosity solution of \eqref{E: level set definition}.  We will not make this completely precise here since the main technical constructions of the paper only require time-stationary solutions.  For a full account of the level-set theory, see the original survey articles \cite{barlessouganidis} and \cite{barlessonersouganidis} and the textbook \cite{giga}.

\subsection{Stationary Viscosity Solutions} The basic plan of attack of the article is the construction of (time-stationary) sets that act as barriers of the evolution, that is, we are interested in sets solving the equation
	\begin{equation} \label{E: stationary flow}
		- a(x) \kappa - Da(x) \cdot n + F = 0
	\end{equation}
or, equivalently, functions $u$ solving the PDE
	\begin{equation} \label{E: stationary level set PDE}
		- a(x) \text{tr} \left( \left( \text{Id} - \widehat{Du} \otimes \widehat{Du} \right) D^{2} u \right) - Da(x) \cdot Du - F \|Du\| = 0.
	\end{equation}  
This is made precise next.  To simplify the notation here and in what follows, we will define the mean-curvature operators $\mathcal{MC}_{*}, \mathcal{MC}^{*} : \mathbb{R}^{d} \times \mathcal{S}_{d} \to \mathbb{R}$ by
	\begin{align*}
		\mathcal{MC}_{*}(p,X) &= \left\{ \begin{array}{r l}
					\text{tr} \left( \left( \text{Id} - \|p\|^{-2} p \otimes p \right) X\right), & \text{if} \, \, p \neq 0, \\
					\min \left\{ \text{tr} \left( \left( \text{Id} - e \otimes e \right) X \right) \, \mid \, e \in S^{d-1} \right\}, & \text{otherwise.}
				\end{array} \right. \\
		\mathcal{MC}^{*}(p,X) &= \left\{ \begin{array}{r l}
					\text{tr} \left( \left( \text{Id} - \|p\|^{-2} p \otimes p \right) X \right), & \text{if} \, \, p \neq 0, \\
					\max \left\{ \text{tr} \left( \left( \text{Id} - e \otimes e \right) X \right) \, \mid \, e \in S^{d-1} \right\}, & \text{otherwise.}
				\end{array} \right.
	\end{align*}
We begin by defining precisely what it means for a function to be a time-independent viscosity super- or subsolution of \eqref{E: level set definition}.  

\begin{definition} Given an open set $U \subset \mathbb{R}^{d}$, we say that a lower (resp. upper) semi-continuous function $u : U \to \mathbb{R}$ is a \emph{viscosity supersolution} (resp.\ \emph{subsolution}) of \eqref{E: stationary level set PDE} if the following property holds: given any point $x_{0} \in U$ and any smooth function $\psi$ defined in a neighborhood of $x_{0}$, if $u - \psi$ has a local minimum (resp.\ maximum) at $x_{0}$, then
	\begin{align*}
		- a(x_{0}) \mathcal{MC}_{*}(D\psi(x_{0}),D^{2}\psi(x_{0})) - Da(x_{0}) \cdot D\psi(x_{0}) - F \|D\psi(x_{0})\| \geq 0 \\
		(\text{resp.} \, \, -a(x_{0}) \mathcal{MC}^{*}(D\psi(x_{0}),D^{2}\psi(x_{0})) - Da(x_{0}) \cdot D\psi(x_{0}) - F \|D\psi(x_{0})\| \leq 0).
	\end{align*}
	
As an abbreviation, we write that $-a(x) \mathcal{MC}_{*}(Du,D^{2}u) - Da(x) \cdot Du - F\|Du\| \geq 0$ holds in $U$ if $u$ is a viscosity supersolution of \eqref{E: stationary level set PDE} in $U$.  A similar convention will be used for subsolutions.
\end{definition}

Roughly speaking, a set $S \subset \mathbb{R}^{d}$ is now a subsolution or supersolution of \eqref{E: flow definition} precisely when its characteristic function $\chi_{S}$ is a viscosity subsolution or supersolution as above.  The precise definition is given next.
	
\begin{definition} Given an open set $U \subset \mathbb{R}^{d}$ and a closed set $S \subset \overline{U}$, we say that $S$ is a \emph{supersolution} of \eqref{E: stationary flow} in $U$ if the characteristic function $u = \chi_{\text{Int}(S)}$ is a viscosity super-solution of \eqref{E: stationary level set PDE} in $U$.  

Similarly, given an open set $U \subset \mathbb{R}^{d}$ and an open set $V \subset U$, $V$ is called a \emph{subsolution} of \eqref{E: stationary flow} if the function $v = \chi_{\overline{V}}$ is a viscosity subsolution of \eqref{E: stationary level set PDE} in $U$.

Finally, if $S \subset \mathbb{R}^{d}$ is such that the closure $\overline{S}$ is a supersolution of \eqref{E: stationary flow} and the interior $\text{Int}(S)$ is a subsolution of \eqref{E: stationary flow}, then we say that $S$ is a \emph{solution} of \eqref{E: stationary flow}. \end{definition} 

It is useful to bear in mind that, in the case of sets, the differential inequalities involved in the definition of viscosity sub- and supersolution only need to be checked on the boundary.  

	\begin{proposition} \label{P: boundary thing} Given an open set $U \subset \mathbb{R}^{d}$ and a closed set $S \subset \overline{U}$, $S$ is a supersolution in $U$ if and only if, for every $x_{0} \in \partial S \cap U$ and every smooth function $\psi$ defined in a neighborhood of $x_{0}$, if $\chi_{\text{Int}(S)} - \psi$ has a local minimum at $x_{0}$, then 
		\begin{equation*}
			-a(x_{0}) \mathcal{MC}^{*}(D\psi(x_{0}),D^{2}\psi(x_{0})) - Da(x_{0}) \cdot D\psi(x_{0}) - F\|D\psi(x_{0})\| \geq 0.
		\end{equation*}
	\end{proposition}
	
		\begin{proof} The ``only if" is immediate.  To see that the ``if" direction holds, recall that $S = \partial S \cup \text{Int}(S)$.  Further, if $x_{0} \in \text{Int}(S)$, then $\chi_{\text{Int}(S)} \equiv 1$ in a neighborhood of $x_{0}$ and then the desired differential inequality follows directly from the second derivative test. \end{proof}
		
Finally, we make frequent use of the fact that the minimum of two supersolutions is once again a supersolution.  The next result states this in the context of sets.

	\begin{proposition} \label{P: minimum} Given an open set $U \subset \mathbb{R}^{d}$, if the closed sets $S_{1}, S_{2} \subset \mathbb{R}^{d}$ are supersolutions of \eqref{E: stationary flow} in $U$, then $S_{1} \cap S_{2}$ is also a supersolution in $U$.\end{proposition}
	
		\begin{proof} Note that $\chi_{\text{Int}(S_{1} \cap S_{2})} = \min\{\chi_{\text{Int}(S_{1})}, \chi_{\text{Int}(S_{2})}\}$.  Thus, the result follows from the well-known fact that the minimum of two supersolutions of \eqref{E: stationary level set PDE} is again a supersolution. \end{proof}
		
\subsection{Diffuse Interfaces} In this section, we briefly provide definitions of solutions of our diffuse interface evolution equations.  Recall that the equation of interest, after scaling out the $\epsilon$ and $\delta$ variables and restriction to time-stationary solutions, is
	\begin{equation} \label{E: stationary diffuse interface}
		- \Delta u + \theta(x) \left( W'(u) - F \right) = 0.
	\end{equation}
As in the last section, we will only define time-stationary sub- and supersolutions since those are all that is needed for our main technical results.  (For relevant definitions for the time-dependent equation, see, e.g., \cite{user}.)

	\begin{definition} Given an open set $U \subset \mathbb{R}^{d}$, we say that a lower (resp.\ upper) semi-continuous function $u : U \to \mathbb{R}$ is a \emph{viscosity supersolution} (resp.\ \emph{subsolution}) of \eqref{E: stationary diffuse interface} in $U$ if, for any $x_{0} \in U$ and any smooth function $\psi$ defined in a neighborhood of $x_{0}$, if $u - \psi$ has a local minimum (resp.\ maximum) at $x_{0}$, then
		\begin{equation*}
			- \Delta \psi(x_{0}) + \theta(x_{0}) \left( W'(u(x_{0})) - F \right) \geq 0 \, \, \, \text{(resp.} \leq 0\text{)}.
		\end{equation*}\end{definition}
		
\subsection{Half-relaxed Limits} We next recall the definition of half-relaxed limits, a notion that has been used extensively in the theory of viscosity solutions since it was introduced by Barles and Perthame \cite{barles_perthame}.  

\begin{definition} \label{D: half-relaxed} Given a family of functions $(u_{\epsilon})_{\epsilon > 0}$ in $\mathbb{R}^{d}$, the \emph{upper and lower half-relaxed limits}, $\limsup^{*} u_{\epsilon}$ and $\liminf_{*} u^{\epsilon}$, are the functions in $\mathbb{R}^{d}$ defined by 
	\begin{align*}
		\limsup \nolimits^{*} u_{\epsilon}(x) &= \lim_{\delta \to 0^{+}} \sup \left\{ u_{\epsilon}(y) \, \mid \, \|x - y\| + \epsilon < \delta \right\}, \\
		\liminf \nolimits_{*} u_{\epsilon}(x) &= \lim_{\delta \to 0^{+}} \inf \left\{ u_{\epsilon}(y) \, \mid \, \|x - y\| + \epsilon < \delta \right\}.
	\end{align*}
\end{definition}

See, for instance, the notes by Barles \cite{barles_introduction} for properties of half-relaxed limits and some of their applications.

\section{Sharp interfaces: medium with a non-trivial pinning interval at every direction}\label{s.mobility}
We will construct a medium $a(x) \in C^{\infty}(\T^2 ; [1,\Lambda])$ which has a plane-like stationary solution $S_e$ of
\begin{equation}\label{e.sieqn}
 -a(x) \kappa -Da(x) \cdot n +F = 0 
 \end{equation}
 and
 \[  \partial S \subset \{-C \leq x \cdot e \leq C\}  \] 
at every direction $e$ and for every forcing $|F| \leq F_{a}$.  This implies that the mobility, defined by \eref{mobility}, has
\[ \mu(e,a) = \left.\frac{d}{dF}\right|_{F=0} c(e,F,a) = 0 \ \hbox{ for all } \ e \in S^{1}.\]
However, since we do not have any general theorem establishing the relationship between this mobility and the homogenization of \eref{siepeqn}, we prove a slightly more general result, stated above in \lref{stationary solutions} and \cref{pinning corollary}: within a unit distance of any sufficiently regular set $K \subset \R^2$, there is a stationary solution of \eref{sieqn}.

The bulk of the work consists in proving the existence of certain sub- and supersolutions.  Toward that end, here is a statement of the main result of this section:

\begin{lemma} \label{L: sub and supersolutions} There is an $R > 0$ (a numerical constant) and an $a \in C^{\infty}(\mathbb{T}^{2}; [1,\Lambda])$ for which there is an $F_{a} > 0$ and a $C > 0$ such that the following holds: if $K \subset \mathbb{R}^{2}$ satisfies an exterior and interior ball condition of radius $R$, then there is a supersolution $S^{*}(K)$ of \eqref{e.sieqn} with $F = F_{a}$ and a subsolution $S_{*}(K)$ of \eqref{e.sieqn} with $F = -F_{a}$ such that
	\begin{equation*}
			S_{*}(K) \subset K \subset S^{*}(K), \quad d_{H}(S_{*}(K),S^{*}(K)) \leq C, \quad d_{H}(\partial S_{*}(K), \partial S^{*}(K)) \leq C.
	\end{equation*}
Furthermore, we can assume that $S_{*}(K)$ and $S^{*}(K)$ have piecewise smooth boundaries.
\end{lemma}

Once Lemma \ref{L: sub and supersolutions} is proved, we invoke a version of Perron's Method to establish the existence of solutions (Lemma \ref{l.stationary solutions}).  One could also construct a solution near $K$ by constrained energy minimization and, as a result, find a pinned ``local energy minimizer."  While that is not the approach taken here, it is worth pointing out for the purpose of exposition that these pinned solutions are not just unstable energy critical points.

As the reader familiar with homogenization will likely realize, Lemma \ref{L: sub and supersolutions} implies homogenization of the gradient flow (Theorem \ref{t.level set pinning}).  For completeness, the details are provided at the end of the section.  

The construction of stationary solutions relies on an explicit construction of super and subsolutions -- found below in Section \ref{S: edges and nodes}.  These two cases are symmetric so we will only need to handle constructing supersolutions.  At a technical level this involves repeated patching together of supersolution pieces, in the $d=2$ case we consider these are just concatenations of curves.  However in service of a possible future generalization to higher dimensions, and to separate the arguments involving patching from the actual explicit construction, we will start by setting up a framework which handles the topological issues of the patching procedure.  Additionally, since this patching will proceed using lattice cubes, we will use certain facts about sets consisting of unions of such cubes.

{The idea of the construction is we show that it is possible to find a coefficient $a$ together with certain curves that are stationary supersolutions of \eqref{e.mcfa} on their boundaries.  These curves are chosen so that they can be combined to bound any union of lattice cubes --- see Figure \ref{f.basicidea} for the basic picture to keep in mind.  The arguments showing that such an $a$ can be found are in Section \ref{S: edges and nodes} --- the reader may wish to start there --- while the remainder of the section formalizes the construction.}

\begin{figure}
\includegraphics[scale=0.4]{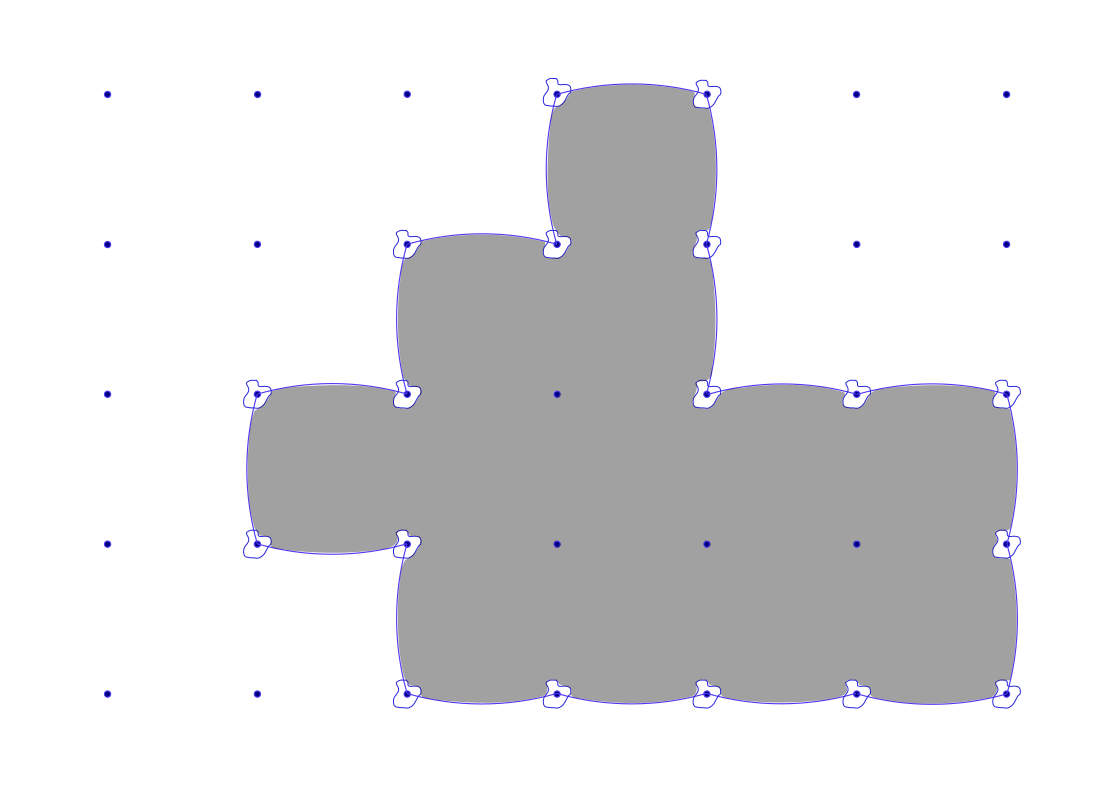}
\caption{A depiction of one of the supersolutions obtained through our construction.  The shaded area is the interior of the supersolution.  The dots indicate the points of the lattice $\mathbb{Z}^{2}$.  Notice that the boundary of the supersolution consists of translated copies of certain basic curves, some of which are circular arcs connecting lattice points (called ``edges" in what follows) and others that are simple closed curves surrounding one (called ``nodes").  }
\label{f.basicidea}
\end{figure}

\subsection{Regular $\mathbb{Z}^{2*}$-measurable sets}  The construction of sub- and supersolutions uses the fact that smooth subsets of $\mathbb{R}^{2}$ admit nice discrete approximations.  This leads us to define regular $\mathbb{Z}^{2*}$-measurable sets.  

In what follows, $\mathbb{Z}^{2*}$ is the dual lattice of $\mathbb{Z}^{2}$, that is, $\mathbb{Z}^{2*} = \mathbb{Z}^{2} + (1/2,1/2)$.  This is a convenient way of indexing the lattice cubes $\{z + [-1/2,1/2]^{2}\}_{z \in \mathbb{Z}^{2*}}$ that will be used in our approximations of smooth sets.  These approximations will consist of unions of such cubes, as in the next definition:

\begin{definition}  We say that $A \subset \mathbb{R}^{2}$ is \emph{$\mathbb{Z}^{2*}$-measurable} if there is a $Z_{A} \subset \mathbb{Z}^{2*}$ such that
	\begin{equation*}
		A = \bigcup_{z \in Z_{A}} (z + [-1/2,1/2]^{2}).
	\end{equation*}
\end{definition}  

The boundary of any $\mathbb{Z}^{2*}$-measurable set is a union of paths in a certain graph.  This will be convenient in the formalism that follows.  By the graph $(\mathbb{Z}^{2},\mathbb{E}^{2})$, we mean the set $\mathbb{Z}^{2} \subset \mathbb{R}^{2}$ together with \emph{directed} edges $[x,y]$ consisting of the line segment connecting two points $x,y \in \mathbb{Z}^{2}$ with $\|x - y\| = 1$.  We identify $[x,y]$ with the oriented line segment, oriented so that its tangent vector is parallel to $y - x$.  The normal vector $n_{[x,y]}$ to this line segment is defined by rotating the tangent vector counter-clockwise (hence, e.g., $n_{[(0,0),(1,0)]} = (0,1)$).  

In the discussion that follows, it will be useful to say that $z,z' \in \mathbb{Z}^{2*}$ are \emph{neighbors} if either $z = z'$, $[z,z'] \in \mathbb{E}^{2}$, or $[z',z] \in \mathbb{E}^{2}$.  We say that they are \emph{wired neighbors} if $\|z - z'\|_{\ell^{\infty}} \leq 1$.

Given a $\mathbb{Z}^{2*}$-measurable set $A$ corresponding to the points $Z_{A} \subset \mathbb{Z}^{2*}$, we define the associated boundary cubes $A^{b}$ and interior cubes $A^{\textup{int}}$ by 
	\begin{align*}
		A^{b} &= \bigcup \left\{z + [-1/2,1/2]^{2} \, \mid \, z \in Z_{A}, \, \, (z + [-1/2,1/2]^{2}) \cap \partial A \neq \emptyset \right\}, \\
		A^{\textup{int}} &= \bigcup \left\{z + [-1/2,1/2]^{2} \, \mid \, z \in Z_{A}, \, \, z + [-1/2,1/2]^{2} \subset \textup{Int}(A) \right\}.
	\end{align*}

We will restrict our attention to a particularly nice class of $\mathbb{Z}^{2*}$-measurable sets for which the boundary $\partial A$ equals the image of simple paths in $(\mathbb{Z}^{2},\mathbb{E}^{2})$.  Toward that end, the following definition will be convenient:

	\begin{definition} \label{D: regular cube sets}  A set $A \subset \mathbb{R}^{2}$ is said to be a \emph{regular $\mathbb{Z}^{2*}$-measurable set} if
		\begin{itemize}
			\item[(i)] For each $z + [-1/2,1/2]^{2} \subset A^{b}$ with $z \in \mathbb{Z}^{2*}$, there is a $z' + [-1/2,1/2]^{2} \subset A^{\textup{int}}$ such that $z' \in \mathbb{Z}^{2*}$ is a wired neighbor of $z$.
			\item[(ii)] For each $z + [-1/2,1/2]^{2} \subset A^{b}$ with $z \in \mathbb{Z}^{2*}$, if $z' + [-1/2,1/2]^{2} \subset A^{b}$ and $z'$ is a wired neighbor of $z$, then there is a $z'' \in \mathbb{Z}^{2*}$ such that $z'' + [-1/2,1/2]^{2} \subset A$ and $z''$ is a neighbor of both $z$ and $z'$.
		\end{itemize}  
	\end{definition}  
	
A topological argument proves that regular $\mathbb{Z}^{2*}$-measurable sets are determined by simple paths.  Precisely, given an interval $E \subset \mathbb{Z}$, we say that $\gamma : E \to \mathbb{Z}^{2}$ is a path in $(\mathbb{Z}^{2},\mathbb{E}^{2})$ if, for each $i, i + 1 \in E$, $[\gamma(i),\gamma(i+1)] \in \mathbb{E}^{2}$.  It is a simple path when $\gamma(i) = \gamma(j)$ only if $i = \min E$ and $j = \max E$.  When $\gamma(\min E) = \gamma(\max E)$, we say that $\gamma$ is closed.  A path that is both simple and closed is called a simple closed path.  

For convenience, we denote by $\{\gamma\} \subset \mathbb{R}^{2}$ the path traced out by $\gamma$, that is, $\{\gamma\} = \bigcup_{\{i,i+1\} \subset E} [\gamma(i),\gamma(i + 1)]$, where the edges $[\gamma(i),\gamma(i + 1)]$ are identified with the corresponding line segments in $\mathbb{R}^{2}$.

The next result shows that the boundary of a regular $\mathbb{Z}^{2*}$-measurable sets is the image of a disjoint union of simple paths in the graph $(\mathbb{Z}^{2},\mathbb{E}^{2})$.  Later, the orientation of the paths will be important.  Hence before stating the result, we define what it means for a path to traverse the boundary of a regular $\mathbb{Z}^{2*}$-measurable set in a clockwise fashion.

Given a $\mathbb{Z}^{2*}$-measurable set $A$ and an edge $[x,y] \in \mathbb{E}^{2}$ such that $[x,y] \subset \partial A$, we say that $[x,y]$ traverses $\partial A$ clockwise if $n_{[x,y]}$ is parallel to the outward normal vector to $\partial A$, or, more precisely,
	\begin{equation*}
		\left\{\frac{1}{2}(x + y) + t n_{[x,y]} \, \mid \, t \in (0,1)\right\} \subset \mathbb{R}^{2} \setminus A.
	\end{equation*}
Otherwise, we say that $[x,y]$ traverses $\partial A$ counter-clockwise.  Notice that $[x,y]$ traverses $\partial A$ clockwise if and only if $[y,x]$ traverses it counter-clockwise.

Given a $\mathbb{Z}^{2*}$-measurable set $A$ and an interval $E \subset \mathbb{Z}$, we say that a path $\gamma : E \to \mathbb{Z}^{2} \cap \partial A$ traverses $\partial A$ clockwise if, for each $\{i, i + 1\} \subset E$, the edge $[\gamma(i),\gamma(i + 1)]$ traverses $\partial A$ clockwise.

	\begin{theorem} \label{T: boundary curve} If $A$ is a regular $\mathbb{Z}^{2*}$-measurable set, then there is a pairwise disjoint collection $(\gamma^{(j)})_{j \in P}$ of simple paths in $\mathbb{Z}^{2}$ indexed by some $P \subset \mathbb{N}$ such that
		\begin{equation*}
			\partial A = \bigcup_{j \in P} \{\gamma^{(j)}\}
		\end{equation*}
	and, for each $j \in P$, the path $\gamma^{(j)}$ traverses $\partial A$ clockwise.  Furthermore, the finite length paths in $(\gamma^{(j)})_{j \in P}$ are all simple closed paths.\end{theorem}  
	
	A sketch of the proof is given next; the details can be found in Appendix \ref{A: boundary construction}.
	
		\begin{proof}[Sketch of proof]  Start at a boundary vertex $x_{0} \in \partial A \cap \mathbb{Z}^{2}$.  By (ii) in the definition of regularity, there is a unique neighbor $x_{1} \in \partial A \cap \mathbb{Z}^{2}$ of $x_{0}$ such that $[x_{0},x_{1}] \subset \partial A$ and the (oriented) edge $[x_{0},x_{1}]$ traverses $\partial A$ clockwise.  Apply this procedure again with $x_{1}$ replacing $x_{0}$, furnishing a neighbor $x_{2}$ of $x_{1}$ so that $[x_{1},x_{2}]$ traverses $\partial A$ clockwise.  Iterating this results in a path $\gamma : \mathbb{N} \cup \{0\} \to \mathbb{Z}^{2}$ given by $\gamma(j) = x_{j}$.  If the image $\{\gamma\}$ is finite, then it is necessarily a simple closed path; otherwise, in the infinite case, it is a simple path.  
		
		If $\partial A \setminus \{\gamma\}$ is nonempty, repeat the construction at a different boundary vertex.  Since $\mathbb{Z}^{2}$ is countable, this process eventually decomposes $\partial A$ as a countable union of pairwise disjoint simple paths in $\mathbb{Z}^{2}$, as claimed. \end{proof}

\subsection{Abstract framework for patching super/subsolutions} \label{S: patching} We set up an abstract framework which is useful to compartmentalize arguments relating to patching together local smooth super/subsolutions to form a global super/subsolution.  We consider only the two dimensional case, and we do not at all consider a fully general notion of admissible patching.  We expect that, with significantly more topological work, these ideas could be generalized and would be useful for constructing examples in higher dimensions.

Since we will be combining sets that only satisfy the supersolution property locally, it will be convenient to track the domain together with the set.  That is the purpose of the next definition.

\begin{definition} Given an open set $U \subset \mathbb{R}^{2}$ and a closed set $S \subset \overline{U}$, we say that the pair $(S,U)$ is a \emph{local supersolution} if $S$ is a supersolution of \eref{sieqn} in $U$.\end{definition}

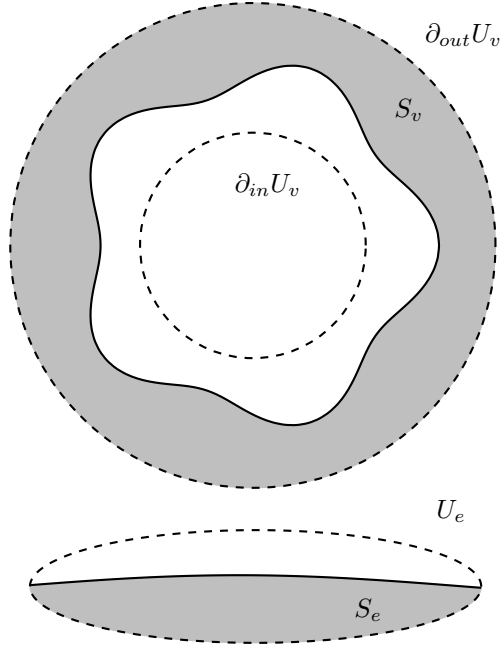
\begin{figure}[t]

\begin{tikzpicture}[scale = .75]

\draw[draw=gray!50!white,fill=gray!50!white]  
plot[smooth,samples=100,domain=0:360] ({4.3*cos(\x)},{4.3*sin(\x)}) --  
plot[smooth,samples=100,domain=360:0] ({(3+.3*cos(5*\x))*cos(\x)},{(3+.3*cos(5*\x))*sin(\x)}); 

\draw[thick] plot[smooth,samples=100,domain=360:0] ({(3+.3*cos(5*\x))*cos(\x)},{(3+.3*cos(5*\x))*sin(\x)}) node[xshift=-.4cm, yshift=1.8cm] {$S_v$}; 

\draw[thick, dashed] (0,0) circle (4.3cm) node[xshift=2.8cm, yshift=2.8cm] {$\partial_{out}U_v$};
\draw[thick, dashed] (0,0) circle (2cm) node[xshift=.2cm, yshift=.8cm] {$\partial_{in}U_v$}; 

\end{tikzpicture}

\begin{tikzpicture}[scale = .75]

\draw[draw=gray!50!white,fill=gray!50!white] plot[smooth,samples=100,domain=0:8] (\x+.3,{.2*sin(180*(\x+.3)/8)}) --
plot[smooth,samples=100,domain=0:-180] ({4.3+4*cos(\x)},{sin(\x)}); 

\draw[thick] plot[smooth,samples=100,domain=0:8] (\x+.3,{.2*sin(180*(\x+.3)/8)}) node[xshift=-1.5cm, yshift=-.3cm] {$S_e$}; 
\draw[thick, dashed] plot[smooth,samples=100,domain=0:360] ({4.3+4*cos(\x)},{sin(\x)}) node[xshift=-.4cm, yshift=1cm] {$U_e$}; 

\end{tikzpicture}

\caption{Topology of a node supersolution and an edge supersolution.}
\label{f.nodes and edges}
\end{figure}

It will be convenient for us to impose some (but not too much) boundary regularity on the local supersolutions we work with.  This is the purpose of the next definition of \emph{piecewise smooth local supersolution}.

\begin{definition} \label{d.nodesandedges} A bounded set $E \subset \mathbb{R}^{2}$ is called a \emph{piecewise smooth domain} if there is an $M \in \mathbb{N}$ and, for each $j \in \{1,\dots,M\}$, a compact interval $I_{j} \subset \mathbb{R}$ and a piecewise smooth simple closed curve $\gamma_{j} : I_{j} \to \mathbb{R}$ such that $\partial E = \bigcup_{j = 1}^{M} \gamma_{j}(I_{j})$ and $\{\gamma_{1}(I_{1}),\dots,\gamma_{M}(I_{M})\}$ is pairwise disjoint. 

A local supersolution $(S,U)$ is called a \emph{piecewise smooth local supersolution} if $S$ and $U$ are both bounded and piecewise smooth.
\begin{enumerate}[label = $\triangleright$]
\item Given a local supersolution we call $n_S$ to be the outward normal to $\partial S$ and $\tau_S$ to be the corresponding tangent vector which is the outward normal rotated by $90$ degrees clockwise. 
\item We say that a pair of local supersolutions $(S_1,U_1)$ and $(S_2,U_2)$ is disjoint if $\overline{U}_1 \cap \overline{U}_2 = \emptyset$.
\item If $U$ is doubly connected (i.e., $U$ is connected and $\mathbb{R}^{2} \setminus U$ consists of exactly two connected components), we call $\partial_{out}U$ to be the boundary of the unbounded component of the complement, $\partial_{in}U$ to be the boundary of the bounded component of the complement, and we call $\textup{fill}(U)$ to be the union of $U$ with the bounded component of the complement.
\item A piecewise smooth local supersolution $(S,U)$ is called a \emph{smooth patch} if $\partial S \cap U$ equals the image of a single smooth curve.
\end{enumerate}
\end{definition}

As a basic building block, we will use so-called \emph{supersolution edges} and \emph{supersolution nodes}, defined next.

\begin{definition} \label{D: edge} A piecewise smooth local supersolution $(S,U)$ is called an \emph{edge} if 
	\begin{itemize}
		\item[(a)] $U$ is simply connected, and
		\item[(b)] $S$ splits both $U$ and $\partial U$ into exactly two components or, more precisely, there is a piecewise smooth curve $\gamma : [0,1] \to \partial S$ such that $\gamma((0,1)) = \partial S \cap U$ and $\{\gamma(0),\gamma(1)\} \subset \partial U$. 
	\end{itemize}
	
A piecewise smooth local supersolution $(S,U)$ is called a \emph{node} if $S$ and $U$ are both doubly connected sets; $\partial_{out}U \subset \text{Int}(S)$; and $\partial_{in} U \subset (\R^2 \setminus S)$.
\end{definition}

See \fref{nodes and edges} for a graphic representation of a node and edge local supersolution.

Now our goal is to define an appropriate notion of combining supersolutions.  We begin with a patching operation that amounts to taking local intersections.

We use the following notation, 
\[ \bigcap_{\substack{\alpha \in I \\ x \in K_\alpha}} E_\alpha  := (\bigcap_{\alpha \in I} E_\alpha \cup (\R^2 \setminus K_\alpha)) \cap (\bigcup_{\alpha \in I} K_\alpha)\]
which is also
\[ \bigcap_{\substack{\alpha \in I \\ x \in K_\alpha}} E_\alpha = \{x \in \cup_{\alpha \in I} K_\alpha: \prod_{\substack{\alpha\in I\\ x \in K_\alpha}} {\bf 1}_{E_\alpha}(x) = 1\}\]
to avoid writing the more (respectively) unintuitive and lengthy formulae on the right.  

\begin{figure} \label{f.patch}
\begin{tikzpicture}
\node at (-3.25,0){
\begin{tikzpicture}[scale = .75]
\node[rotate=15] at (-1,0) {
\begin{tikzpicture}
\draw[thick,dashed] (0,0) circle (2);
\filldraw[gray,opacity = .4] (-2,0) -- (2,0) arc ( 0:-180:2);
\draw[thick] (-2,0)--(.8,0);
\end{tikzpicture}
};
\node[rotate=-15] at (1,0) {
\begin{tikzpicture}
\draw[thick,dashed] (0,0) circle (2);
\filldraw[gray,opacity = .4] (-2,0) -- (2,0) arc ( 0:-180:2);
\draw[thick] (-.8,0)--(2,0);
\end{tikzpicture}
};

\end{tikzpicture}
};
\node at (3.25,0){
\begin{tikzpicture}[scale = .75]
\node[rotate=15] at (-1,0) {
\begin{tikzpicture}
\draw[thick,dashed] (0,0) circle (2);
\filldraw[gray,opacity = .4] (-2,0) -- (2,0) arc ( 0:-180:2);
\end{tikzpicture}
};
\node[rotate=-15] at (1,0) {
\begin{tikzpicture}
\draw[thick,dashed] (0,0) circle (2);
\filldraw[gray,opacity = .4] (-2,0) -- (2,0) arc ( 0:-180:2);
\end{tikzpicture}
};
\node[rotate=0] at (0,-1) {
\begin{tikzpicture}
\draw[thick,dashed] (0,0) circle (2);
\filldraw[gray,opacity = .4] plot[smooth,samples=100,domain=-1.6304:1.6304] (\x,1.5-\x*\x) -- (1.6304,-1.1582) arc (-35.393:-180+35.393:2);
\draw[thick] plot[smooth,samples=50,domain=-1.6304:-.9] (\x,1.5-\x*\x) --(0,.92)--plot[smooth,samples=50,domain=.9:1.6304] (\x,1.5-\x*\x);
\end{tikzpicture}
};

\end{tikzpicture}
};
\end{tikzpicture}
\caption{Left: patching of two local supersolutions the outlined boundary is the boundary of the patch.  Right: patching of three local supersolutions.  Condition \eqref{E: property we care about} is satisfied in both cases.}
\end{figure}
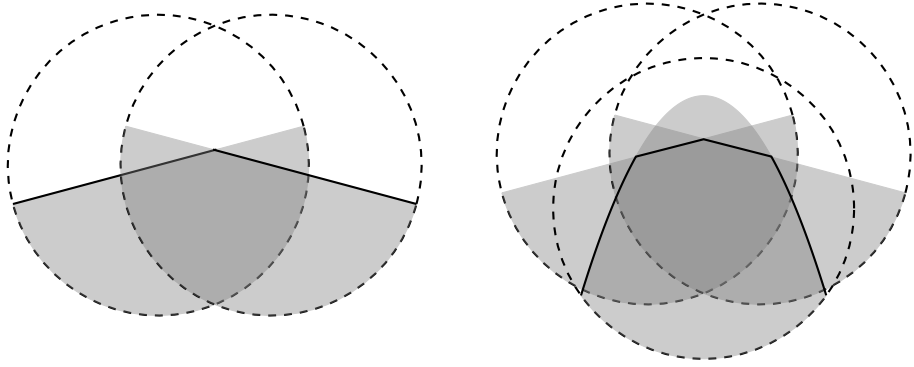
\begin{lemma}[Simple patching]\label{l.edgepatch}
Suppose that $\{(S_e,U_e)\}_{e \in I}$ is a finite collection of local supersolutions such that for each $e \in I$ and each $x_{0} \in \partial U_{e}$, there is a relatively open set $V_{e}$ in $\overline{U_{e}}$ such that $x_{0} \in V_{e}$ and
	\begin{equation} \label{E: property we care about}
		S_{e} \cap V_{e} \supseteq \bigcap_{\substack{e' \in I  \\  x_{0} \in U_{e'}}} S_{e'} \cap V_{e}.
	\end{equation}
Then
\[ \textup{patch}((S_e,U_e)_{e \in I}) := (\bigcap_{\substack{e \in I\\ x \in \overline{U}_e}} S_e, \bigcup_{e \in I} U_e)\]
is a local supersolution in $\cup_{e \in I} U_e$.
\end{lemma}

See \fref{patch} for a visualization of the patch operation.

The idea of the condition in \lref{edgepatch} is simply that a collection of supersolutions on several overlapping domains can be patched together by locally taking the minimum as long as each supersolution is not the minimal supersolution at any point of the boundary of its own domain (which is in the closure of one of the other domains).
\begin{proof}
Call $(S_*,U_*) = \textup{patch}((S_e,U_e)_{e \in I})$.  We just need to check, for any interior point $x_0 \in U_*$ there is a neighborhood $N(x_0)$ in which $S_*$ is an intersection only of $S_e$ for which $N(x_{0}) \subset U_e$ . In that case, $-\chi_{S_{*}}$ is a minimum of supersolutions in $N(x_0)$ and so it is a supersolution in $N(x_0)$.  More precisely, for a sufficiently small neighborhood $N(x_{0})$ of $x_0$, we claim
\[ N(x_{0}) \cap S_* = \bigcap_{\substack{e \in I \, : \, N(x_{0}) \subset U_e}} S_e \cap N(x_{0}).\]
This is immediate unless $x_0 \in \partial U_{e}$ for some $e$.  

In that case, call $J = \{e \in I: x_0 \in \partial U_e\}$.  For each $e \in J$, let $V_{e}$ be the relatively open set in $\overline{U}_{e}$ such that $x_{0} \in V_{e}$ and \eqref{E: property we care about} holds.  Note that we can fix an open ball $N(x_{0})$ containing $x_{0}$ such that the following hold:
	\begin{gather*}
		N(x_{0}) \subset U_{e} \quad \textup{if} \, \, x_{0} \in U_{e}, \quad N(x_{0}) \subset \mathbb{R}^{2} \setminus \overline{U}_{e} \quad \textup{if} \, \, x_{0} \in \mathbb{R}^{2} \setminus \overline{U}_{e}, \\
		N(x_{0}) \cap \overline{U}_{e} \subset V_{e} \quad \textup{if} \, \, e \in J.
	\end{gather*}
A direct argument involving \eqref{E: property we care about} shows that
	\begin{equation*}
		S_{*} \cap N(x_{0}) = \bigcap_{e \in I \, : \, N(x_{0}) \subset U_{e}} S_{e} \cap N(x_{0}).
	\end{equation*}
\end{proof}

Of course this patching procedure, does not require the supersolutions involved to be edges, however the hypothesis will typically not hold when one of the supersolutions involved is a node, see \fref{nodejoin}. It is a bit more topologically delicate to explain how to join a pair of edges to a node.  Toward that end, we begin by defining a criterion for the node and two edges to be admissible.

\begin{definition} \label{d.incident} Given a node $(S,U)$ and an edge $(S',U')$, we say that $(S',U')$ is \emph{admissibly incident} on $(S,U)$ if
	\begin{itemize}
		\item[(i)] $\partial U' \cap \partial [\text{fill}(U)]$ consists of exactly two points, and the corresponding arc of $\partial_{out} U$ separates $U'$ into two connected components,
		\item[(ii)] $(\partial S' \cap U') \cap \partial U$ consists of exactly two points, one in $\partial_{out}U$ and the other, in $\partial_{in}U$.  (In particular, $\partial S'$ contains a piecewise smooth path $\gamma : [0,1] \to \partial S' \cap U'$ such that $\gamma(0) \in \partial_{out}U$, $\gamma(1) \in \partial_{in}U$, and $\gamma(t) \in U$ for each $t \in (0,1)$.)
		\item[(iii)] $(\partial S' \cap U') \cap (\partial S \cap U)$ consists of a single point.
	\end{itemize}
If the path $\gamma$ in (ii) is such that the velocity $\dot{\gamma}$ is parallel to the tangent vector $\tau_{S'}$ to $\partial S'$, then we say that $(S',U')$ is \emph{incoming} at $(S,U)$.  Otherwise, if $\dot{\gamma}$ is anti-parallel to $\tau_{S'}$, we say that $(S',U')$ is \emph{outcoming} at $(S,U)$.\end{definition}

Given a node $(S_v,U_v)$ and a pair of disjoint incident edges $(S_{e_1},U_{e_1})$ and $(S_{e_2},U_{e_2})$, respectively incoming and outgoing, we now explain how to patch and define an edge supersolution $(S_{join},U_{join})$, where $U_{join}$ is defined by 
\[ U_{join} := U_{e_1} \cup \textup{fill}(U_v) \cup U_{e_2}.\]  
See \fref{nodejoin} for a graphical representation of the patching procedure; the figure will be used to describe $S_{join}$.  Let $\gamma$ be the curve obtained by starting at point (I); traversing $\partial S_{e_{1}} \cap U_{e_{1}}$ until point (II); then traversing $\partial S_{v}$ counterclockwise until point (III); traversing $\partial S_{e_{2}} \cap U_{e_{2}}$ until point (IV); then following $\partial U_{e_{2}}$ clockwise to (V); following $\partial U_{v}$ to (VI); and, finally, proceeding to (I) along $\partial U_{e_{1}}$ clockwise.  In other words, $\gamma$ is the curve that bounds the dark gray shaded region in the figure.  We let $S_{join}$ be the closure of the domain bounded by $\gamma$, i.e. the shaded region, which is well-defined by the Jordan Curve Theorem.   

Note that although \fref{nodejoin} depicts a particular configuration of nodes and edges, the construction above makes sense whenever $(S_{e_{1}},U_{e_{1}})$ is incoming at $(S_{v},U_{v})$ and $(S_{e_{2}},U_{e_{2}})$ is outgoing at $(S_{v},U_{v})$.  The reason is the points (I), (II), (III), (IV), (V), and (VI) and the paths between them are well-defined due to the definitions.  For instance, (II) is the unique point in $(\partial S_{e_{1}} \cap U_{e_{1}}) \cap (\partial S_{v} \cap U_{v})$, whose existence and uniqueness is guaranteed by Definition \ref{d.incident}.

In conclusion, we define the node join operation by
	\begin{equation*}
		\textup{node.join}((S_{e_{1}},U_{e_{1}}),(S_{v},U_{v}),(S_{e_{2}},U_{e_{2}})) := (S_{join},U_{join}).
	\end{equation*}

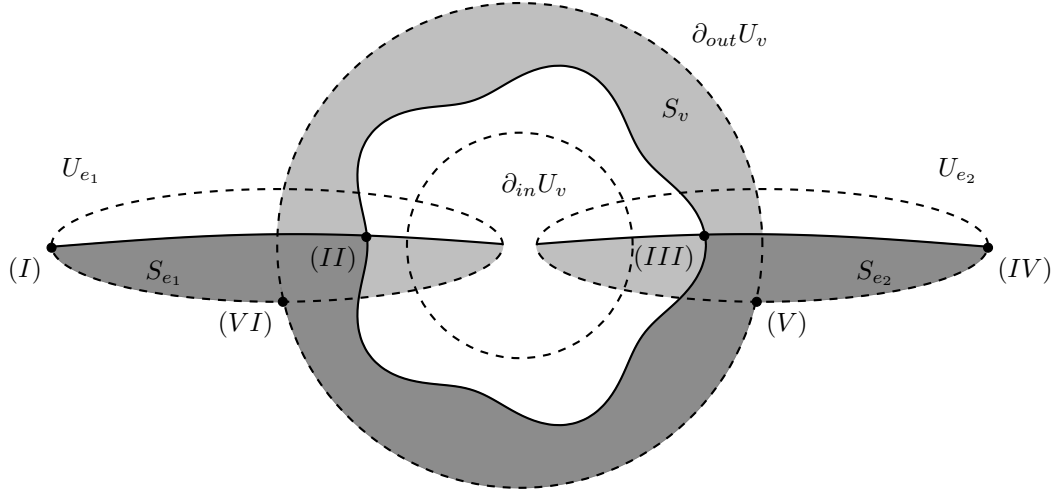
\begin{figure}[t]

\begin{tikzpicture}[scale = .75]


\begin{scope}

\draw[draw=gray!50!white,fill=gray!50!white] plot[smooth,samples=100,domain=0:8] (\x+.3,{.2*sin(180*(\x+.3)/8)}) --
plot[smooth,samples=100,domain=0:-180] ({4.3+4*cos(\x)},{sin(\x)}); 

\draw[draw=gray!50!white,fill=gray!50!white] plot[smooth,samples=100,domain=0:8] (-\x-.3,{.2*sin(180*(\x+.3)/8)}) --
plot[smooth,samples=100,domain=0:-180] ({-4.3-4*cos(\x)},{sin(\x)}); 


\draw[draw=gray!50!white,fill=gray!50!white]  
plot[smooth,samples=100,domain=0:360] ({4.3*cos(\x)},{4.3*sin(\x)}) --  
plot[smooth,samples=100,domain=360:0] ({(3+.3*cos(5*\x))*cos(\x)},{(3+.3*cos(5*\x))*sin(\x)}); 

\draw[draw=gray,fill=gray!90!white]  plot[smooth,samples=100,domain=2.4:8] (-\x-.3,{.2*sin(180*(\x+.3)/8)}) --
plot[smooth,samples=100,domain=0:-92] ({-4.3-4*cos(\x)},{sin(\x)})--
plot[smooth,samples=100,domain=190:347] ({4.3*cos(\x)},{4.3*sin(\x)})--
plot[smooth,samples=100,domain=-92:0] ({4.3+4*cos(\x)},{sin(\x)})--
plot[smooth,samples=100,domain=8:3] (\x+.3,{.2*sin(180*(\x+.3)/8)})--
plot[smooth,samples=100,domain=0:-180] ({(3+.3*cos(5*\x))*cos(\x)},{(3+.3*cos(5*\x))*sin(\x)});

\end{scope}


\draw[thick] plot[smooth,samples=100,domain=0:8] (\x+.3,{.2*sin(180*(\x+.3)/8)}) node[xshift=-1.5cm, yshift=-.35cm] {$S_{e_2}$}; 
\draw[thick, dashed] plot[smooth,samples=100,domain=0:360] ({4.3+4*cos(\x)},{sin(\x)}) node[xshift=-.4cm, yshift=1cm] {$U_{e_2}$}; 

\draw[thick] plot[smooth,samples=100,domain=0:8] (-\x-.3,{.2*sin(180*(\x+.3)/8)}) node[xshift=1.5cm, yshift=-.35cm] {$S_{e_1}$}; 
\draw[thick, dashed] plot[smooth,samples=100,domain=0:360] ({-4.3-4*cos(\x)},{sin(\x)}) node[xshift=.4cm, yshift=1cm] {$U_{e_1}$}; 

\draw[thick] plot[smooth,samples=100,domain=360:0] ({(3+.3*cos(5*\x))*cos(\x)},{(3+.3*cos(5*\x))*sin(\x)}) node[xshift=-.4cm, yshift=1.8cm] {$S_v$}; 
\draw[thick, dashed] (0,0) circle (4.3cm) node[xshift=2.8cm, yshift=2.8cm] {$\partial_{out}U_v$};
\draw[thick, dashed] (0,0) circle (2cm) node[xshift=.2cm, yshift=.8cm] {$\partial_{in}U_v$}; 

\filldraw (-8.3,-.04) circle (.08)  node[anchor = north east] {$(I)$};
\filldraw (-2.72,.15) circle (.08)  node[anchor = north east] {$(II)$};
\filldraw (3.27,.17) circle (.08)  node[anchor = north east] {$(III)$};
\filldraw (8.3,-.04) circle (.08) node[anchor = north west] {$(IV)$};
\filldraw (4.2,-1) circle (.08)  node[anchor = north west] {$(V)$};
\filldraw (-4.2,-1) circle (.08)  node[anchor = north east] {$(VI)$};
\end{tikzpicture}
\caption{Two edges intersect a node, how to patch to get a supersolution.  The darker shaded region is $S_{join}$.}
\label{f.nodejoin}
\end{figure}

\begin{lemma} \label{L: join a node and an edge}\label{l.join a node and an edge}
Let $(S_{v},U_{v})$, $(S_{e_{1}},U_{e_{1}})$, $(S_{e_{2}},U_{e_{2}})$ be piecewise smooth local supersolutions, and suppose that $(S_{e_{1}},U_{e_{1}})$ and $(S_{e_{2}},U_{e_{2}})$ are disjoint (i.e., $\overline{U}_{e_{1}} \cap \overline{U}_{e_{2}} = \emptyset$).  If $(S_v,U_v)$ is a node and $(S_{e_1},U_{e_1})$, and $(S_{e_2},U_{e_2})$ are, respectively,  incoming and outgoing edges admissibly incident on $(S_v,U_v)$, then the pair $(S_{join},U_{join})$ is a local supersolution edge. \end{lemma}

\begin{proof} First, to see that $U_{join}$ is simply connected, observe that it equals the bounded component of the simple closed curve obtained by starting at point (I) in \fref{nodejoin}; proceeding clockwise around $\partial U_{e_{1}}$ until it first intersects $\partial U_{v}$; continuing around $\partial U_{v}$ until it first intersects $\partial U_{e_{2}}$; proceeding to (IV); then continuing on to (V), (VI), and back to (I).  This follows from the definition of $U_{join}$, Definition \ref{d.incident}, and the Jordan Curve Theorem.

Further, by construction, $\partial S_{join} \cap U_{join}$ equals the part of $\partial S_{join}$ that starts at point (I), then proceeds to (II) and (III) before ending at (IV).  Parametrizing this path as $\gamma : [0,1] \to \partial S_{join}$, we note that $\gamma([0,1]) \cap \partial U = \{\gamma(0),\gamma(1)\}$.  Hence (b) in Definition \ref{D: edge} holds.

Finally, by Propositions \ref{P: boundary thing} and \ref{P: minimum}, to prove that $S_{join}$ is a supersolution in $U_{join}$, it suffices to verify that, for any $x_{0} \in \partial S_{join} \cap U_{join}$, there is an $r > 0$ such that $S_{join} \cap B(x_{0},r)$ equals a finite intersection of supersolutions inside $B(x_{0},r)$.  The only place this is delicate is where (a) $x_{0} \in \partial U_{v} \cap U_{e_{i}}$ for some $i \in \{1,2\}$ or (b) $x_{0} \in U_{v} \cap \partial U_{e_{i}}$ for some $i$.  However, in case (a), by the definition of admissibly incident, $S_{join} \cap B(x_{0},r) = S_{v} \cap B(x_{0},r)$ for $r$ small enough.  Similarly in case (b), $S_{join} \cap B(x_{0},r) = S_{e_{i}} \cap B(x_{0},r)$ for small $r$.  Therefore, $S_{join}$ is a supersolution in $U_{join}$, and  $(S_{join},U_{join})$ is a local supersolution.
\end{proof}

\subsection{$(\mathbb{Z}^{2},\mathbb{E}^{2})$-indexed local supersolution networks}  We now show how to use the supersolution patching procedure to produce supersolutions that approximate regular cube sets.  To abstract away some of the details, we start by defining a type of network that will allow us to associate a local supersolution to each edge and vertex of the graph $(\mathbb{Z}^{2},\mathbb{E}^{2})$.  

\begin{definition}  \label{D: supersolution network} We say that a family of pairs $(S_{e},U_{e})_{e \in \mathbb{E}^{2}}$ and $(S_{v},U_{v})_{v \in \mathbb{Z}^{2}}$ forms a \emph{$(\mathbb{Z}^{2},\mathbb{E}^{2})$-compatible local supersolution network} if:
	\begin{itemize}
		\item[(i)] There is an $F \in \mathbb{R}$ such that, for any $x \in \mathbb{Z}^{2}$ and any $e \in \mathbb{E}^{2}$, the pairs $(S_{x},U_{x})$ and $(S_{e},U_{e})$ are piecewise smooth local supersolutions of \eqref{e.sieqn}.
		\item[(ii)] For any $x \in \mathbb{Z}^{d}$ and $e \in \mathbb{E}^{2}$, $(S_{x},U_{x})$ and $(S_{e},U_{e})$ are smooth patches.
		\item[(iii)] For any $x \in \mathbb{Z}^{d}$, $\text{fill}(U_{x})$ is a neighborhood of $x$, and if $y \in \mathbb{Z}^{d} \setminus \{x\}$, then $\overline{\text{fill}(U_{x})} \cap \overline{\text{fill}(U_{y})} = \emptyset$.
		\item[(iv)] If $e_{1},e_{2} \in \mathbb{E}^{2}$ and the line segments $e_{1}$ and $e_{2}$ are disjoint, then $(S_{e_{1}},U_{e_{1}})$ and $(S_{e_{2}},U_{e_{2}})$ are disjoint.
		\item[(v)] For each $[x,y] \in \mathbb{E}^{2}$, the pairs $(S_{[x,y]},U_{[x,y]})$ and $(S_{x},U_{x})$ and $(S_{[x,y]},U_{[x,y]})$ and $(S_{y},U_{y})$ are admissibly incident, the former outgoing at $x$ and the latter, incoming at $y$.
		\item[(vi)] For each $[x,y] \in \mathbb{E}^{2}$, if $z + [-1/2,1/2]^{2}$ is the (unique) $\mathbb{Z}^{2*}$-measurable cube such that $[x,y]$ traverses $z + \partial [-1/2,1/2]^{2}$ clockwise, then $\overline{U}_{[x,y]} \setminus S_{[x,y]} \subset \mathbb{R}^{2} \setminus (z + [-1/2,1/2]^{2})$.
		\item[(vii)] For each $[x,y] \in \mathbb{E}^{2}$, $\textup{fill}(U_{x}) \cup U_{[x,y]} \cup \textup{fill}(U_{y})$ is a neighborhood of the line segment $[x,y]$.  
		\item[(viii)] For any $[x,y] \in \mathbb{E}^{2}$ and $v \in \mathbb{Z}^{2}$, $(S_{[x + v,y + v]},U_{[x+v,y+v]}) = (v + S_{[x,y]},v + U_{[x,y]})$ and $(S_{x +v}, U_{x + v}) = (v + S_{x},v + U_{x})$.
	\end{itemize}
\end{definition}

\begin{remark} \label{R: translation-invariance-useful} The translation invariance assumption, that is, condition (viii) in the above definition, is useful for two reasons.  First, it automatically implies that there is a constant $C > 0$ such that $\text{diam}(U_{\nu}) \leq C$ for each $\nu \in \mathbb{Z}^{2} \cup \mathbb{E}^{2}$.  Further, the functions parametrizing the curves  $\{\partial S_{\nu} \cap U_{\nu}\}_{\nu \in \mathbb{Z}^{2} \cup \mathbb{E}^{2}}$ satisfy uniform $C^{2}$ estimates.  \end{remark}

Combining our abstract supersolution patching procedure with the notion of a local supersolution network, we now describe how to approximate an arbitrary regular $\mathbb{Z}^{2*}$-measurable set by a supersolution.  We assume in the discussion that follows that we have fixed a $(\mathbb{Z}^{2},\mathbb{E}^{2})$-compatible local supersolution network consisting of edges $(S_{e},U_{e})_{e \in \mathbb{Z}^{2}}$ and nodes $(S_{v},U_{v})_{v \in \mathbb{Z}^{2}}$.  

By Theorem \ref{T: boundary curve} the boundary of every regular $\mathbb{Z}^{2*}$-measurable set is a disjoint union of simple paths.  Thus to create a supersolution approximating such a boundary, we start by describing the method to approximate a single simple path in $(\mathbb{Z}^{2},\mathbb{E}^{2})$.  To begin with, given $m, n \in \mathbb{Z}$, suppose that $\gamma: \{m,\dots,n\} \to \mathbb{Z}^{2}$ is a finite simple path in $\mathbb{E}^{2}$, simple paths with infinite length will be considered later.  We can create an edge supersolution along $\gamma$ by the following procedure:  call $\Sigma_{m} = (S_{[\gamma(m),\gamma(m + 1)]},U_{[\gamma(m),\gamma(m + 1)]})$ and then, inductively, define
\[\Sigma_{i} = \textup{node.join}(\Sigma_{i-1},(S_{\gamma(i)},U_{\gamma(i)}),(S_{[\gamma(i),\gamma(i+1)]},U_{[\gamma(i),\gamma(i+1)]}))\]
for $m + 1 \leq i \leq n - 1$.  Note that this results in a local supersolution edge
\[(S_\gamma,U_\gamma) := \Sigma_{n-1} \]
which is incident on $\gamma(m)$ and $\gamma(n)$, respectively outgoing and incoming.

When $\gamma$ is a simple closed path $(S_{\gamma},U_{\gamma})$, as defined above, joins all the nodes/edges along $\gamma$ except misses the node at $\gamma(m)$.  Of course we can simply patch this node in using basically the same procedure as before, although it is slightly awkward to phrase in our terminology.  Simply take $(S_*,U_*)$ to be the $\textup{node.join}$ of the ordered triple 
\[((S_{[\gamma(n-1),\gamma(n)]},U_{[\gamma(n-1),\gamma(n)]}),(S_{\gamma(m)},U_{\gamma(m)}),(S_{[\gamma(m),\gamma(m+1)]},U_{[\gamma(m),\gamma(m+1)]}))\]
and then this can be joined with $(S_{\gamma},U_{\gamma})$ by the patch operation
\[ (S_{\bar{\gamma}},U_{\bar{\gamma}})=\textup{patch}((S_{\gamma},U_{\gamma}),(S_{*},U_{*})).\]
In the future we will simply omit the bars and write $(S_{\gamma},U_{\gamma})$ abusing notation in the case when $\gamma$ is a simple closed path.

\begin{lemma} \label{L: boundary approximation}
If $(S_{e},U_{e})_{e \in E}$ and $(S_{v},U_{v})_{v \in V}$ define a $(\mathbb{Z}^{2},\mathbb{E}^{2})$-compatible local supersolution network, and if $\gamma : \{m,\dots,n\} \to \mathbb{Z}^{2}$ is any simple path in $(\mathbb{Z}^{2},\mathbb{E}^{2})$, possibly closed, then $(S_\gamma,U_\gamma)$ as defined in the paragraphs above is a local supersolution.  

Furthermore, if $\gamma : \mathbb{Z} \to \mathbb{Z}^{2}$ is a simple path in $(\mathbb{Z}^{2},\mathbb{E}^{2})$ and if $\gamma_{[m,n]}$ denotes the restriction of $\gamma$ to $\{m,\dots,n\}$, then the pair $(S_{\gamma},U_{\gamma})$ defined by 
	\begin{equation*}
		S_{\gamma} = \bigcup_{N = 1}^{\infty} S_{\gamma_{[-N,N]}}, \quad U_{\gamma} = \bigcup_{N = 1}^{\infty} U_{\gamma_{[-N,N]}}
	\end{equation*}
also defines a local supersolution.
\end{lemma}

	\begin{proof}  The first statement is a direct consequence of Lemmas \ref{l.edgepatch} and \ref{L: join a node and an edge}.  For the second statement, first, observe that $\gamma$ is locally finite: that is, for each compact set $K \subset \mathbb{R}^{2}$, $\#\{i \in \mathbb{Z} \, \mid \, \gamma(i) \in K\} < \infty$.  Combining this with the diameter bound in Remark \ref{R: translation-invariance-useful}, we find that $(S_{\gamma} \cap K,U_{\gamma} \cap K) = (S_{\gamma_{[-N,N]}} \cap K, U_{\gamma_{[-N,N]}} \cap K)$ for all $N$ large enough.  Thus, as a local uniform limit of supersolutions, $S_{\gamma}$ is a supersolution in $U_{\gamma}$.   \end{proof}  
	
Now we describe how to approximate regular $\mathbb{Z}^{2*}$-measurable sets by supersolutions.  If $A$ is a regular $\mathbb{Z}^{2*}$-measurable set, it can be written as the sum of countably many connected components $A = \cup_{n \in J} A_{n}$ for some $J \subset \mathbb{N}$, where $(A_{n})_{n \in \mathbb{N}}$ are regular $\mathbb{Z}^{2*}$-measurable sets.  Let us thus start in the case that $A$ is simply connected. 

First, let $\gamma : E \to \mathbb{Z}^{2}$ be a simple path such that $\partial A = \{\gamma\}$ and $\gamma$ traverses $\partial A$ clockwise; such a path exists by Theorem \ref{T: boundary curve} and simple connectedness.  Let $(S_{\gamma},U_{\gamma})$ be the local supersolution constructed as in Lemma \ref{L: boundary approximation}.  

It will be useful to know, and follows essentially from condition (vi) in Definition \ref{D: supersolution network}, that $\partial S_{\gamma} \setminus A \subset U_{\gamma}$.  

	\begin{lemma} \label{L: supersolution thing} If $A$ is a simply connected, regular $\mathbb{Z}^{2}$-measurable set and the curve $\gamma$ and local supersolution $(S_{\gamma},U_{\gamma})$ are as constructed above, then $\partial S_{\gamma} \setminus A \subset U_{\gamma}$.  Furthermore, $\partial S_{\gamma} \cap U_{\gamma}$ is a piecewise smooth simple curve. \end{lemma}
	
The proof is deferred to Appendix \ref{A: boundary issue}.  

Next, for each cube $z + [-1/2,1/2]^{2} \subset A^{b}$, let $(S_{z},U_{z})$ denote the local supersolution
	\begin{equation*}
		S_{z} = z + [-1/2,1/2]^{2}, \quad U_{z} = z + (-1/2,1/2)^{2}.
	\end{equation*}
Letting $ \{z_{n}\}_{n \in \mathbb{N}} \subset \mathbb{Z}^{2*}$ be an enumeration of these cubes, define $(S_{A^{b}_{n}},U_{A^{b}_{n}})_{n \in \mathbb{N}}$ recursively by 
	\begin{align*}
		(S_{A^{b}_{1}},U_{A^{b}_{1}}) &= \textup{patch}((S_{\gamma},U_{\gamma}),(S_{z_{1}},U_{z_{1}})), \\
		(S_{A^{b}_{n+1}},U_{A^{b}_{n+1}}) &= \textup{patch}((S_{A^{b}_{n}},U_{A^{b}_{n}}), (S_{z_{n+1}},U_{z_{n+1}})).
	\end{align*}
Finally, let $(S_{A^{b}},U_{A^{b}})$ be the limiting supersolution with $S_{A^{b}} = \cup_{n = 1}^{\infty} S_{A^{b}_{n}}$ and $U_{A^{b}} = \cup_{n = 1}^{\infty} U_{A^{b}_{n}}$.  As in the proof of Lemma \ref{L: boundary approximation}, this defines a local supersolution.

It only remains to ``fill in" the rest of $A$.  Let $\{q_{n}\}_{n \in \mathbb{N}} \subset \mathbb{Z}^{2*}$ be an enumeration of the cubes $q + [-1/2,1/2]^{2}$ contained in $A^{\textup{int}}$.  Define $(S_{A},U_{A})$ by 
	\begin{equation*}
		S_{A} = S_{A^{b}} \cup \bigcup_{n = 1}^{\infty} (q_{n} + [-1/2,1/2]^{2}), \quad U_{A} = U_{A^{b}} \cup \bigcup_{n = 1}^{\infty} (q_{n} + [-1/2,1/2]^{2}),
	\end{equation*}
Note that, with this definition, $A^{\text{int}} \subset \text{Int}(S_{A})$.  Since each cube $q_{n} + [-1/2,1/2]^{2}$ is surrounded by cubes in $A$, one readily checks that $(S_{A},U_{A})$ is a local supersolution.  We claim that, in fact, $S_{A}$ is a supersolution in $\mathbb{R}^{2}$.

	\begin{lemma} \label{L: supersolution whole space} $S_{A}$ is a supersolution in $\mathbb{R}^{2}$ with piecewise smooth boundary. \end{lemma}
	
		\begin{proof} By Proposition \ref{P: boundary thing}, it suffices to check that, for every $x_{0} \in \partial S_{A}$, there is an open set $U$ containing $x_{0}$ such that $S_{A}$ is a supersolution in $U$.  If $x_{0} \in \partial S_{A} \cap A$, then this is immediate since $A \subset U_{A}$; $U_{A}$ is open; and $S_{A}$ is a supersolution in $U_{A}$.  Otherwise, if $x_{0} \in \partial S_{A} \setminus A$, then $x_{0} \in \partial S_{\gamma}$ by construction.  Yet Lemma \ref{L: supersolution thing} implies that $\partial S_{\gamma} \setminus A \subset U_{\gamma}$.  Since $U_{\gamma} \subset U_{A}$, we are done. \end{proof}
		
We showed how to construct a supersolution $S_{A}$ in case $A$ is a simply connected, regular $\mathbb{Z}^{2*}$-measurable set.  If, on the other hand, $A$ is only connected and not simply connected, we proceed by letting $S = S_{A_{j}}$, where $A_{j}$ is chosen so that $\mathbb{R}^{2} \setminus A_{j}$ is the $j$th connected component of $\mathbb{R}^{2} \setminus A$.  Since any compact set in $\mathbb{R}^{2}$ sees at most finitely many boundary paths of $A$, $\cap_{j} S_{A_{j}} \cap B(0,R)$ equals a finite intersection of supersolutions for any $R > 0$.  Hence $S_{A} = \cap_{j} S_{A_{j}}$ is a supersolution itself.

When $A$ is not even connected, we let $\{S_{A_{n}}\}$ be the supersolutions associated to its connected components.  By Definition \ref{D: supersolution network} and the regularity of $A$, these supersolutions are pairwise disjoint.  Hence the union $S_{A} := \cup_{n} S_{A_{n}}$ is also a supersolution.

Summing up, we have
	
	\begin{lemma} \label{L: supersolutions finally} If $(S_{e},U_{e})_{e \in \mathbb{E}^{2}}$ and $(S_{v},U_{v})_{v \in \mathbb{Z}^{2}}$ form a $(\mathbb{Z}^{2},\mathbb{E}^{2})$-compatible local supersolution network, then there is a constant $C > 0$ such that, for each regular $\mathbb{Z}^{2*}$-measurable set $A$, there is a closed set $S_{A}$, which is a supersolution in $\mathbb{R}^{2}$, such that $A^{int} \subset \textup{Int}(S_{A})$ and
		\begin{equation*}
			d_{H}(\partial S_{A},\partial A) \leq C, \quad d_{H}(S_{A},A) \leq C.
		\end{equation*}
	\end{lemma}
		
		\begin{proof}  We only need to verify the distance bounds.  Since any point in $S_{A} \setminus A$ is contained in the set $S_{\gamma}$ as defined above, the diameter bound in Remark \ref{R: translation-invariance-useful} implies that $S_{A} \subset A + B_{C}$.  At the same time, $A^{int} \subset \text{Int}(S_{A})$ so, by condition (i) in the definition of regularity, we have $A \subset S_{A} + B_{\sqrt{2}}$.  This gives $d_{H}(S_{A},A) \leq C$.  Since the union of the images of the paths $\gamma$ defined above is precisely $\partial A$, the same reasoning shows $d_{H}(\partial S_{A}, \partial A) \leq C$.    \end{proof}

\subsection{Approximating smooth sets}  We still need to show that we can approximate sufficiently smooth sets by regular $\mathbb{Z}^{2*}$-measurable sets.  Toward that end, the main technical result we need follows:

	\begin{lemma} \label{L: annoying regularity argument} There is an $R > 0$ such that if $K \subset \mathbb{R}^{2}$ satisfies an interior and exterior ball condition of radius $R$, then the $\mathbb{Z}^{2*}$-measurable approximation $A_{K}$ of $K$ defined by
		\begin{align*}
			A_{K} = \bigcup \left\{ z + [-1/2,1/2]^{2} \, \mid \, z \in \mathbb{Z}^{2*}, \, \, (z + [-1/2,1/2]^{2}) \cap K \neq \emptyset \right\}
		\end{align*}
	is regular.  Furthermore, 
		\begin{equation*}
			K \subset A_{K}, \quad d_{H}(K,A_{K}) \leq R, \quad d_{H}(\partial K, \partial A_{K}) \leq R.
		\end{equation*}
	\end{lemma}  
	
		\begin{proof}  To show that (i) and (ii) in Definition \ref{D: regular cube sets} hold for $R > 0$ large enough, we argue by contradiction.  Where (ii) is concerned, if the lemma is not true, then, after translating and rotating, we can find sets $(K_{n})_{n \in \mathbb{N}}$ such that, for each $n \in \mathbb{N}$, $K_{n}$ satisfies an interior and exterior ball condition of radius $n$ and			\begin{gather} \label{E: contradiction inclusions}
				K_{n} \cap ((-1/2,1/2) + [-1/2,1/2]^{2}) \neq \emptyset, \quad K_{n} \cap ((1/2,-1/2) + [-1/2,1/2]^{2}) \neq \emptyset, \\
				K_{n} \cap ((1/2,1/2) + [-1/2,1/2]^{2}) = K_{n} \cap ((-1/2,-1/2) + [-1/2,1/2]^{2}) = \emptyset. \nonumber
			\end{gather}
		By compactness of $[-1/2,1/2]^{2}$, we conclude that there is a half-space $K_{\infty} \subset \mathbb{R}^{2}$ still satisfying \eqref{E: contradiction inclusions}.  This is readily shown to be impossible.  
		
		A similar approach establishes (i).
		
		The remaining claims follow directly from the definitions.  \end{proof}  
		
\subsection{Proofs of the Main Lemmas}  All that remains to prove Lemmas \ref{l.stationary solutions} and \ref{L: sub and supersolutions} is to show that the preceding discussion is not vacuous.  In other words, we need to show there is a medium $a \in C^{\infty}(\mathbb{T}^{2}; [1,\Lambda])$ for which a local supersolution network can be constructed.  This is true, and it will be proved in Section \ref{S: edges and nodes}.  Let us state it as its own proposition for now:

\begin{proposition} \label{P: existence of a network} There is an $a \in C^{\infty}(\mathbb{T}^{2}; [1,\Lambda])$ for which a $(\mathbb{Z}^{2},\mathbb{E}^{2})$-compatible local supersolution network (Definition \ref{D: supersolution network}) exists for some force $F > 0$.  In fact, given $\zeta > 0$, this can be done so that $\|a - 1\|_{L^{\infty}(\mathbb{T}^{d})} \leq \zeta$. \end{proposition}

Combining this with Lemmas \ref{L: supersolutions finally} and \ref{L: annoying regularity argument}, we obtain Lemmas \ref{l.stationary solutions} and \ref{L: sub and supersolutions}:

\begin{proof}[Proof of Lemma \ref{L: sub and supersolutions}]  Let $R$ be the constant of Lemma \ref{L: annoying regularity argument} and $a$ be a medium as in Proposition \ref{P: existence of a network}.  Let $F_{a} = F > 0$ be the associated force.   Given a set $K$ satisfying an interior and exterior ball condition of radius $R$, let $A_{K}$ be the approximation of Lemma \ref{L: annoying regularity argument}.  By applying Lemma \ref{L: supersolutions finally} to $a$ and $F_{a}$, we obtain a closed set $S_{K} = S_{A_{K}}$ which is a supersolution of \eqref{e.sieqn} with $F = F_{a}$.  We note that $S^{*}(K) = S_{K}$ has all the desired properties by concatenating the bounds and inclusions of the two lemmas.

To obtain a subsolution with the desired properties, we repeat the previous procedure with $K$ replaced by $\mathbb{R}^{2} \setminus K$, which still satisfies interior and exterior ball conditions of radius $R$ since $K$ does.  This leads to a closed set $S_{\mathbb{R}^{2} \setminus K}$ containing $\mathbb{R}^{2} \setminus K$ which is a supersolution of \eqref{e.sieqn} with $F = F_{a}$.  We conclude by defining the open set $S_{*}(K) = \mathbb{R}^{2} \setminus S_{\mathbb{R}^{2} \setminus K}$ and noting that $S_{*}(K)$ is a subsolution of \eqref{e.sieqn} with $F = -F_{a} < 0$.  \end{proof}  

\begin{proof}[Proof of Lemma \ref{l.stationary solutions}]  Let $a$ be a medium as in Lemma \ref{L: sub and supersolutions} and let $R$ and $F_{a}$ be the constants of that same lemma.  Given a set $K \subset \mathbb{R}^{2}$ satisfying exterior and interior ball conditions of radius $R$, and given any $F \in (-F_{a},F_{a})$, let $S^{*}(K)$ and $S_{*}(K)$ be the respective super- and subsolution guaranteed by Lemma \ref{L: sub and supersolutions}.  The conclusions of Lemma \ref{L: sub and supersolutions} readily imply that the hypotheses of Proposition \ref{p.perron method}, which is Perron's Method in this context, apply in this situation.  Thus, we obtain a solution $S \subset \mathbb{R}^{2}$ of \eqref{e.sieqn} such that $S_{*}(K) \subset S \subset S^{*}(K)$.  Further, a quick computation shows that the claimed inclusions and distance bound also hold.  \end{proof}

Finally, Corollary \ref{c.pinning corollary} follows by scaling:

\begin{proof}[Proof of Corollary \ref{c.pinning corollary}]  If $K$ is compact with $C^{2}$ boundary, then there is a $\delta > 0$ such that $K$ satisfies exterior and interior ball conditions of radius $\delta$.  Hence there is an $\ep_{0}(K) > 0$ such that $\ep^{-1} K$ satisfies exterior and interior ball conditions of radius $R$ for any $\ep \in (0, \ep_{0}(K))$.  Applying Lemma \ref{L: sub and supersolutions} and blowing down space by a factor $\ep$, we obtain, for each $\ep > 0$, a stationary subsolution $S_{*}^{\ep}$ and a stationary supersolution $S^{*,\ep}$ of \eqref{e.sieqn} such that
	\begin{equation*}
		S_{*}^{\ep} \subset K \subset S^{*,\ep}, \quad d_{H}(S_{*}^{\ep},S^{*,\ep}) \leq C\ep, \quad d_{H}(\partial S_{*}^{\ep},\partial S^{*,\ep}) \leq C\ep.
	\end{equation*}
Therefore, if $(S^{\ep}_{t})_{t \geq 0}$ is the solution flow of \eqref{e.sieqn} with initial datum $K$, then, for each $t > 0$, the comparison principle implies 
	\begin{equation*}
		S_{*}^{\ep}(K) \subset S^{\ep}_{t} \subset S^{*,\ep}.
	\end{equation*}
We then use the distance bounds on $S_{*}^{\ep}$ and $S^{*,\ep}$ to deduce those for $S^{\ep}_{t}$ and $K$.
\end{proof}

\subsection{Homogenization of the Level Set PDE}  In view of \cref{pinning corollary}, it is straightforward to conclude that solutions of the level set PDE are also pinned. 

\begin{proof}[Proof of \tref{level set pinning}]  As in Definition \ref{D: half-relaxed}, define half-relaxed limits $\bar{u}^{*} = \limsup^{*} u^{\ep}$ and $\bar{u}_{*} = \liminf_{*} u^{\ep}$.  We claim that $\bar{u}^{*} = \bar{u}_{*} = u_{0}$.  To avoid repetition, we will only prove that $\bar{u}_{*} \geq u_{0}$; similar arguments show that $u_{0} \geq \bar{u}_{*}$.  

Fix $x_0 \in \mathbb{R}^{d}$.  We will show that, for each $\delta > 0$ and $t \geq 0$, $\bar{u}_{*}(x_0,t) \geq u_{0}(x_0) - \delta$.  To see this, choose an $r > 0$ such that $B_{r}(x_0) \subset\subset \{u_{0} > u_{0}(x_0) - \delta\}$.  By comparison, for each $\ep > 0$, if $(S^{\ep}_{t})_{t \geq 0}$ are the solutions of \eref{siepeqn} with $S_{0} = B_{r}(x_0)$, then 
	\begin{equation*}
		S^{\ep}_{t} \subset \{u^{\ep}(\cdot,t) > u_{0}(x_0) - \delta\} \quad \textup{for each} \, \,  t > 0.
	\end{equation*}  
At the same time, \cref{pinning corollary} implies that, for all sufficiently small $\ep > 0$, we have
	\begin{equation*}
		B_{r/2}(x_0) \subset \{u^{\ep}(\cdot,t) > u_{0}(x_0) - \delta\}. 
	\end{equation*}
This implies $\bar{u}_{*}(x_{0},t) \geq u_{0}(x_{0}) - \delta$ for all $t \geq 0$.

 \end{proof}

\subsection{Supersolutions edge and node construction} \label{S: edges and nodes}  Below we construct a scalar field $a : \mathbb{T}^{2} \to [1,2]$ of the form $a(x) = 1 + \varphi(x)$ with the goal of constructing a $(\mathbb{Z}^{2},\mathbb{E}^{2})$-indexed local supersolution network.  In other words, we will prove Proposition \ref{P: existence of a network}.  

Let $\zeta > 0$.  In what follows, we will construct a $\varphi \in C^{\infty}(\mathbb{T}^{d})$ so that $a = 1 + \varphi$ admits a local supersolution network and 
	\begin{equation*}
		0 \leq \varphi(x) \leq \zeta \quad \textup{in} \, \, \mathbb{T}^{2}.
	\end{equation*}

We break the construction down into steps.  The idea is $\varphi$ will be a bump function (in $\mathbb{T}^{2})$ centered at zero.

\textbf{Step 1: Edge supersolutions.}   We begin with the edge supersolutions, corresponding to each directed edge of $\mathbb{Z}^{2}$ (near which the node supersolutions will be constructed) via certain (oriented) circular arcs.  The circular arcs will have small positive curvature which will create the positivity needed for the supersolution property.  We just need to ensure that the interiors of any two distinct edges are disjoint.

By translation invariance, we can fix our attention on $(0,0)$ and the $8$ directed edges incident on $(0,0)$ connecting to its $\Z^2$ neighbors $(\pm 1,0)$ and $(0,\pm 1)$.  It is convenient to start by defining the arcs connecting $(0,0)$ to $(1,0)$ and $(1,0)$ to $(0,0)$ since the other arcs are constructed the same way.  We will construct two arcs $\gamma^\pm$, $\gamma^+$ for the directed edge $[(0,0),(1,0)]$ and $\gamma^-$ for the directed edge $[(1,0),(0,0)]$, see Figures \ref{f.gamma_plus} and \ref{f.gamma_minus} for the role each plays.  

Notice that the point $(1/2,-t)$ is equidistant to $(0,0)$ and $(1,0)$, no matter the choice of $t$, the distance being $R(t) = \sqrt{t^{2} + 1/4}$.  For $t > 0$ sufficiently large (to be fixed later), let $\gamma^{+}$ be the circular arc connecting $(0,0)$ to $(1,0)$ with radius of curvature $R(t)$ and center $(1/2,-t)$ and $\gamma^{-}$ be the ``reflected" arc obtained by doing the same construction, but with center $(1/2,t)$.  Each arc $\gamma^\pm$ is oriented by the outward normal vector to the corresponding ball.  In particular each arc has small positive mean curvature
\[ \kappa_{\gamma^\pm} = \frac{1}{R(t)}.\]
Choosing $R$ large enough we can guarantee that the circular arcs are close to the line segment $[(0,0),(1,0)]$
\begin{equation}\label{e.closetoedge}
 d_H(\gamma^\pm,[(0,0),(1,0)]) \leq \frac{1}{4}.
 \end{equation}

The interiors of these arcs are clearly disjoint, being separated by the chord connecting $(0,0)$ and $(1,0)$.  Notice, further, that the angle formed between the tangent line to $\gamma^{+}$ (or $\gamma^{-}$) at $(0,0)$ and the aforementioned chord goes to zero as $t \to +\infty$.

Repeat the same construction between $(0,0)$ and each of its other $\Z^2$ nearest neighbors.  This results in $8$ distinct arcs.  By the previous observation on tangent  lines at $(0,0)$ all $8$ arcs are disjoint, except for possible intersections at their endpoints, provided the radius of curvature $R$ is chosen large enough.  This condition along with \eref{closetoedge} fixes our choice of $R$.

Analogously, for any directed lattice edge $e \in \E^2$, we define an associated arc $\gamma_{e}$ by translating the corresponding arc incident at the origin.  (Note that this hides the $\pm$ notation we have used here in Step 1 inside of the directed edge $e$.)

\begin{figure}
\includegraphics[scale=0.4]{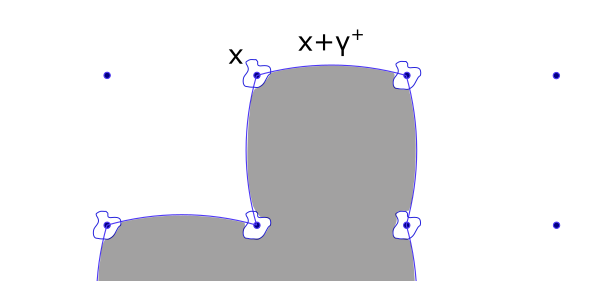}
\caption{A translate of $\gamma^{+}$ in the boundary of the supersolution of \fref{basicidea}.}
\label{f.gamma_plus}
\end{figure}

\begin{figure}
\includegraphics[scale=0.4]{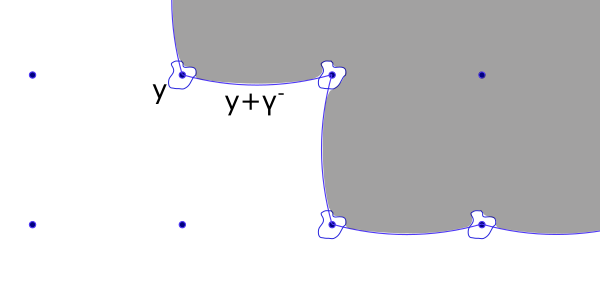} 
\caption{A translate of $\gamma^{-}$ in the boundary of the supersolution of \fref{basicidea}.}
\label{f.gamma_minus}
\end{figure}

\textbf{Step 2: Node shape.}  We add nodes to our network at each $\Z^2$ vertex.  The same construction will be repeated at each one so we restrict attention to $(0,0)$.  The key point is to create large radial gradients to allow for a node supersolution but also to enforce that the incoming/outgoing edge supersolutions are tangential to $Da$ so that the large gradients will not destroy their supersolution property.

To start with, let $B_{r}$ denote the disk centered at $(0,0)$ with radius $r \ll 1$.  Each of the eight arcs $\gamma$ incident on $(0,0)$ passes through $\partial B_{r}$ at some point.  As just discussed, the injectivity of the map sending arcs to tangent vectors shows that each arc is associated to a unique intersection point on $\partial B_r$ provided $r$ is small enough.  By making a small perturbation of $\R^2 \setminus B_{r}$, we can construct a smooth region $\mathcal{O}$ with the property that each arc emanating from $0$ intersects $\partial \mathcal{O}$ at a unique intersection point, and the normal vector of $\partial \mathcal{O}$ is parallel to the tangent line of the arc at the intersection point; see \fref{perturbation}.  We can also choose $\mathcal{O}$ to be symmetric with respect to $\pi/2$ rotations and $\R^2\setminus \mathcal{O} \subset B_{1/4}(0)$, so that $\partial \mathcal{O}$ only intersects the arcs incident on $(0,0)$ and
\[d_H(\R^2 \setminus \mathcal{O},\{0\}) \leq \frac{1}{4}. \]

\begin{figure}[t]
\begin{tikzpicture}[scale = .5]
\draw[ thick, dashed] (0,0) circle (4.3cm) node[xshift=1.8cm, yshift=1.8cm] {$\partial \mathcal{O}$}; 
\draw[ thick] (0,0) arc (109.47:95:30cm) node[anchor = north east] {$\gamma_{e_1}^+$};
\draw[ thick] (0,0) arc (289.47:275:30cm);
\draw[ thick] (0,0) arc (199.47:185:30cm);
\draw[ thick] (0,0) arc (19.47:5:30cm) ;
\draw[ thick] (0,0) arc (70.54:85:30cm);
\draw[ thick] (0,0) arc (160.54:175:30cm) ;
\draw[ thick] (0,0) arc (250.54:265:30cm) node[anchor = south east] {$\gamma_{e_1}^-$};
\draw[ thick] (0,0) arc (340.54:355:30cm);
\draw[ thick] (0,0) circle (2.5pt) node[xshift=1.5cm, yshift=0cm]{$(0,0)$}; 
\draw[ thick] (4.5,1.2) -- (4.4,1.55) -- (4.05,1.45);
\draw[ thick] (4.51,-1.19) -- (4.6,-.82) -- (4.24,-0.76);
\end{tikzpicture}
\caption{Perturbing circle so that it intersects arcs orthogonally.}
\label{f.perturbation}
\end{figure}
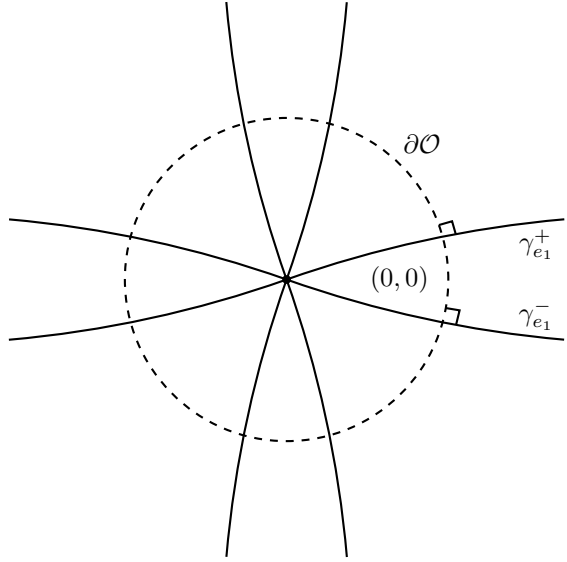

 \textbf{Step 3: Construction of $\varphi$, part 1.}  Since $\mathcal{O}$ is smooth, there is a $\nu > 0$ such that the signed distance function $d_{\mathcal{O}}$ (positive in $\mathcal{O}$ and negative outside) is smooth in $\{|d_{\mathcal{O}}| < \nu\}$.  Let $\eta : [-\nu,\nu] \to [0,1]$ be a smooth function such that 	
 	\begin{gather*}
		\eta(s) = \|\eta\|_{L^{\infty}([-\nu,\nu])} \quad \text{if} \quad s \leq -\nu/2, \quad \eta(s) = 0 \quad \text{if}  \quad s \geq \nu/2, \\
		\eta'(0) = \|\eta'\|_{L^{\infty}([-\nu,\nu])}.
	\end{gather*}  
Define $\varphi : (-1/2,1/2] \times (-1/2,1/2] \to [0,1]$ by     
	\begin{equation*}
		\varphi(x) = \left\{ \begin{array}{r l}
					\eta(d_{\mathcal{O}}(x)), & \textup{if} \, \, |d_{\mathcal{O}}(x)| \leq \nu \\
					0, & \textup{if} \, \, d_{\mathcal{O}}(x) \geq \nu \\
					\|\eta\|_{L^{\infty}([-\nu,\nu])}, & \textup{if} \, \, d_{\mathcal{O}}(x) \leq -\nu
				\end{array} \right.
	\end{equation*}
Extend $\varphi$ $\mathbb{Z}^{2}$-periodically to $\mathbb{R}^{2}$.  

Define the parameter 
\[A := \eta'(0) = \|\eta'\|_{L^{\infty}([-\nu,\nu])}, \]
which we will need to choose large below using our freedom to choose $\eta$.

Let $n_{\partial \mathcal{O}}$ be the outward pointing normal to $\mathcal{O}$ and $\kappa_{\partial \mathcal{O}}$, the mean curvature (following the sign convention $\kappa_{\partial \mathcal{O}} = - \textup{tr}(D^{2} d_{\mathcal{O}})$).  Modify $\eta$ if necessary so that $A$ satisfies
	\begin{equation} \label{E: A condition}
		A > 2 \|(\kappa_{\partial \mathcal{O}})_{-}\|_{L^\infty(\partial \mathcal{O})}.
	\end{equation}
We then find, for each $x\in \partial \mathcal{O}$,
	\begin{equation*}
		-(1 + \varphi(x)) \kappa_{\partial \mathcal{O}}(x) - D\varphi(x)\cdot n_{\partial \mathcal{O}}(x)  \leq 2 \|(\kappa_{\partial \mathcal{O}})_{-}\|_{L^\infty(\partial \mathcal{O})} - A   < 0.
	\end{equation*}
In other words $\overline{\mathcal{O}}$ is a supersolution of \eref{sieqn} for some $ F_1>0$ or, in level set form,  $u = \ind_{\mathcal{O}}$ is a supersolution of the equation
	\begin{equation*}
		-(1 + \varphi(x)) \textup{tr} \left( \left( \textup{Id} - \widehat{Du} \otimes \widehat{Du} \right) D^{2} u \right) - D\varphi(x)\cdot Du  - F_{1}|Du| \geq 0 \quad \textup{in} \, \, \mathbb{R}^{2}.
	\end{equation*}

\textbf{Step 4: Construction of $\varphi$, part 2.} We proceed to ensure that the edges of the network satisfy the necessary differential inequalities outside of $\bigcup_{k \in \mathbb{Z}^{d}} (k + \mathbb{R}^{2} \setminus \mathcal{O})$ (actually outside of a neighborhood of the closure).  Given an edge $\gamma$ in the network, orient it so that its normal vector points away from the center of the corresponding circle.  If $\mathcal{O}_{1}$ and $\mathcal{O}_{2}$ are the two regions intersecting $\gamma$ at either end, first, assume that $x \in \gamma \cap \{d_{\mathcal{O}_{1}} \geq \nu\} \cap \{d_{\mathcal{O}_{2}} \geq  \nu\}$.  It follows that $\varphi$ vanishes in a neighborhood of $x$ and, thus,
	\begin{equation*}
		- (1 + \varphi(x)) \kappa_{\gamma}(x) -  D\varphi(x)\cdot n_{\gamma}(x)  = - \kappa_{\gamma}(x) = -\frac{1}{R} < 0,
	\end{equation*}
where $R$ is the radius of curvature fixed earlier.  

It remains to check the requisite inequalities near a vertex, which we can take to be $(0,0)$ by symmetry.  Assume that $x \in \gamma \cap \{|d_{\mathcal{O}}| \leq \nu\}$.  We are only interested in the part of $\gamma$ in a small neighborhood of $\overline{\mathcal{O}}$, so as long as we can prove the requisite supersolution property for $d_{\mathcal{O}}(x)$ in a neighborhood of  $[0,\nu]$ we will be done (in particular for small negative values of $d_{\mathcal{O}}$ since values $\geq \nu$ have already been handled).  

We start at the intersection point $\bar{x} \in \gamma \cap \partial \mathcal{O}$ and work outwards.  By construction, 
\[  n_{\gamma}(\bar{x})\cdot n_{\mathcal{O}}(\bar{x})  = 0.\]
  Thus, by continuity, there is a $\nu' \in (0,\nu/2)$ such that $x \in \gamma \cap \{- \nu' \leq d_{\mathcal{O}} \leq \nu'\}$ implies 
	\begin{equation*}
		| n_{\gamma}(x) \cdot Dd_{\partial \mathcal{O}}(x)| \leq (2R A)^{-1}.
	\end{equation*}
 Hence, for such $x$, we find
	\begin{equation*}
		- (1 + \varphi(x)) \kappa_{\gamma}(x) -  D\varphi(x) \cdot n_{\gamma}(x) \leq -\frac{1}{R} + A |n_{\gamma}(x)\cdot Dd_{\partial \mathcal{O}}(x)| \leq -\frac{1}{2 R}.
	\end{equation*}

Next, we consider the case when $x \in \gamma \cap \{ \nu \geq d_{\mathcal{O}_{1}} \geq -\nu'\}$.  Recall that in the construction of $\varphi$ through $\eta$, so far we have only needed to know that $\eta'(0) = A = \|\eta'\|_{L^{\infty}([-\nu,\nu])}$ with $A$ a fixed constant satisfying \eqref{E: A condition}.  Hence, we are still free at this stage to require the following condition on $\eta$:
	\begin{equation*}
		|\eta'(s)| \leq (2R)^{-1} \quad \textup{if} \, \, s \in [-\nu,-\nu'].
	\end{equation*} 
 With this condition in hand, we find
	\begin{equation*}
		- (1 + \varphi(x)) \kappa_{\gamma}(x) -  D\varphi(x)\cdot n_{\gamma}(x)  \leq -\frac{1}{R} + |\eta'(d_{\mathcal{O}}(x))| \leq -\frac{1}{2 R}.
	\end{equation*}
	
		Also notice that the restrictions on $\eta'$ are loose enough that we can still require $\|\eta\|_{L^{\infty}([-\nu,\nu])} \leq \zeta$, where $\zeta$ was the small parameter fixed at the start of the proof.
		
To summarize, in this step of the proof, we have shown that there is a constant $F_{2} > 0$ such that, for any $x \in \gamma$ satisfying $d_{k + \mathcal{O}}(x) > -\nu'$ for all $k \in \mathbb{Z}^{2}$, we have
	\begin{equation*}
		- (1 + \varphi(x)) \kappa_{\gamma}(x) -  D\varphi(x)\cdot n_{\gamma}(x) \leq -F_{2}.
	\end{equation*}

\textbf{Defining $F$, $(S_{e},U_{e})_{e \in \mathbb{E}^{2}}$, and $(S_{v},U_{v})_{v \in \mathbb{Z}^{2}}$.}  Finally, we define the local supersolutions that comprise our network.  To begin with, let $F = \min\{F_{1},F_{2}\}$.  

Recall from condition (viii) in Definition \ref{D: supersolution network} that we require translation invariance, i.e., $(S_{\nu + x},U_{\nu + x}) = x + (S_{\nu},U_{\nu})$ for all $x \in \mathbb{Z}^{2}$ and $\nu \in \mathbb{Z}^{2} \cup \mathbb{E}^{2}$, hence we only need to construct the supersolution node $(S_{0},U_{0})$ associated with the origin and the supersolution edges $(S_{e},U_{e})$ for edges $e$ containing the origin.

Let us begin with the node $(S_{0},U_{0})$.  Recall the smooth open set $\mathcal{O}$ defined above.  Let $U_{v} = \{x \in B_{1/4}(0) \, \mid \, d_{\mathcal{O}}(x) > -c \}$ for some small $c \in (0,\nu'/2)$ to be determined and let $S_{0} = \overline{\mathcal{O}} \cap \{x \in B(0,C) \, \mid \, d_{\mathcal{O}}(x) \geq -c\}$.   

Note that $\partial S_{0} \cap U_{0} = \partial \mathcal{O}$ so $(S_{0},U_{0})$ is a local supersolution of \eqref{e.sieqn} by Step 3 of the proof and Proposition \ref{P: boundary thing}.

Next, we construct the edge supersolutions $(S_{e},U_{e})$ for edges $e$ containing the origin.  We begin with $e = [(0,0),(1,0)]$.  Let $\gamma^{+}$ be the circular arc constructed above connecting $(0,0)$ to $(1,0)$.  Let $\delta \ll 1$ and define $(S_{e},U_{e})$ as follows:
	\begin{align*}
		U_{e} &= B_{R(t) + \delta}((1/2,-t)) \cap \{x \in \mathbb{R}^{2} \, \mid \, x_{2} > -\delta, \, \, d_{\mathcal{O}}(x) > -2c, \, \, d_{(1,0) + \mathcal{O}}(x) > -2c\}, \\
		S_{e} &= \overline{B_{R(t)}((1/2,-t))} \cap \{x \in \mathbb{R}^{2} \, \mid \, x_{2} \geq -\delta, \, \, d_{\mathcal{O}}(x) \geq -2c, \, \, d_{(1,0) + \mathcal{O}}(x) \geq -2c\}.
	\end{align*}
Notice that $\partial S_{e} \cap U_{e} = \gamma \cap U_{e}$ by construction, and hence $(S_{e},U_{e})$ is a local supersolution of \eqref{e.sieqn} with $F = F_{a}$.  Furthermore, $\overline{U}_{e} \setminus S_{e}$ is disjoint from the cube $(1/2,-1/2) + [-1/2,1/2]^{2}$ (cf.\ condition (vi) in Definition \ref{D: supersolution network}).

For the remaining edges, we use symmetries of $\mathbb{Z}^{2}$.  Define the rotation map $\mathcal{I}$ and reflection map $\mathcal{R}$ by $\mathcal{I}(x_{1},x_{2}) = (-x_{2},x_{1})$ and $\mathcal{R}(x_{1},x_{2}) = (x_{1},-x_{2})$.  Let $S_{[(1,0),(0,0)]} = \mathcal{R}(S_{[(0,0),(1,0)]})$ and $U_{[(1,0)],(0,0)]} = \mathcal{R}(U_{[(0,0),(1,0)]})$.  Similarly, we define
	\begin{align*} 
		S_{[(0,0),(0,1)]} = \mathcal{I}(S_{[(0,0),(1,0)]}), \quad U_{[(0,0),(0,1)]} = \mathcal{I}(U_{[(0,0),(1,0)]}), \\
		S_{[(0,1),(0,0)]} = \mathcal{I}(S_{[(1,0),(0,0)]}), \quad U_{[(0,1),(0,0)]} = \mathcal{I}(U_{[(0,1),(0,0)]}).
	\end{align*}
The remaining edge supersolutions $(S_{e},U_{e})$ with $e \ni 0$ are prescribed by translation invariance:
	\begin{align*}
		(S_{[(-1,0),(0,0)]},U_{[(-1,0),(0,0)]}) = (-1,0) + (S_{[(0,0),(1,0)]},U_{[(0,0),(1,0)]}), \\
		(S_{[(0,0),(-1,0)]},U_{[(0,0),(-1,0)]}) = (-1,0) + (S_{[(1,0),(0,0)]},U_{[(1,0),(0,0)]}), \\
		(S_{[(0,0),(0,-1)]},U_{[(0,0),(0,-1)]}) = (0,-1) + (S_{[(0,1),(0,0)]},U_{[(0,1),(0,0)]}), \\
		(S_{[(0,-1),(0,0)]},U_{[(0,-1),(0,0)]}) = (0,-1) + (S_{[(0,0),(0,1)]},U_{[(0,0),(0,1)]}).
	\end{align*} 

It is not hard to verify that if $\delta$ and $c$ are small enough, then the constructed network satisfies the conditions in Definition \ref{D: supersolution network}.  The details are left to the reader.


\section{Pinning in a Diffuse Interface Model}\label{s.diffusepinning}

In this section, we treat the diffuse interface setting, completing a construction analogous to that in Section \ref{s.mobility}. The basic idea is straightforward, we have proven the existence of pinned super/subsolutions for a non-trivial interval of forcing parameters in the sharp interface model, this gives us the room to approximate these solutions by a diffuse interface in the natural way and maintain the strict sub/supersolution property. 

At a technical level there are two main issues that we need to address.  First, the super/subsolutions we constructed in the previous section are not smooth, they have corner-type gradient discontinuities at a discrete set of points.  

Further, as we will see, \eref{mcfa} differs from \eref{forced allen cahn} by a square root.  Hence, in what follows, we let $a$ be as in Section \ref{s.mobility} and define $\theta \in C^{\infty}(\mathbb{T}^{2})$ by
	\begin{equation} \label{E: square root}
		\theta(x) = a(x)^{2}.
	\end{equation}
	Before proceeding further, notice that the first equations in \eref{forced allen cahn} are related through the scaling $(x,t) \mapsto (\ep^{-1} x, \ep^{-2}t)$.  Accordingly, in what follows, we will be frequently interested in the unscaled equation:
	\begin{equation} \label{e.aca forced}
		\delta (u_{\delta,t} - \Delta u_{\delta}) + \theta(x) (\delta^{-1} W'(u_{\delta}) - F ) = 0 \quad \textup{in} \, \, \mathbb{R}^{2} \times (0,\infty).
	\end{equation}
	
Lastly, we need to make explicit our assumptions on $W$:
	\begin{gather}
		W \in C^{3}([-3,3];[0,\infty)), \quad \{W' = 0\} = \{-1,0,-1\}, \label{A: smooth and zeros}\\
		(-1,0) \subset \{W' > 0\}, \quad (0,1) \subset \{W' < 0\}, \label{A: shape of W}\\
		\min\{W''(-1), W''(1)\} > 0, \quad W''(0) < 0 \label{A: nondegenerate zeros}
	\end{gather}
	
	Here is the main technical result of this section, which will be the key component of the proof of Theorem \ref{t.diffuse interface pinning}:

	\begin{lemma} \label{l.stationary solutions diffuse interface}  If $a$ is as in \lref{stationary solutions}, then there are constants $\delta_{0}, \beta, \bar{F}, C > 0$ such that, for each $K \subset \mathbb{R}^{2}$ satisfying an interior and exterior ball condition with large enough radius and each $\delta \in (0,\delta_{0})$, there is a continuous, stationary supersolution $u^{+}$ of \eqref{e.aca forced} with $F = \bar{F}$ and a continuous, stationary subsolution $u^{-}$ of \eqref{e.aca forced} with $F = -\bar{F}$ such that 
		\begin{gather*}
			\{x \in \mathbb{R}^{d} \setminus K \, \mid \, \textup{dist}(x,\partial K) \geq C\} \subset \{u^{-} = -1 - \beta \delta, \, \, u^{+} = -1 + 2\beta \delta\}, \\
			 \{x \in K \, \mid \, \textup{dist}(x,\partial K) \geq C\} \subset \{u^{-} = 1 - 2\beta \delta, \, \, u^{+} = 1 + \beta \delta \}, \\
			-1 + 2 \beta \delta \leq u^{+} \leq 1 + \beta \delta, \quad -1 - \beta \delta \leq u^{-} \leq 1 - 2 \beta \delta.
		\end{gather*}
	In particular, for each $F \in [-\bar{F},\bar{F}]$, there is a stationary solution $u$ of \eqref{e.aca forced} taking values in $[-(1 + \beta \delta),1 + \beta \delta]$ such that
		\begin{gather*}
			\{x \in K \, \mid \, \textup{dist}(x,K) \geq C\} \subset \{1 - \beta \delta \leq u \leq 1 + \beta \delta\}, \\
			\{x \in \mathbb{R}^{2} \setminus K \, \mid \, \textup{dist}(x,K) \geq C\} \subset \{-(1 + \beta \delta) \leq u \leq -1 + \beta \delta\}.
		\end{gather*}
	  \end{lemma}  
	
The construction of the supersolution $u^{+}$ proceeds in three steps.  First, for a sharp interface supersolution $E$ as in \sref{mobility}, we construct a ``level set function" $d$ with the property that the interfaces $\{d = s\}$ are sharp interface supersolutions close to $\partial E$ for all $s$ close enough to zero.  The second and third steps follow \cite{barles souganidis}.  In the second step, we use $d$ and the standing wave solution of the homogeneous Allen-Cahn equation to build a diffuse interface supersolution in the domain $\{|d| < \gamma\}$ for a suitable $\gamma > 0$.  Lastly, we extend this diffuse interface supersolution to the entire space.  

The subsolution $u^{-}$ is built analogously.  These sub- and supersolutions will be used to prove that the scaled problem \eref{forced allen cahn} is pinned (\tref{diffuse interface pinning}).

By taking $K$ to be a half space, we establish the existence of plane-like stationary solutions:

\begin{remark} \label{R: plane like}  Given $e \in S^{1}$, let $K = \{x \in \mathbb{R}^{2} \, \mid \, x \cdot e \leq 0\}$ in \lref{stationary solutions diffuse interface} and let $u$ be the associated stationary solution.  If $\delta$ is small enough, then it is not hard to show that $u$ satisfies
	\begin{equation*}
		\lim_{s \to \pm \infty} \sup \left\{ |u(x) - u^{\pm}(\alpha \delta)| \, \mid \, \pm (x \cdot e) \geq s \right\} = 0,
	\end{equation*} 
where $u^{-}(\alpha \delta) < u^{+}(\alpha \delta)$ are the unique stable critical points of $W(u) + \alpha \delta u$.  Hence $u$ is a plane-like solution heteroclinic to the two spatially homogeneous stationary solutions.   \end{remark}  

\subsection{Preliminaries}  In what follows, we let $D^{a} : \mathbb{R}^{2} \times \mathbb{R}^{2} \to [0,\infty)$ be the metric induced by $a$.  Specifically, this is the function defined by 
	\begin{align*}
		D^{a}(x,y) &= \inf \left\{ \int_{0}^{T} a(\gamma(s)) \|\dot{\gamma}(s)\| \, ds \, \mid \, T > 0, \right.\\
				&\qquad \qquad \left. \gamma \in AC([0,T]; \mathbb{R}^{2}), \, \, \gamma(0) = x, \, \, \gamma(T) = y \right\}.
	\end{align*}
Recall that $D^{a}$ is a metric on $\mathbb{R}^{2}$ equivalent to the Euclidean metric.  Furthermore, $D^{a}$ is invariant under integer translations in the following sense:
	\begin{equation}
		D^{a}(x + k,y+k) = D^{a}(x,y) \quad \textup{if} \, \, x,y \in \mathbb{R}^{2}, \, \, k \in \mathbb{Z}^{d}.
	\end{equation}

Given a (non-empty) set $A \subset \mathbb{R}^{2}$, define the $a$-distance $\textup{dist}^{a}(\cdot,A) : \mathbb{R}^{2} \to [0,\infty)$ to $A$ as follows:
	\begin{equation*}
		\textup{dist}^{a}(x,A) = \inf \left\{ D^{a}(x,y) \, \mid \, y \in A \right\}.
	\end{equation*}
	
	We also introduce the Allen-Cahn one-dimensional transition front associated with the homogeneous energy function with $\theta \equiv 1$.  We call $q:\R \to \R$ to be the solution of the second order ODE
	\begin{equation} \label{E: standing wave}
		\ddot{q} = W'(q) \ \hbox{ with } \ \lim_{s \to -\infty} q(s) = -1 \ \hbox{ and } \ \lim_{s \to \infty} q(s) = 1.
	\end{equation}
	Standard computations find that
	\[ \dot{q} = \sqrt{2 W(q)} \]
	and from this first order ODE plus the previous boundary conditions at $\pm \infty$ it is easy to see
	\[ q \in (-1,1) \ \hbox{ and } \ \dot{q} > 0.\]

\subsection{Modifying the Interfaces}  Given a set $K$ satisfying exterior and interior ball conditions of radius $R$, let $E = E(K)$ be the supersolution of \eref{sieqn} constructed by the algorithm of Section \ref{s.mobility}.  Let $d_{E} : \mathbb{R}^{2} \to \mathbb{R}$ be the signed distance to $E$, that is, the function given by
	\begin{equation*}
		d_{E}(x) = \left\{ \begin{array}{r l}
							\textup{dist}^{a}(x,\mathbb{R}^{2} \setminus E), & \textup{if} \, \, x \in \overline{E}, \\
							-\textup{dist}^{a}(x,E), & \textup{if} \, \, x \in x \in \mathbb{R}^{2} \setminus E,
						\end{array} \right.
	\end{equation*}
If $E$ were smooth and compact, then it would be easy to see that, at least close to $E$, $d_{E}$ is a supersolution of a stationary level set PDE.  Our setting complicates things slightly, but not irredeemably.  

The following property about $E$ is sufficient for our immediate purposes:

\begin{property}\label{prop.locallyamin}
 There is a collection of local supersolutions $(S_i,U_i)|_{i \in I}$ of \eref{sieqn} with some positive forcing $F = F_0>0$ such that the sets $\partial S_i \cap U_{i}$ are smooth uniformly in $i$ and there is an $r>0$ so that, for all $x_0 \in \partial E$, there is a finite sub-collection $I'(x_{0}) \subset I$ such that
\[ E \cap B_r(x_0) = (\cap_{i \in I'(x_{0})} S_i ) \cap B_r(x_0).\]

\end{property}

In words, the supersolutions constructed in \sref{mobility} are, locally, an intersection of a finite number of local smooth supersolutions.  See the proofs of \lref{edgepatch} and \lref{join a node and an edge}, which show that the $\textup{patch}$ and $\textup{node.join}$ operations create a supersolution which is, locally, an intersection of the input supersolutions.  (Recall from Definition \ref{D: supersolution network} and Remark \ref{R: translation-invariance-useful} after that the ``basic building block" supersolution edges and nodes in the network are uniformly smooth.)

	\begin{proposition} \label{P: distance viscosity supersolution} There is an $r > 0$ depending on the network constructed in \sref{mobility}, but not the particular choice of $E$, such that $d_{E}$ satisfies the following viscosity inequalities:
		\begin{align}\label{e.distfuncequation}
			\|Dd_{E}\|^{2} &\geq a(x)^{2} \quad \textup{in} \, \, \{0 < d_{E} < r\}, \\
			 - a(x) \mathcal{MC}_{*}(Dd_{E}, D^{2}d_{E}) &- Da(x) \cdot Dd_{E} \geq \frac{F_0}{2}\|Dd_{E}\| \quad \textup{in} \, \, \{0 < d_{E} < r\}. \notag
		\end{align}
	\end{proposition}  
	
		\begin{proof}  
		Let $(S_i,U_i)$ be the collection of supersolutions from Property \ref{prop.locallyamin}. The $S_i$ have smooth boundary in $U_i$ uniformly in $i$ so there is $r_1>0$ such that the signed $a$-distance functions $d_{S_i}$ in the tubular neighborhoods $\{|d_{S_i}| < r_1\} \cap U_i$ are smooth and satisfy the following differential inequalities in the classical sense:
			\begin{align*}
				\|Dd_{S_i}\|^{2} &= a(x)^{2}, \\
				- a(x) \textup{tr} \left( \left(\textup{Id} - \widehat{Dd_{S_i}} \otimes \widehat{Dd_{S_i}} \right) D^{2} d_{S_i} \right) &- Da(x) \cdot Dd_{S_i} \geq \frac{F_0}{2} \|Dd_{S_i}\|.
			\end{align*}

		Let $x_0 \in \partial E$.  By Property~\ref{prop.locallyamin} there is $r_2>0$ (independent of $x_0$ and, without loss, smaller than $\frac{r_1}{2 \Lambda}$) so that
		
		\[ E \cap B_{r_2}(x_0) = ( \bigcap_{i \in I'} S_i ) \cap B_{r_2}(x_0)\] 
		for some subcollection $I' \subset I$.  We can also add the following requirement without loss: $\partial S_i \cap B_{r_2}(x_0) \neq \emptyset$ for all $I \in I'$.   With this additional property, and since $r_2 \leq \frac{r_1}{2\Lambda}$,
		\[ B_{r_2}(x_0) \subset \{|d_{S_i}| < r_1 \} \ \hbox{ for all } \ i \in I'\]
		and so $d_{S_i}$ are supersolutions of \eref{distfuncequation} in $B_{r_2}(x_0)$ for all $i \in I'$.
		
		Further note that for $x \in B_{\alpha r}(x_0)$ the unsigned distance satisfies 
		\[ |d_E(x)| \leq D^a(x,x_0) \leq \Lambda \alpha r \ \hbox{ and} \  |d_{ B_r(x_0)}(x)| \geq (1-\alpha)r\]
		and so
		\[ d_E(x) = d_{E \cap B_{r}(x_0)}(x) \ \hbox{ for } x \in B_{\frac{r}{1+\Lambda}}(x_0)\]
		Thus
		\begin{equation}\label{e.mindists}
		 d_{E}(x) =  \inf_{i \in I'} d_{S_i}(x) \ \hbox{ for } \ x \in B_{\frac{r_2}{1+\Lambda}}(x_0) \cap E.
		 \end{equation}
		 Let us call $r_3 = \frac{r_2}{1+\Lambda}$.

		  		 By the formula \eref{mindists}, and since the minimum of supersolutions is a supersolution, we find that $d_E$ is a supersolution of \eref{distfuncequation} in the region
			\[  B_{r_3}(x_0) \cap \textup{Int}(E) \] 
			 Since $x_0 \in \partial E$ was arbitrary and the radius $r_3$ did not depend on the particular $x_0$ we have that $d_E$ is a supersolution of \eref{distfuncequation} in the region
			 \[ (\partial E + B_{r_3}(0)) \cap E \supset \{ 0 < d_E(x) < r_3\}.\]
		
				\end{proof}

\subsection{Diffuse interface near $\partial E$}   

It will be convenient in what follows to recenter around $d_E = r/2$ by
	\begin{equation*}
		d(x) = d_{E}(x) -\frac{r}{2}.
	\end{equation*} 
Hence $\{-r/2 < d < r/2\} = \{0 < d_{E} < r\}$.  Note this changes none of the viscosity inequalities we have proven above since they are all invariant under addition of constants.

For the moment, let $\delta, \beta > 0$ be free variables.  Define $v_{\delta} : \{-r/2 < d < r/2\} \to [-2,2]$ by 
	\begin{equation*}
		v_{\delta}(x) = q\left(\frac{d(x)}{\delta} \right) + 2\beta \delta
	\end{equation*}
	where $q(s)$ is the solution of \eqref{E: standing wave}.
We claim that if $\delta$ and $\beta$ are chosen appropriately, then $v_{\delta}$ is a strict supersolution of \eqref{e.aca forced} in $\{-r/2 < d < r/2\}$ for sufficiently small $\alpha>0$ .  

Note that $q$ is increasing and smooth so if a smooth test function $\varphi$ touches $v_{\delta}$ from below at some point $x_0$, then $\delta q^{-1}(\varphi-2\beta \delta)$ touches $d$ from below at $x_0$.  We will compute as if $d$ is smooth, but technically one does the computations on a smooth touching test function as is standard in viscosity solution theory (also the specific $d$ under consideration is smooth at any point where it can be touched from below by a smooth test function in the neighborhood considered).  We compute
	\begin{align*}
		-\delta \Delta v_{\delta}(x) + \delta^{-1} \theta(x) W'(v_{\delta}(x)) &= - \delta^{-1} \ddot{q}\left(\frac{d(x)}{\delta}\right) \left( \|Dd(x)\|^{2} - a(x)^{2} \right) \\
		&\quad - \dot{q}\left(\frac{d(x)}{\delta}\right) \Delta d(x) + 2\beta \theta(x) W''\left(q\left(\frac{d(x)}{\delta}\right)\right)+O(\beta^2\delta)
	\end{align*}
	where the $O(\beta^2\delta)$ error term can be bounded, more precisely, by
	\[ 2\beta^2 \delta \theta(x)\sup_{[-1,1]} |W'''|.\]
Since $\|Dd\|^{2} = a^{2}$ in a neighborhood of $x$, it follows that $D^{2}d(x) Dd(x) = a(x) Da(x)$.  Thus,
	\begin{align*}
		- \Delta d(x) &= - \textup{tr} \left( \left(\textup{Id} - \widehat{Dd}(x) \otimes \widehat{Dd}(x) \right) D^{2} d(x) \right) - a(x)^{-2} D^{2} d(x) Dd(x) \cdot Dd(x) \\
		&= - \textup{tr} \left( \left( \textup{Id} - \widehat{Dd}(x) \otimes \widehat{Dd}(x) \right) D^{2} d(x) \right) - a(x)^{-1} Da(x) \cdot Dd(x) \geq \frac{F_0}{2}.
	\end{align*}  
Recalling that $W''$ is bounded from below away from $0$ in a neighborhood $N$ of $\{-1,1\}$, we can choose $\beta>0$ small so that
\[  \frac{F_0}{4}\inf_{[-1,1] \setminus N} \sqrt{2W(q)} - 2\beta \Lambda \sup_{[-1,1]} |(W'')_-| \geq 0\]
and
\[ \beta \sup_{[-1,1]} |W'''| \leq  \inf_{N} W''.\]
Then, as in \cite[Lemma 4.3]{barles souganidis}, we deduce that there is $F_1(F_0,W) > 0$ such that, for any $\delta<1$,
	\begin{align}  \notag
			- \delta \Delta v_{\delta}(x) + &\delta^{-1} \theta(x) W'(v_{\delta}(x)) \\
			\notag &\geq \frac{F_0}{2} \dot{q}\left(\frac{d(x)}{\delta}\right) + 2 \beta \theta(x) W''\left(q\left(\frac{d(x)}{\delta}\right)\right)  - 2\beta^2 \delta \theta(x)\sup_{[-1,1]} |W'''|\\
		&\geq F_1.  \label{e.choose beta}
	\end{align}
	
\begin{remark} \label{R: other energies construction}  This section is the only part of the argument where we use the specific form of \eqref{e.diffuseenergy}.  If instead we wanted to build sub- or supersolutions for the Euler-Lagrange equation associated with the energy model \eqref{E: gradient model}, the $L^2$ gradient flow is 
	\begin{equation*}
		\delta (u_{\delta,t} - \theta(x) \Delta u_{\delta} - D\theta(x) \cdot Du_{\delta}) + \delta^{-1} W'(u_{\delta}) = 0.
	\end{equation*}
Hence when we invoke an ansatz of the form $u_{\delta}(x,t) = q(\frac{d(x)}{\delta}) + \dots$, we find
	\begin{equation*}
		0 = \delta^{-1} (-\theta(x) \ddot{q}(\tfrac{d}{\delta})\|Dd\|^{2} + W'(q(\tfrac{d}{\delta}))) + \dot{q}(\tfrac{d}{\delta})(- \Delta d - D\theta(x) \cdot Dd) + \dots 
	\end{equation*}
Notice that, in this case, the highest order term suggests the identity $a(x)^{2} \|Dd(x)\|^{2} = \theta(x) \|Dd(x)\|^{2} = 1$.  Thus, the only change necessary is to replace the Riemannian metric $D^{a}$ above by $D^{a^{-1}}$ (i.e.\ interchange $a$ with $a^{-1}$).

Where \eqref{E: weight model} is concerned, since $a = \sqrt{\theta}$, the gradient flow is
	\begin{equation*}
		\delta (u_{\delta, t} - a(x) \Delta u_{\delta} - D a(x) \cdot Du_{\delta}) + \delta^{-1} a(x) W'(u_{\delta}) = 0.
	\end{equation*}
Employing the ansatz $u_{\delta}(x,t) = q(\delta^{-1}d(x)) + \dots$, we obtain
	\begin{equation*}
		0 = \delta^{-1} a(x) (-\ddot{q}(\tfrac{d}{\delta}) \|Dd\|^{2} + W'(q(\tfrac{d}{\delta}))) + \dot{q}(\tfrac{d}{\delta}) (- a(x) \Delta d - Da(x) \cdot Dd) + \dots
	\end{equation*}
Accordingly, in this case, the Euclidean distance should replace $D^{a}$ in the definition of $d$.
\end{remark}  
	
\subsection{Diffuse interface outside of $\{-r/2 < d < r/2\}$}  We proceed to extend $v_{\delta}$ to the whole space.  On the one hand, when $d \geq r/2$, the function $v_{\delta}$ as defined above is almost $1$ so we can simply take the minimum.  When $d \leq -r/2$, we interpolate between $v_{\delta}$ and $-1$ using a partition of unity. 

Most of the work is in the interpolation.  Let $\lambda : \mathbb{R} \to [0,1]$ be a smooth, increasing function such that $\lambda(u) = 0$ if $u \leq -\frac{3r}{8}$ and $\lambda(u) = 1$ if $u \geq -\frac{ r}{8}$.  We wish to define $u_{\delta} : \{-r/2 < d_{E} < r/4\} \to [-2,2]$ by 
	\begin{equation*}
		u_{\delta}(x) = \lambda(\underline{d}(x)) v_{\delta}(x) + (1 - \lambda(\underline{d}(x)))(-1 + 2\beta \delta)
	\end{equation*}
for some suitable smoothed function $\underline{d}$ approximating $d$.

	\begin{lemma}  There is a smooth function $\underline{d} : \{-r/2 < d < r/2\} \to \mathbb{R}$ with bounded first and second derivatives such that the following inclusions hold:
		\begin{equation*}
			\left\{d \leq -\frac{7r}{16}\right\} \subset \left\{\underline{d} \leq -\frac{3r}{8}\right\} \subset \left\{d \leq -\frac{3r}{8}\right\}, \\
			\left\{ d \geq -\frac{r}{8} \right\} \subset \left\{\underline{d} \geq -\frac{ r}{8}\right\} \subset \left\{d \geq -\frac{r}{4}\right\}.
		\end{equation*}
	\end{lemma}    
	
		\begin{proof}  Given a mollifying family $(\rho_{\zeta})_{\zeta > 0}$, define $\underline{d}$ by 
			\begin{equation*}
				\underline{d}(x) = \int_{\mathbb{R}^{d}} d(y) \rho_{\zeta}(x - y) \, dy + c
			\end{equation*}
		for small constants $c,\zeta > 0$.  The boundedness of $\|Dd_{E}\|$ implies $\underline{d}$ has bounded first and second derivatives with bounds depending only on $\zeta$.  Further, for the same reason, $\zeta$ can be chosen independently of $E$ (or $K$).  \end{proof}  
		
It is now a more-or-less straightforward adaptation of \cite{barles souganidis} to show that $u_{\delta}$ is a supersolution in $\{-r/2 < d(x) < 0\}$.  

	\begin{lemma} \label{l.first part}  There is a $\delta_{1} > 0$ and $\bar{F}>0$ depending only on $\theta$, $W$, and the choice of the network in \sref{mobility} such that if $\delta \in (0,\delta_{1})$ is sufficiently small, then $u_{\delta}$ is a supersolution of 
		\begin{equation} \label{e.acsupersoln1}
		-\delta \Delta u_{\delta} + \delta^{-1} \theta(x) W'(u_{\delta}) \geq \bar{F} \theta(x) \quad \textup{in} \, \, \{ d(x) < r/2\}.
	\end{equation}  
  \end{lemma}  
	
		\begin{proof}  
		First note that $u_{\delta} \equiv v_{\delta}$ in $\{ d(x) > -r/8\}$ so, since $v_{\delta}$ is supersolution of \eref{choose beta}, so is $u_{\delta}$ in $\{ -r/8 < d(x) < r/2\}$.  
		
		Meanwhile in $\{d(x) < -7r/16\}$ we have $u_{\delta} \equiv -1 + 2\beta \delta$ which is also supersolution of \eref{acsupersoln} for small enough $\delta$ since $W''(-1)>0$.  
		
		This leaves to check the region $\{ -7r/16 \leq d(x) \leq -r/8\}$.  Note that the supersolution property \eref{choose beta} for $v_{\delta}$ does hold in this region.

		If $d$ and, correspondingly, $v_{\delta}$ were smooth then the computation of \cite[Lemma 6]{allencahnmobility} (cf.\ \cite[Lemma 4.5]{barles souganidis}) would go through exactly to find \eref{acsupersoln1} for $\delta>0$ sufficiently small.  
		
		As is standard in viscosity solution theory we can carry over the computations which rely on differentiability to a touching test function.  The only issue is that $x \mapsto \lambda(\underline{d}(x))$ is not strictly positive so some care is required at points this function vanishes.  
		
		To address this, for $n \in \mathbb{N}$, define $u_{\delta}^{(n)}$ by 
			\begin{equation*}
				u_{\delta}^{(n)}(x) = (\lambda(\underline{d}(x)) + n^{-1} \delta) v_{\delta}(x) + (1 - \lambda(\underline{d}(x)))(-1 - 2 \beta \delta).
			\end{equation*}
			Now if $\varphi$ is a smooth test function touching $u_{\delta}^{(n)}$ from below at some point then
			\[ \tilde{\varphi}(x)=\frac{\varphi(x) - (1 - \lambda(\underline{d}(x)))(-1 - 2 \beta \delta)}{\lambda(\underline{d}(x)) + n^{-1} \delta}\]
			will touch $v_{\delta}$ from below at the same point.
	 
		Arguing as in \cite[Lemma 6]{allen cahn mobility}, we see that there is a $\delta_{1} > 0$ such that if $\delta \in (0,\delta_{1})$, then $u_{\delta}^{(n)}$ is a supersolution of \eref{acsupersoln1} as soon as $n$ is large enough.  Sending $n \to \infty$, we deduce that $u_{\delta}$ is a supersolution by stability.    
		
		As for the dependence of $\delta_{1}$, we only need to be able to eliminate the error terms in the construction above; and these depend only on $\Lambda$, $W$, and the size of the derivatives of $\underline{d}$, which are determined by $\zeta$.  \end{proof}  
		
Finally, we extend $u_{\delta}$ to a supersolution $\bar{u}_{\delta}$ as follows:
	\begin{equation*}
		\bar{u}_{\delta}(x) = \left\{ \begin{array}{r l}
								\min\left\{u_{\delta}(x), 1 + \beta \delta \right\}, & \textup{if} \, \,  d(x) <  r/2, \\
								1 + \beta \delta, & \textup{if} \, \, d(x) \geq r/2.
						\end{array} \right.
	\end{equation*}

\begin{proposition} \label{p.second part}  There are constants $\bar{\delta} > 0$ and $\bar{F} > 0$ depending on the network constructed in \sref{mobility} and on $W$, but not on $E$, such that if $\delta \in (0,\bar{\delta})$ is sufficiently small, then $\bar{u}_{\delta}$ is continuous in $\mathbb{R}^{2}$ and satisfies the following differential inequality in the viscosity sense:
	\begin{equation} \label{e.acsupersoln}
		-\delta \Delta \bar{u}_{\delta} + \delta^{-1} \theta(x) W'(\bar{u}_{\delta}) \geq \bar{F} \theta(x) \quad \textup{in} \, \, \mathbb{R}^{2}.
	\end{equation}  \end{proposition}

	\begin{proof}  
	
	Given \lref{first part} we just need to check that the constant function $1+\beta \delta$ is a supersolution of \eref{acsupersoln} and that $\bar{u}^\delta$ is identically equal to $1+\beta \delta$ in a neighborhood of $\{ d(x) = r/2\}$.

Since $W''(1) > 0$ and $\theta > 0$, the constant $1 + \beta \delta$ is a supersolution of \eref{acsupersoln} as soon as $\delta > 0$ is small enough.  
	
	If $x \in \{d(x) > \frac{r}{8}\}$ then, from the exponential convergence of $q$, 
		\begin{equation*}
			u_{\delta}(x) = v_{\delta}(x) \geq 1 - C \exp\left(-\frac{r}{8C \delta}\right) + 2\beta \delta > 1 + \beta \delta.
		\end{equation*}
	for $\delta > 0$ sufficiently small.  Therefore, $\bar{u}_{\delta} = 1 + \beta \delta$ in $\{d(x) > \frac{r}{8}\}$.

	\end{proof}    
	
The next remark puts the computations above into some context.
	
\begin{remark} \label{R: sharp interface limit} By reprising the arguments just presented, one can show that, as $\delta \to 0^{+}$, solutions of the Cauchy problem
	\begin{equation}
		\left\{ \begin{array}{r l}
				\delta \left(u_{\delta, t} - \Delta u_{\delta} \right) - \delta^{-1} \theta(x) W'(u_{\delta}) = 0 & \textup{in} \, \, \mathbb{R}^{2} \times (0,\infty), \\
				u_{\delta} = u_{0} & \textup{on} \, \, \mathbb{R}^{2} \times \{0\},
			\end{array} \right.
	\end{equation}
concentrate along interfaces whose motion is governed by the level set PDE
	\begin{equation*}
		\left\{ \begin{array}{r l}
				a(x) u_{*,t} - a(x) \textup{tr} \left( \left(\textup{Id} - \widehat{Du_{*}} \otimes \widehat{Du_{*}} \right) D^{2} u_{*} \right) - Da(x) \cdot Du_{*} = 0 & \textup{in} \, \, \mathbb{R}^{2} \times (0,\infty), \\
				u_{*} = u_{0} & \textup{on} \, \, \mathbb{R}^{2} \times \{0\}.
			\end{array} \right.
	\end{equation*}
This can be seen using the ansatz $u_{\delta}(x,t) = q(\delta^{-1} d(x,t))$, where $d$ is the signed distance to $\{u_{*}(\cdot,t) = 0\}$ with respect to the Riemannian metric $D^{a}$ above.  

  Similarly, arguing as in \cite{MR1969814}, one can show that the energy \eref{diffuseenergy} $\Gamma$-converges to \eref{surfaceenergy} as $\delta \to 0^{+}$. \end{remark}

\subsection{Proof of \lref{stationary solutions diffuse interface}}  The computations of the previous three sections readily lead to a proof of the main result on stationary solutions with non-zero forcing.

\begin{proof}[Proof of \lref{stationary solutions diffuse interface}]  Given $K$, \lref{stationary solutions} furnishes a stationary supersolution $E$ of \eref{sieqn} with positive forcing $F_0>0$ such that $K \subset E$ and $d_{H}(K,E) + d_{H}(\partial K, \partial E) \leq C$.  The arguments of the previous subsection show that, for $\delta > 0$ small enough depending only on the coefficient $a$, there is a stationary supersolution $u^{+}$ of \eqref{e.aca forced} with $\alpha = \bar{F}>0$ such that 
	\begin{equation*}
		\left\{ \begin{array}{r l}
			u^{+} = 1 + \beta \delta & \textup{in} \, \, \{x \in E \, \mid \, \textup{dist}(x,\partial E) > r\}, \\
		u^{+} \leq -1 + \beta \delta & \textup{in} \, \, \mathbb{R}^{2} \setminus E.
		\end{array} \right.
	\end{equation*}
We are taking $r < 1$ so we have
	\begin{gather*}
		\{x \in K \, \mid \, \textup{dist}(x,\partial K) \geq C + 1\} \subset \{u^{+} = 1 + \beta \delta\}, \\
		\{x \in \mathbb{R}^{2} \setminus K \, \mid \, \textup{dist}(x,\partial K) \geq C + 1\} \subset \{u^{+} \leq -1 + \beta \delta\}.
	\end{gather*}

As in the sharp interface setting, the existence of supersolutions implies that of subsolutions.  To see this, notice that $u$ is a subsolution of \eqref{e.aca forced} if and only if $-u$ is a supersolution of \eqref{e.aca forced} with $W'$ replaced by the function $u \mapsto -W'(-u)$ and $\alpha$ replaced by $-\alpha$.  This does not change $a$ (and the equation for $E$ is correspondingly changed) so the construction goes through.

Finally, given $\alpha \in [-\bar{F},\bar{F}]$, we construct a stationary solution $u$ with the desired properties employing Perron's Method with $u^{+}$ and $u^{-}$ serving as barriers.  \end{proof}

\subsection{Sharp Interface Limit}  Using the stationary solutions furnished by \lref{stationary solutions diffuse interface}, we now prove that, for any sufficiently small external force $\alpha \in [-\bar{F},\bar{F}]$, the macroscopic interfaces associated with \eqref{e.aca forced} are pinned, i.e. we prove \tref{diffuse interface pinning}.

  In view of the assumptions \eqref{A: smooth and zeros}, \eqref{A: shape of W}, \eqref{A: nondegenerate zeros} on $W$, there is $\alpha_0>0$ so that for any $\alpha \in [-\alpha_0,\alpha_0]$ the perturbed potential $W_{\alpha}$ given by $W_{\alpha}(u) = W(u) - \alpha u$ satisfies the same assumptions.  This follows directly from the implicit function theorem.  

In particular, for $\alpha \in [-\alpha_{0},\alpha_{0}]$, we can let $u^{-}(\alpha) < u^{0}(\alpha) < u^{+}(\alpha)$ denote the critical points of $W_{\alpha}$ in $[-3,3]$, and $W_{\alpha}$ satisfies
		\begin{gather*}
			\{W'_{\alpha} = 0\} = \{u^{-}(\alpha),u^{0}(\alpha),u^{+}(\alpha)\}, \quad W''_{\alpha}(u^{\pm}(\alpha)) > 0, \quad W''_{\alpha}(u^{0}(\alpha)) < 0, \\
			(u^{-}(\alpha),u^{0}(\alpha)) \subset \{W'_{\alpha} > 0\}, \quad (u^{0}(\alpha), u^{+}(\alpha)) \subset \{W'_{\alpha} < 0\},
		\end{gather*} 
		and, for each $j \in \{+,-,0\}$,
		\[ |u^j(\alpha) - u^j(0)| \leq C\alpha. \]
Note that this implies there is an $F_{1} \in [0,\alpha_{0})$ (depending on $\beta$ from Lemma \ref{l.stationary solutions diffuse interface} and $C$ from the previous line) such that, for $\delta<1$ and $F \in [-F_{1},F_{1}]$, we have
	\begin{equation*}
		-(1 + \beta \delta) < u^{-}(F\delta) < -1 + \beta \delta < u^{0}(F\delta) < 1 - \beta \delta < u^{+}(F\delta) < 1 + \beta \delta.
	\end{equation*}
	
In the proof that follows, we will invoke the next ``initialization" result, which ensures that solutions of \eqref{e.aca forced} concentrate at the minimizers $u^{+}(\alpha)$ and $u^{-}(\alpha)$:

	\begin{proposition} \label{p.initialization} Fix an $a \in C(\mathbb{T}^{2}; [1,\Lambda])$, $\delta \in (0,1)$, and  $F \in [-F_{1},F_{1}]$.  Suppose that $y_{0} \in \mathbb{R}^{2}$, $r_{0} > 0$, $u_{0} \in UC(\mathbb{R}^{2}; [-3,3])$, and $B_{r_{0}}(y_{0}) \subset \{u_{0} > u^{0}(F\delta)\}$.  For each $\nu \in (0,r_{0})$, there are constants $\tau_{\nu}, \ep_{\nu} > 0$ such that if $(u^{\ep})_{\ep > 0}$ are the solutions of \eref{forced allen cahn}, then, for each $\ep \in (0,\ep_{\nu})$,
		\begin{equation*}
			u^{\ep}(\cdot,\tau_{\nu} \ep^{2} |\log(\ep)|) \geq (u^{+}(F\delta) - \ep) \ind_{B_{\nu}(y_{0})} -(1 + \beta \delta) \ind_{\mathbb{R}^{2} \setminus B_{\nu}(y_{0})} \quad \textup{in} \, \, \mathbb{R}^{2}.
		\end{equation*}\end{proposition} 

The proof of \pref{initialization}, which follows as in \cite{barles souganidis} and \cite{allen cahn mobility}, is briefly reviewed at the end of this section.  Note that a symmetrical result applies if instead $B_{r_{0}}(y_{0}) \subset \{u_{0} < u^{0}(F\delta)\}$.

\begin{proof}[Proof of \tref{diffuse interface pinning}]   Let $a$ be the coefficient field constructed in \sref{mobility} and $\theta(x)= a(x)^2$. Also recall $\bar{\delta}, \bar{F} > 0$ from \pref{second part} and $F_{1}$ from the discussion preceding Proposition \ref{p.initialization}.  If necessary we can make $\bar{F}$ and $\bar{\delta}$ smaller so that $\bar{F} < F_{1}$ and $\bar{\delta} < 1$.

Given $F \in [-\bar{F},\bar{F}]$ and $\delta \in (0,\bar{\delta})$, we will show below that the solutions $(u^{\ep})_{\ep > 0}$ of \eref{forced allen cahn} converge to $u^{+}(F\delta)$ locally uniformly in $\{u_{0} > u^{0}(F\delta)\} \times (0,\infty)$.  The corresponding argument for $\{u_{0} < u^{0}(F\delta)\} \times (0,\infty)$ is left to the reader.

Fix $y_{0} \in \mathbb{R}^{2}$ and $r > 0$ such that $B := B_{r}(y_{0}) \Subset \{u_{0} > u^{0}(F\delta)\}$.  If $\ep > \ep_{0}(B)$, then the ball $\ep^{-1} B$ satisfies the hypotheses of \lref{stationary solutions diffuse interface}.  Thus, for such $\ep$, there is a stationary subsolution $\tilde{u}^{-,\ep}$ of \eqref{e.aca forced} such that 
	\begin{gather*}
		\{y \in \ep^{-1} B \, \mid \, \textup{dist}(y, \partial (\ep^{-1} B)) \geq C\} \subset \{\tilde{u}^{-,\ep} \geq 1 - \beta \delta \}, \\
		\{y \in \mathbb{R}^{2} \setminus \ep^{-1} B \, \mid \, \textup{dist}(y, \partial (\ep^{-1} B) \geq C\} \subset \{\tilde{u}^{-,\ep} = -1 - \beta \delta\}.
	\end{gather*}
Henceforth, define $u^{-,\ep}$ by $u^{-,\ep}(x) = \tilde{u}^{-,\ep}(\ep^{-1} x)$.  It is a subsolution of \eref{forced allen cahn}.

Fix a small $\nu > 0$ such that $B_{r + \nu}(y_{0}) \subset \{u_{0} > u^{+}(F\delta)\}$.  \pref{initialization} implies that there are constants $\tau_{*}, \ep_{*} > 0$ depending only on $\nu$ such that if $\ep < \ep_{*}$, then
	\begin{equation*}
		u^{\ep}(\cdot, \tau_{*} \ep^{2} |\log(\ep)|) \geq (u^{+}(F\delta) - \ep) \ind_{B_{r + \nu/2}(y_{0})} - (1 + \beta \delta) \ind_{\mathbb{R}^{2} \setminus B_{r + \nu/2}(y_{0})} \quad \textup{in} \, \, \mathbb{R}^{2}.
	\end{equation*}
Since $u^{+}(F\delta) > 1 - \beta \delta$, it follows that, for all $\ep > 0$ sufficiently small,
	\begin{equation*}
		(u^{+}(F\delta) - \ep) \ind_{B_{r + \nu/2}(y_{0})} - (1 + \beta \delta) \ind_{\mathbb{R}^{2} \setminus B_{r + \nu/2}(y_{0})} \geq u^{-,\ep} \quad \textup{in} \, \, \mathbb{R}^{2}.
	\end{equation*}
Therefore, by the comparison principle,
	\begin{equation*}
		u^{\ep}(\cdot + (0,\tau_{*} \ep^{2} |\log(\ep)|)) \geq u^{-,\ep} \quad \textup{in} \, \, \mathbb{R}^{2} \times (0,\infty).
	\end{equation*}

Since $u^{-,\ep} \geq 1 - \beta \delta \geq u^{0}(F\delta)$ in $B_{3r_{0}/4}(y_{0})$ for small enough $\ep$, we can apply \pref{initialization} again to find $\tau_{**}, \ep_{**} > 0$ such that
	\begin{equation*}
		u^{\ep}(\cdot + (0,(\tau_{*} + \tau_{**}) \ep^{2} |\log(\ep)|)) \geq (u^{+}(F\delta) - \ep) \ind_{B_{r_{0}/2}(y_{0})} - (1 + \beta \delta) \ind_{\mathbb{R}^{2} \setminus B_{r_{0}/2}(y_{0})} \quad \textup{in} \, \, \mathbb{R}^{2} \times (0,\infty).
	\end{equation*}
From this, we conclude
	\begin{equation*}
		\liminf \nolimits_{*} u^{\ep} \geq u^{+}(F\delta) \quad \textup{in} \, \, B_{r_{0}/2}(y_{0}) \times (0,\infty).
	\end{equation*}
In particular, since $B$ was an arbitrary ball, it follows that $\liminf_{*} u^{\ep} \geq u^{+}(F\delta)$ in $\{u_{0} > u^{0}(F\delta)\} \times (0,\infty)$. 

On the other hand, since $u_{0} \leq 3$ in $\mathbb{R}^{2}$, it follows that $u^{\ep} \leq v^{\ep}$ in $\mathbb{R}^{2} \times (0,\infty)$, where $v^{\ep}(x,t) = \tilde{v}(\ep^{-2} t)$ is determined by the solution of the ODE
	\begin{equation*}
		\dot{\tilde{v}} = - \delta^{-1} W'(\tilde{v}) + F \tilde{v} \quad \textup{in} \, \, \mathbb{R}, \quad \tilde{v}(0) = 3.
	\end{equation*}
It is easy to check from the phase line analysis that $\tilde{v}(T) \to u^{+}(F\delta)$ as $T \to \infty$.  Thus, 
	\begin{equation*}
		\limsup \nolimits^{*} u^{\ep} \leq \lim_{\ep \to 0^{+}} v^{\ep} = u^{+}(F\delta) \quad \textup{in} \, \, \mathbb{R}^{2} \times (0,\infty),
	\end{equation*}  
We conclude that $u^{\ep} \to u^{+}(F\delta)$ locally uniformly in $\{u_{0} > u^{0}(F\delta)\} \times (0,\infty)$.

A similar argument proves convergence to $u^{-}(F\delta)$ in $\{u_{0} < u^{0}(F\delta)\} \times (0,\infty)$. \end{proof}

Finally, here is a sketch of the proof of \pref{initialization}:

\begin{proof}[Sketch of Proof of \pref{initialization}]  We argue as in \cite[Proposition 4.1]{barles souganidis} constructing the desired subsolution as in \cite[Appendix A]{allen cahn mobility}.  In the notation of the latter reference, in the present setting, $\overline{f}$ is defined by 
	\begin{equation*}
		\overline{f}(u) = \left\{ \begin{array}{r l}
							-\delta^{-1} W'_{F \delta}(u), & \textup{if} \, \, u \in [-3,u^{-}(F \delta)]\cup [u^{0}(F \delta),u^{+}(F \delta)], \\
							- \Lambda^{-1} \delta^{-1} W'_{F \delta}(u), & \textup{if} \, \, u \in [u^{-}(F \delta),u^{0}(F \delta)] \cup [u^{+}(F \delta),3].
						\end{array} \right.
	\end{equation*}
\end{proof}  


\section{Surface tension with gradient discontinuities at all directions satisfying a rational relation}\label{s.stgaps}

In this section, we prove \tref{generic} concerning generic discontinuities of $D\overline{\sigma}$ and \tref{level set pinning}, which proves that ``bubbling" is a generic feature of the gradient flow.  The basic strategy involves building compact subsolution barriers and the results apply in all dimensions $d \geq 2$.

On the one hand, where the behavior of $D\overline{\sigma}$ is concerned, we avail ourselves of the work of Chambolle, Goldman, and Novaga \cite{GCN}.  They prove that the behavior of the subdifferential of the surface tension $\overline{\sigma}$ closely mirrors the structure of the plane-like minimizers of the energy.  In particular, the key question is whether or not the plane-like minimizers in a given direction foliate space or not.  At least philosophically, it is clear that sliding arguments can be used to show that the existence of barriers is an obstruction to the formation of foliations.  This is precisely the strategy taken in what follows.

At the level of the gradient flow, on the other hand, the maximum principle implies that if a smooth open subset is a strict subsolution of the flow, then any set that initially contains the subsolution continues to do so at later times.  Accordingly, such subsolutions are also relevant for the dynamics.

\subsection{Plane-like minimizers} \label{S: plane like} We need to consider plane-like minimizers which only have locally finite perimeter, it is natural to consider the class of sets which minimize $E_a$ under compact perturbations.  In the literature these are referred to as Class A minimizers and we repeat the definition here for clarity:
\begin{definition}\label{d.classA}
A set $S$ of locally finite perimeter is called a \emph{Class A minimizer} of $E_a$ if $S$ minimizes $E_a(\cdot \,;B_R)$ with respect to compact perturbations in $B_R$ for all $R>1$.  More precisely, given any $R > 0$, if $S' \subset \mathbb{R}^{d}$ is a set of locally finite perimeter and the symmetric difference $S' \Delta S$ satisfies $S' \Delta S \subset \subset B_{R}$, then
	\begin{equation*}
		E_{a}(S; B_{R}) \leq E_{a}(S';B_{R}).
	\end{equation*}
\end{definition}
Following \cite{GCN}, given $n \in S^{d-1}$, we say that an open set of locally finite perimeter $S \subset \mathbb{R}^{d}$ is a \emph{strongly Birkhoff plane-like minimizer in the} $n$ \emph{direction} if (i) $S$ is a Class A minimizer of \eref{surfaceenergy}, (ii) $S$ equals its set of Lebesgue density one points, (iii) there is a $c \in \mathbb{R}$ and an $M > 0$ such that 
	\begin{equation*}
		\{x \in \mathbb{R}^{d} \, \mid \, x \cdot n < c - M\} \subset S \subset \{x \in \mathbb{R}^{d} \, \mid \, x \cdot n < c + M \},
	\end{equation*}
and (iv) $S$ has the strong Birkhoff property, that is,
	\begin{equation*}
		S + k \subset S \quad \textup{if} \, \, k \in \mathbb{Z}^{d} \, \, \textup{and} \, \, k \cdot n \leq 0, \quad S \subset S + k \quad \textup{if} \, \, k \in \mathbb{Z}^{d} \, \, \textup{and} \, \, k \cdot n \geq 0.
	\end{equation*}
We denote the family of all such sets by $\mathcal{M}(n)$.
	
We will need the following properties of $\mathcal{M}(n)$:

	\begin{proposition} \label{P: properties of plane like minimizers}
	\begin{enumerate}[label=(\roman*)]
	\item (\cite[Corollary 4.20]{GCN}) $\mathcal{M}(n)$ is totally ordered, that is, for each $S,S' \in \mathcal{M}(n)$, either $S \subset S'$ or $S' \subset S$.
	\item (\cite[Proposition 3.1]{GCN}) For each $n \in S^{d-1}$ and $K \Subset \mathbb{R}^{d}$, the set of $S \subset \mathcal{M}(n)$ such that $\partial S \cap K \neq \emptyset$ is compact in $L^{1}_{\textup{loc}}(\mathbb{R}^{d})$.  
	\item (\cite[Corollary 1]{MR1190161}) If $a \in C^1(\T^d; [1,\Lambda])$ then any $S \in \mathcal{M}(n)$ is a stationary viscosity solution of the unforced equation
		\begin{equation} \label{e.unforced}
			-a(x) \kappa - D a(x) \cdot n = 0.
		\end{equation}
	\end{enumerate}
	\end{proposition}
	
	Concerning (iii), Caffarelli and Cordoba \cite{MR1190161} show the viscosity solution property just for the perimeter functional, but small modifications of their arguments work for our heterogeneous energy \eref{surfaceenergy} as well.
	
In \cite{GCN}, the authors observe that, for energies of the form \eref{surfaceenergy} (isotropic), a result of Simon \cite{simonresult} implies that the interfaces $\{\partial S \, \mid \, S \in \mathcal{M}(n)\}$ are disjoint so that $\mathcal{M}(n)$ is a lamination.  For more general types of surface energy (anisotropic) it is only known that no intersections can occur at regular points \cite[Proposition 3.4]{GCN}.  Although it is convenient for sliding-type arguments, we will avoid using this fact below so that our arguments apply to other forms of energy as well (cf.\ \rref{otherenergies} below).

\subsection{Gaps}  Before proceeding to the proof of \tref{generic}, we define the notion of a gap and recall the main result of \cite{GCN}.

\begin{definition} \label{D: gap}
We say that a compact set $K \subset \R^d$ with non-empty interior is a \emph{gap at direction $n$} for the medium $a$ if $\partial S \cap K = \emptyset$ for every $S \in \mathcal{M}(n)$.  
\end{definition}

In the next result, we show that the property of having a gap at a direction $n \in S^{d-1}$ is an open condition with respect to uniform norm perturbations of the medium.

\begin{lemma}\label{l.aopen}
If a compact set $K \subset \R^d$ with non-empty interior is a {gap} for the medium $a$ at direction $n$ then there exists $\delta>0$ so that if $b \in C(\T^d ;[1,\Lambda])$ with $\|b - a\|_{C(\mathbb{T}^{d})} < \delta$ then $K$ is a gap for $b$ at direction $n$.
\end{lemma}

In the proof of the lemma, we will need to know that Class A minimizers are well-behaved under perturbation of the coefficient $a$.  More precisely, we will use the following fact, which is proved in \aref{appendixtechnical}.

\begin{lemma}\label{l.aperturbation}    Suppose that $(a_{k})_{k \in \mathbb{N}} \subset C(\mathbb{T}^{d})$ converges uniformly to some positive function $a \in C(\mathbb{T}^{d})$ and $S$ and $(S_{k})_{k \in \mathbb{N}}$ are sets of locally finite perimeter such that $S_{k} \to S$ in $L^{1}_{\text{loc}}(\mathbb{R}^{d})$ as $k \to \infty$.  If $S_{k}$ is a Class A minimizer of $E_{a_{k}}$ for each $k \in \mathbb{N}$,
then $S$ is a Class A minimizer of $E_{a}$.
\end{lemma}

\begin{proof}[Proof of \lref{aopen}]
We will prove the contrapositive, that is, that the set of coefficients for which $K$ is \emph{not} a gap is closed. Suppose that there is a sequence $a_k \to a$ uniformly on $\T^d$ and $S_k \in \mathcal{M}(n,a_k)$ with $\partial S_{k} \cap K \neq \emptyset$.  We claim that there is an $S \in \mathcal{M}(n,a)$ such that $S \cap K \neq \emptyset$.

$S$ can be obtained as a subsequential limit of $(S_{k})_{k \in \mathbb{N}}$.  To see this, recall that standard local perimeter bounds give that the $S_k$ have uniformly bounded perimeter on any compact region. Thus, by standard $BV$ compactness results, we can choose a subsequence, not relabeled, so that
\[ S_k \to_{L^1_{loc}} S \]
and, for any $R>0$,
\[ E_{a}(S; B_R) \leq \liminf_{k \to \infty} E_{a}(S_k; B_R).\] 
By \lref{aperturbation}, $S$ is a Class A minimizer of $E_{a}$.

As for the strong Birkhoff property, for each $k \in \mathbb{N}$,
\begin{equation*}
	S_k + y \supset S_k \quad \textup{if} \, \, \ y \in \mathbb{Z}^{d} \, \, \textup{and} \, \, y \cdot n \geq 0
\end{equation*}
and
\begin{equation*}
S_k + y \subset S_k \quad \textup{if} \, \, y \in \mathbb{Z}^{d} \, \, \textup{and} \, \, \ y \cdot n \leq 0
\end{equation*}
so the same holds for the limit $S$.

It remains to prove that $\partial S \cap K \neq \emptyset$.  By density estimates (e.g.\ \cite[Proposition 3.1]{GCN}), for any $x \in \partial S_k$ and any $r>0$,
\[ |B_r(x) \cap S| \wedge |B_r(x) \setminus S| \geq c r^d\]
for a positive constant $c$ depending only on $d,\Lambda$.  By assumption, we can fix $x_k \in \partial S_k \cap K$ for all $k \in \mathbb{N}$.  Let $x_* \in K $ be any limit point of the sequence $x_k$.  By the $L^1_{\textup{loc}}$ convergence, for any $r>0$,
\[ |S \cap B_r(x_*)| \wedge | B_r(x_*) \setminus S| = \lim_{k \to \infty} |S \cap B_r(x_k)| \wedge | B_r(x_k) \setminus S| \geq cr^d.  \]
Since $r>0$ was arbitrary, we conclude that $x_* \in \partial S$. \end{proof}

Next, we recall the main result of \cite{GCN}, which gives a direct relationship between regularity of the effective surface tension and the existence of gaps in $\mathcal{M}(n)$.

\begin{theorem}[Chambolle, Goldman and Novaga] \label{t.GCN}  Let $n \in S^{d-1}$ and let $V(n)$ be the subspace of $\mathbb{R}^{d}$ spanned by the rational relations satisfied by $n$.  If $\textup{dim}(V(n)) = 0$, we say $n$ is totally irrational.
\begin{itemize}
\item If $n$ is totally irrational, then $D \overline{\sigma}(n)$ exists.
\item The same holds if $\mathcal{M}(n)$ has no gaps.
\item If $n$ is not totally irrational and $\mathcal{M}(n)$ has a gap, then $\partial \overline{\sigma}(n)$ is a convex subset of $V(n)$ of full dimension.
\end{itemize}
\end{theorem}

In \cite[Section 6]{GCN}, the authors give some examples of media where $\overline{\sigma}$ is not differentiable at any direction satisfying a rational relation.  We will show that this phenomenon is generic in the topological sense.

The strategy of proof uses the Euler-Lagrange equation.  The key observation is that if the equation associated to $a$ admits a smooth, bounded open set as a strict subsolution, or the complement of a smooth, bounded open set as a strict subsolution, then these will act as a barrier to foliations.

\begin{lemma} \label{l.barriers}  Given a medium $a \in C(\mathbb{T}^{d}; [1,\Lambda])$, if there is a nonempty $C^{2}$ bounded open set $\Omega \subset \mathbb{R}^{d}$ such that the indicator function $\ind_{\overline{\Omega}}$ is a strict subsolution of \eref{unforced}, then, for each $n \in S^{d-1}$, the family of strongly Birkhoff plane-like minimizers of \eref{surfaceenergy} has a gap.
\end{lemma}

The main result of this section is compact barriers exist generically:

\begin{lemma}\label{l.adense}
For any medium $a \in C(\mathbb{T}^{d}; [1,\Lambda])$ and any $\delta > 0$, there exists a medium $a_\delta \in C^{1}(\mathbb{T}^{d}; [1,\Lambda])$ with $\|a - a_\delta\|_{C(\T^d)} \leq \delta$ and a $C^{2}$ open set $\Omega$, which is bounded and nonempty, such that $\ind_{\overline{\Omega}}$ is a strict subsolution of \eref{unforced}.  

If, in addition, $a \in C^{2}(\mathbb{T}^{d};[1,\Lambda])$ and $p \in [1,\infty)$, then this estimate can be improved to $\|a - a_{\delta}\|_{W^{1,p}(\mathbb{T}^{d})} \leq \delta$.
\end{lemma}

Once Lemmas \ref{l.barriers} and \ref{l.adense} are proved, \tref{generic} follows easily, as we now show.

\begin{proof}[Proof of \tref{generic}]  
Given $n \in S^{d-1}$, let $\mathcal{A}_{n}$ be the family of coefficients $a$ given by
\[ \mathcal{A}_{n} = \{ a \in C^{\infty}(\T^d ;[1,\Lambda]) : \hbox{ there is a gap at direction $n$ for $a$}\} \]
By Lemma \ref{l.aopen}, $\mathcal{A}_{n}$ is open in $C^{\infty}(\mathbb{T}^{d};[1,\Lambda])$ with the $C(\T^d)$ norm topology.  Since the inclusion $W^{1,p}(\mathbb{T}^{d}) \hookrightarrow C(\mathbb{T}^{d})$ is continuous for $p \in (d,\infty)$, $\mathcal{A}_{n}$ is also open in the $W^{1,p}(\mathbb{T}^{d})$ norm topology.  Combining \lref{barriers}, \lref{adense}, and \tref{GCN}, we see that $\cap_{n \in S^{d-1}} \mathcal{A}_{n}$ is dense in either topology. \end{proof}

The remainder of this section is devoted to the proofs of Lemmas \ref{l.barriers} and \ref{l.adense} and \tref{level set pinning}.

\subsection{Gap barriers}  We now show that compact subsolution barriers occur generically.  The proof proceeds by exploiting the structure of the level sets of a generic medium.  We start with a few preliminary reductions.

First of all, we make some room by observing that any function in $C(\mathbb{T}^{d}; [1,\Lambda])$ can be approximated by functions $(a_{n})_{n \in \mathbb{N}}$ in $C^\infty(\mathbb{T}^{d},[1,\Lambda))$ satisfying 
	\begin{equation} \label{E: strict inequality}
	\max_{\mathbb{T}^{d}} a_{n} < \Lambda \quad \textup{for each} \, \, n \in \mathbb{N}.
	\end{equation} 
Therefore, in what follows, we always assume \eqref{E: strict inequality} holds.  

The next lemma shows we can also assume that $a$ attains its maximum at unique, non-degenerate critical points:

	\begin{lemma}  If $a \in C^{2}(\mathbb{T}^{d})$ satisfies \eqref{E: strict inequality} and $\delta > 0$, then there is an $a_{\delta} \in C^{2}(\mathbb{T}^{d})$ satisfying \eqref{E: strict inequality} such that the following holds:
		\begin{itemize}
			\item[(i)] $\|a - a_{\delta}\|_{C^{2}(\mathbb{T}^{d})} \leq \delta$.
			\item[(ii)] There is an $x_{0} \in \mathbb{T}^{d}$ such that 
				\begin{equation*}
					\{x_{0}\} = \left\{x \in \mathbb{T}^{d} \, \mid \, a(x) = \sup_{\T^d} a\right\}, \quad D^{2}a_{\delta}(x_{0}) < 0.
				\end{equation*}
		\end{itemize} 
	\end{lemma}  
	
		\begin{proof}  Let $x_{0} \in \mathbb{T}^{d}$ be a point where $a$ achieves its maximum.  Let $f \in C^{\infty}_{c}(\mathbb{R}^{d})$ be a radially decreasing bump function satisfying
			\begin{equation*}
				f(0) = \max_{\mathbb{R}^{d}} f, \quad D^{2} f(0) < 0, \quad f = 0 \quad \textup{in} \, \, \mathbb{R}^{d} \setminus B_{1/4}(0).
			\end{equation*}
		Then let $\tilde{f} = \sum_{k \in \mathbb{Z}^{d}} f(\cdot -x_0+ k)$ which is periodic on $\R^d$. 
		It is easy to see that $a_{\delta} = a + \delta \tilde{f}$ has the desired properties provided $\delta$ is small enough.  \end{proof}  
		
With these preliminaries out of the way, we are prepared for the proof.  The strategy is as follows: replacing $a$ by $a_{\delta}$ if necessary, we assume that $a$ attains its maximum at a unique, non-degenerate critical point.  This implies that there is $c > 0$ close to $\max a$ such that $\{a = c\}$ is a topologically trivial hypersurface in $\mathbb{T}^{d}$.  

Using a tubular neighborhood of $\partial \Omega = \partial \{a > c\}$, we define a function $\varphi$ such that
	\begin{equation*}
		- a(x)(1+\varphi(x)) \kappa_{\partial \Omega}(x) -  (1+\varphi(x)) Da(x) \cdot n_{\partial \Omega}(x) - c D\varphi(x) \cdot n_{\partial \Omega}(x) > 0 \quad \textup{if} \, \, x \in \partial \Omega.
	\end{equation*}
It follows that the set $\Omega = \{a > c\}$ is a strict subsolution associated to the coefficient $a_{\varphi} = (1 + \varphi) \cdot a$.   The complement of $\Omega$ will, correspondingly, be a strict supersolution.

	\begin{proof}[Proof of Lemma \ref{l.adense}]  By the previous considerations, we can assume that $a \in C^{2}(\mathbb{T}^{d})$ satisfies \eqref{E: strict inequality} and attains its maximum at a unique, non-degenerate critical point $x_{0}$.  Fix $\ep > 0$ and $p \in (d,\infty)$.  We will find a function $a_{\ep} \in C^{2}(\mathbb{T}^{d})$ satisfying \eqref{E: strict inequality} such that $a_{\ep}$ satisfies the conclusions of the theorem and $\|a_{\ep} - a\|_{W^{1,p}(\mathbb{T}^{d})} < (2 + \|Da\|_{L^{\infty}(\mathbb{T}^{d})}) \ep$.  Notice that this is enough to obtain an estimate in $C(\mathbb{T}^{d})$ by Morrey's inequality.
	
	To start with, notice that if $c$ is close enough to $a(x_0)$, then $\{a > c\}$ is an open, simply connected subset of $\mathbb{T}^{d}$ with $C^{2}$ boundary. Let $\Omega= \{a > c\}$. 
	
	Fix $r > 0$ such that the signed distance $d$ to $\partial \Omega = \{a = c\}$, positive in $\Omega$ and negative outside, is smooth in an $r$-neighborhood of the surface.  Letting $\nu \in (0,r)$ be a small constant to be determined, choose a smooth function $\eta : (-\nu,\nu) \to [-\ep/2\Lambda, \ep/2\Lambda]$ such that 
		\begin{gather*}
			2 \|\kappa_{\partial \Omega}\|_{L^{\infty}(\partial \Omega)} \geq \|\eta'\|_{L^{\infty}([-\nu,\nu])} \geq \eta'(0) > \|\kappa_{\partial \Omega}\|_{L^{\infty}(\partial \Omega)} , \quad \eta(0) = 0, \\
			\eta = 0 \quad \text{in a neighborhood of} \quad \{-\nu,\nu\}.
		\end{gather*}
Define $\varphi : \mathbb{T}^{d} \to [-\ep/2\Lambda,\ep/2\Lambda]$ by 
		\begin{equation*}
			\varphi(x) = \left\{ \begin{array}{r l}
							\eta(d(x)), & \textup{if} \, \, |d(x)| < \nu \\
								0, & \textup{otherwise}
						\end{array} \right.
		\end{equation*}
	This is a $C^{2}$ function by the choice of $\eta$.  		
	
	Let $a_{\varphi} = (1+ \varphi) \cdot a$.  Notice that $\|a_{\varphi} - a\|_{L^{\infty}(\mathbb{T}^{d})} \leq \ep$ and, by the coarea formula,
		\begin{equation*}
			\int_{\mathbb{T}^{d}} \|D\varphi(x)\|^{p} \, dx = \int_{\{|d| < \nu\}} |\eta'(d(x))|^{p} \, dx = \int_{-\nu}^{\nu} |\eta'(s)|^{p} \, ds \leq 2^{p + 1} \|\kappa_{\partial \Omega}\|_{L^{\infty}(\partial \Omega)}^{p} \nu.
		\end{equation*}  
	Thus, if $\nu$ is sufficiently small, we obtain
		\begin{equation*}
			\|a_{\varphi} - a\|_{W^{1,p}(\mathbb{T}^{d})} \leq (1 + \|Da\|_{L^{\infty}(\mathbb{T}^{d})})\ep + \Lambda 2^{1 + p^{-1}} \|\kappa_{\partial \Omega}\|_{L^{\infty}(\partial \Omega)} \nu^{\frac{1}{p}} < (2 + \|Da\|_{L^{\infty}(\mathbb{T}^{d})}) \ep.
		\end{equation*}
	
	Finally, we claim that $\Omega$ has the desired properties for the medium $a_\varphi$.  To see this, start by noting that $Da$ and $D\varphi$ are aligned with the outward normals to $\Omega$ along $\partial \Omega$, i.e., for $x \in \partial \Omega$,
	\[ Da(x) \cdot n_{\partial \Omega}(x) = -|Da(x)| \ \hbox{ and } \ D\varphi(x) \cdot  n_{\partial \Omega}(x) = -\eta'(0).\]
	Accordingly, for each $x \in \partial \Omega$, we have
		\begin{align*}
			- a_{\varphi}(x) \kappa_{\partial \Omega} (x) - Da_{\varphi}(x) \cdot n_{\partial \Omega}(x) &= -c \kappa_{\partial \Omega}(x)  +   |Da(x)|+ c \eta'(0) \\
			&\geq c(\eta'(0)-\|\kappa_{\partial \Omega}\|_{L^{\infty}(\partial \Omega)} )\\
			&> 0.
		\end{align*}
	Thus, the indicator function $\ind_{\overline{\Omega}}$ is a strict subsolution of \eref{unforced}.    \end{proof}
	
\begin{remark}\label{r.otherenergies}  The approach above provides a general strategy for showing that the plane-like minimizers of a given surface energy has gaps, even when the energy does not have the form \eref{surfaceenergy}.  For example, given a $\psi \in C^{\infty}(\mathbb{T}^{d}; \mathbb{R}^{d})$ such that $\|\psi\|_{L^{\infty}(\mathbb{T}^{d})} < 1$, consider the energy given by 
	\begin{equation} \label{E: finsler energy}
		\int_{\partial E} \left(1 + \psi(x) \cdot n(x) \right) \, \mathcal{H}^{d-1}(d x).
	\end{equation}
By the divergence theorem, the Euler-Lagrange equation associated with this energy is
	\begin{equation*}
		\kappa + \textup{div} \, \psi = 0
	\end{equation*}
Up to making a small perturbation, we can assume that $\textup{div} \, \psi \not \equiv 0$.  Hence there is a ball $B$ such that $\textup{div} \, \psi < 0$ in $B$, and then we can find a smooth perturbation $\tilde{\psi}$, which is arbitrarily close to $\psi$ in $C(\mathbb{T}^{d})$, such that 
	\begin{equation*}
		\kappa_{\partial B} + \textup{div} \, \tilde{\psi} < 0 \quad \textup{in} \, \, \partial B.
	\end{equation*}
Thus, $B$ is a smooth compact subsolution and we deduce that a small perturbation of \eqref{E: finsler energy} has gaps in every direction.  In particular, by \cite{GCN}, typically, the associated surface tension is non-differentiable at every lattice direction.
\end{remark}  

\subsection{Existence of Gaps}  Once a smooth compact barrier is known to exist, no plane-like minimizer can touch it if the subsolution property is strict.  By the monotonicity of the family of strongly Birkhoff plane-like minimizers, this means the barrier has to be contained in a gap.  As we will see below, proving this is somewhat technical compared to the diffuse interface case --- the basic issue being that, for the sake of generality, we will not use the fact that $\{\partial E \, \mid \, E \in \mathcal{M}(n)\}$ is pairwise disjoint.

We will need the following lemma:

\begin{lemma}\label{l.nogapconclusion}
If $\mathcal{M}(n)$ does not have a gap then, for every $S \in \mathcal{M}(n)$,
\[ S = \bigcup \{S' \in \mathcal{M}(n): S' \subsetneq S \} = \bigcap \{S' \in \mathcal{M}(n): S' \supsetneq S \} \quad \textup{Lebesgue a.e.}\]
\end{lemma}
\begin{proof}
The arguments are symmetric so we just do the intersection case.  By \cite[Lemma 6.3]{CdlL2001} and \cite[Proposition 3.1]{GCN}, if we define $S^{**} \subset \mathbb{R}^{d}$ by
\[ S^{**} = \bigcap \{S' \in \mathcal{M}(n): S' \supsetneq S \}, \]
then the set $S^{*}$ of Lebesgue density one points of $S^{**}$ is in $\mathcal{M}(n)$.  Note that, by the ordering property of $\mathcal{M}(n)$ and density estimates (i.e.\ \cite[Proposition 3.1]{GCN}), the inclusion $S \subset S^{*} \subset S^{**}$ holds.  

Suppose that $x_0 \in S^* \setminus S $.  Applying density estimates again, we can find a ball $B \subset S^{*} \setminus S$ close to $x_{0}$.  If $\tilde{S} \in \mathcal{M}(n)$ and $\partial \tilde{S} \cap B \neq \emptyset$, then $\tilde{S}$ must be a strict subset of $S^{*}$ and a strict superset of $S$.  In particular, $S \subsetneq \tilde{S} \subsetneq S^{*} \subset S^{**}$, in violation of the definition of $S^{**}$.  Hence $B$ is a gap according to Definition \ref{D: gap}, contradicting the hypothesis.
\end{proof}

Now we show how to use the lemma in a sliding argument.

\begin{proof}[Proof of \lref{barriers}]  We argue by contradiction.  Fix $n \in S^{d-1}$ and let $\mathcal{M}(n)$ denote the family of strongly Birkhoff plane-like minimizers in the $n$ direction.  Assume $\mathcal{M}(n)$ has no gaps. 

Let $\Omega$ be the bounded strict subsolution which was assumed to exist in the statement.  Define
\[ S^{**} =  \bigcap \left\{ S \in \mathcal{M}(n) \, \mid \,\Omega \subset S \right\}\]
and let $S^{*}$ be the set of Lebesgue density one points of $S^{**}$.  By \cite[Lemma 6.3]{CdlL2001} and \cite[Proposition 3.1]{GCN}, $S^{*} \in \mathcal{M}(n)$.  Furthermore, since $\Omega$ is open, $\Omega \subset S^{*}$ necessarily holds.  

We claim that, due to the no gap assumption, we must have $\partial S^* \cap \partial \Omega \neq \emptyset$.  Since $\mathcal{M}(n)$ has no gaps, by \lref{nogapconclusion},
\[S^* = \bigcup \{S \in \mathcal{M}(n): S \subset S^*, \ S \neq S^* \}. \]
If $\Omega \setminus S \neq \emptyset$ for all $S \subsetneq S^*$ then, by compactness (cf.\ \cite[Proposition 3.2]{GCN}), there is an $x \in \overline{\Omega} \cap \partial S^* = \partial S^{*} \cap \partial \Omega$.  Otherwise $\Omega \subset S$ for some $S \subsetneq S^*$, which contradicts the definition of $S^*$.  Thus, henceforth we can fix $x_{0} \in \partial S^{*} \cap \partial \Omega$.

	     Let $d_{\Omega}$ be the signed distance function to $\Omega$ with $\{d_{\Omega} > 0\} = \Omega$.  Since $\ind_{\overline{\Omega}}$ is a strict subsolution of \eref{unforced}, it follows that there is an $r' > 0$ such that $d_{\Omega}$ is smooth in $\{|d_{\Omega}| < r'\}$ and 
	     	\begin{equation} \label{E: strict subsolution}
			- a(x) \Delta d_{\Omega} - Da(x) \cdot Dd_{\Omega} < 0 \quad \textup{in} \, \, B_{r'}(x_{0}).
		\end{equation}
		
	It is straightforward to check that there is an $r > 0$ such that $\ind_{S^{*}} - d_{\Omega}$ achieves its minimum in $B_{r}(x_{0})$ at $x_{0}$.  Since $S^{*}$ is a plane-like minimizer, $\ind_{S^{*}}$ is a viscosity supersolution of \eref{unforced} by Proposition \ref{P: properties of plane like minimizers}.  Thus, the following inequality holds: 
	     	\begin{equation*}
			0 \leq - a(x_{0}) \Delta d_{\Omega}(x_{0}) - Da(x_{0}) \cdot Dd_{\Omega}(x_{0}).
		\end{equation*}
	However, this directly contradicts \eqref{E: strict subsolution}.  
	
	We conclude that $\mathcal{M}(n)$ has gaps, no matter the choice of $n \in S^{d-1}$. \end{proof}

\subsection{Proof of \cref{level set bubbling}}  The previous arguments show that the existence of a smooth, compact strict subsolution of \eref{sieqn} forces the surface tension $\bar{\sigma}$ to have corners.  It also has consequences for the gradient flow, as we now show.  While we do not know if it implies pinning in the strongest sense (i.e. pinning of the entire interface as considered in the example of \sref{mobility}) it does seem to rule out the possibility of homogenization in the usual way by pinning some compact connected components of the negative-phase.  This phenomenon has been observed many times in the study of interface homogenization, see, for example, Cardaliaguet, Lions and Souganidis \cite{CLS2009}. 

\begin{proof}[Proof of \cref{level set bubbling}]  Suppose that $x_{0} \in \mathbb{R}^{d}$, $c \in \mathbb{R}$, and $u_{0}(x_{0}) > c$.  We will show that $\bar{u}^{*}(x_{0},t) \geq c$ for all $t > 0$.  
	
	Let $\Omega \subset \mathbb{R}^{d}$ be a $C^{2}$ bounded open set such that $\ind_{\overline{\Omega}}$ is a strict stationary subsolution of \eref{unforced}.  Given $\ep > 0$, choose a $k_{\ep} \in \mathbb{Z}^{d}$ such that $\|\ep^{-1} x_{0} - (x + k_{\ep}) \| \leq 1$.  Define $\ind^{\ep}$ by 
		\begin{equation*}
			\ind^{\ep}(y) = \left\{ \begin{array}{r l}
									c, & \textup{if} \, \, \ep^{-1} y \in k_{\ep} + \overline{\Omega} \\
									-\infty, & \textup{otherwise} 
								\end{array} \right.
		\end{equation*}
	This is now a time-stationary subsolution of the $\ep$ scaled mean curvature flow \eref{levelsetepeqn}.  Since $u_{0}$ is continuous, we know that $\ep(k_{\ep} + \overline{\Omega}) \Subset \{u_{0} > c\}$ if $\ep > 0$ is small enough.  Thus, since $\ind^{\ep}$ is a stationary subsolution, it follows that $u^{\ep}(\cdot,t) \geq \ind^{\ep}$ for each $t \geq 0$.  From this, we find that $\bar{u}^{*}(x_{0},t) \geq c$.  
	
	The statement for $\bar{u}_{*}$ follows the same way using the complement compact supersolution $\R^d \setminus \Omega$.
	  \end{proof}

	  \section{Gaps in the plane-like minimizer lamination are generic: diffuse interface case}\label{s.stgaps-diffuse}

In this section, we prove results on the existence of gaps and weak pinning analogous to those of the previous one.  Once again, we proceed by perturbing around the sharp interface $\delta = 0$ setting.  The existence of a compact strict subsolution for the sharp interface model will imply the same for the diffuse interface model when $\delta>0$ is small.  By a sliding argument the existence of such barriers causes a gap in the family of strong Birkhoff plane-like minimizers just as in the sharp interface case.

We stop short of proving any results concerning the gradient of the diffuse surface tension $\bar{\sigma}_{AC}$.  The reason is there is currently no proven analogue of the result of \cite{GCN} in the diffuse interface case.  We believe that such an analogue does hold and leave its proof to future work.

As in the previous section, the results presented here apply in all dimensions $d \geq 2$.

\subsection{Plane-like minimizers and gaps}  Let us introduce some notation and terminology to be used in what follows.  To begin with, as in the sharp interface case, we recall the definition of a Class A minimizer of the diffuse interface energy $\mathcal{AC}^{\delta}_{\theta}$ (see \eref{diffuseenergy}).

\begin{definition} \label{D: diffuse interface class a} A function $u : \mathbb{R}^{d} \to [-1,1]$ is said to be a \emph{Class A minimizer} of the energy functional $\mathcal{AC}^{\delta}_{\theta}$ if, for any $R > 0$ and any $v \in u + H_{0}^{1}(B_{R})$,
	\begin{equation*}
		\mathcal{AC}^{\delta}_{\theta}(u; B_{R}) \leq \mathcal{AC}^{\delta}_{\theta}(v; B_{R}).
	\end{equation*}
\end{definition}

Given $\theta \in C(\mathbb{T}^{d}; [1,\Lambda^{2}])$, $\delta > 0$, and $n \in S^{d-1}$, we say that a Class A minimizer $u : \mathbb{R}^{d} \to (-1,1)$ of $\mathcal{AC}^{\delta}_{\theta}$ is a \emph{strongly Birkhoff plane-like minimizer in the direction $n$} if, for each $k \in \mathbb{Z}^{d}$,
	\begin{equation*}
		u(x - k) \geq u(x) \quad \textup{if} \, \, k \cdot n \geq 0, \quad u(x - k) \leq u(x) \quad \textup{if} \, \, k \cdot n \leq 0.
	\end{equation*}
Notice that $\lim_{x \cdot n \to \pm \infty} u(x) = \mp 1$ automatically holds since the only periodic Class A minimizers are the constants $1$ and $-1$.  

We let $\mathcal{M}^{\delta}_{\theta}(n)$ denote the family of all strongly Birkhoff plane-like minimizers.

Arguing as in \cite{Bangert2} (cf.\ \cite{morsetypeuniqueness}), one can prove that $\mathcal{M}^{\delta}_{\theta}(n)$ forms a lamination.  That is, for each $u_{1},u_{2} \in \mathcal{M}^{\delta}_{\theta}(n)$, 
	\begin{equation*}
		\textup{either} \quad u_{1} < u_{2} \quad \textup{in} \, \, \mathbb{R}^{d}, \quad u_{1} > u_{2} \quad \textup{in} \, \, \mathbb{R}^{d}, \quad \textup{or} \quad u_{1} = u_{2} \quad \textup{in} \, \, \mathbb{R}^{d}.
	\end{equation*}
(For the connection between Moser-Bangert theory and \eqref{e.diffuseenergy}, which allows us to invoke results from \cite{Bangert2}, see \cite{JGV09} and the introduction of \cite{RSbook}.)

We will say that $\mathcal{M}^{\delta}_{\theta}(n)$ \emph{has a gap} if the graphs of its elements fail to foliate $\mathbb{R}^{d} \times (-1,1)$.  That is, $\mathcal{M}^{\delta}_{\theta}(n)$ has a gap if 
	\begin{equation*}
		\{(x,u(x)) \, \mid \, x \in \mathbb{R}^{d}, \, \, u \in \mathcal{M}^{\delta}_{\theta}(n) \} \neq \mathbb{R}^{d} \times (-1,1).
	\end{equation*}

\subsection{Parametrizations of $\mathcal{M}^{\delta}_{\theta}(n)$}  As in the sharp interface case, it will be convenient to know that $\mathcal{M}^{\delta}_{\theta}(n)$ has no gaps if and only if it admits a continuous parametrization.

\begin{proposition} \label{P: continuous parametrization}  If $\mathcal{M}^{\delta}_{\theta}(n)$ does not have a gap, then there is a bijection $\gamma \mapsto u(\cdot \,; \gamma)$ from $\mathbb{R}$ onto $\mathcal{M}^{\delta}_{\theta}(n)$ such that:
	\begin{itemize}
		\item[(i)] $\gamma \mapsto u(\cdot \,; \gamma)$ is continuous with respect to the topology of local uniform convergence,
		\item[(ii)] If $\gamma_{1} < \gamma_{2}$, then $u(x; \gamma_{1}) < u(x; \gamma_{2})$ for each $x \in \mathbb{R}^{d}$.
	\end{itemize}  \end{proposition}    
	
In fact, $\mathcal{M}^{\delta}_{\theta}(n)$ has no gaps if and only if, in the terminology of \cite{variationaleinstein}, there is a pulsating standing wave $U_{n} \in UC(\mathbb{R} \times \mathbb{T}^{d})$ in the direction $n$ (cf.\ Remark \ref{R: pulsating standing waves} below).  To keep things short, we will not prove this stronger statement here.

\begin{proof}  Given $\gamma > 0$, let $u(\cdot; \gamma)$ be the unique element of $\mathcal{M}^{\delta}_{\theta}(n)$ such that 
	\begin{equation*}
		u(0; \gamma) = -\tanh(\gamma).
	\end{equation*}
The existence of such an element follows from the assumption that $\mathcal{M}^{\delta}_{\theta}(n)$ has no gaps; uniqueness follows from the fact that it forms a lamination.

The monotonicity of $\gamma \mapsto u(\cdot \, ; \gamma)$ also follows from the lamination property.  It remains to check the continuity.  

Suppose that $\bar{\gamma} \in \mathbb{R}$, $(\gamma_{k})_{k \in \mathbb{N}} \subset \mathbb{R}$, and $\lim_{k \to \infty} \gamma_{k} = \bar{\gamma}$.  Elliptic estimates imply that $(u(\cdot \, ; \gamma_{k}))_{k \in \mathbb{N}}$ is compact in the topology of local uniform convergence.  Further, it is not hard to show that any subsequential limit is itself in $\mathcal{M}^{\delta}_{\theta}(n)$.  Thus, given a subsequence $(k_{j})_{j \in \mathbb{N}} \subset \mathbb{N}$, there is a further subsequence $(k_{j_{\ell}})_{\ell \in \mathbb{N}}$ and a $\tilde{u} \in \mathcal{M}^{\delta}_{\theta}(n)$ such that $u(\cdot; \gamma_{k_{j_{\ell}}}) \to \tilde{u}$ locally uniformly.  In particular, $\tilde{u}(0) = -\tanh(\bar{\gamma})$ so $\tilde{u} = u(\cdot \, ; \bar{\gamma})$.  Since $(k_{j})_{j \in \mathbb{N}}$ was arbitrary, we are left to conclude that $u(\cdot \, ; \bar{\gamma}) = \lim_{k \to \infty} u(\cdot \,; \gamma_{k})$ as desired.  \end{proof}

\subsection{Obstruction}  We saw above that if there are no gaps, we can continuously parametrize $\mathcal{M}^{\delta}_{\theta}(n)$.  Hence, in that case, classical sliding techniques can be used to rule out the existence of certain (sub- or super-) solutions of \eref{aceqn1}.  In particular, bump (strict) subsolutions cannot occur:

\begin{proposition} \label{p.diffuse gaps}  If there is an upper semi-continuous function $u_{\delta} : \mathbb{R}^{d} \to [-2,1]$ and an $F > 0$ such that 
	\begin{equation*}
		- \delta \Delta u_{\delta} + \delta^{-1} \theta(x) W'(u_{\delta}) \leq - F \quad \textup{in} \, \, \mathbb{R}^{d},
	\end{equation*}
$\{u_{\delta} \geq -1\}$ is compact, $\{u_{\delta} > -1\}$ is non-empty, and $u_{\delta}$ is smooth in a neighborhood of $\{u_{\delta} \geq -1\}$, then, for each $n \in S^{d-1}$, $\mathcal{M}^{\delta}_{\theta}(n)$ has gaps.  \end{proposition}

\begin{proof}  To start with, observe that there is a constant $c \in (0,1)$ such that $u_{\delta} \leq 1 - c$ in $\mathbb{R}^{d}$.  Indeed, were this not the case, then, by the compactness of $\{u_{\delta} \geq - 1\}$, we could find an $x_{0} \in \mathbb{R}^{d}$ such that $u_{\delta}(x_{0}) = 1 = \max_{\mathbb{R}^{d}} u_{\delta}$, but this would contradict the strict subsolution property.

We argue by contradiction.  If $\mathcal{M}^{\delta}_{\theta}(n)$ has no gaps, then Proposition \ref{P: continuous parametrization} implies that there is a continuous, increasing parametrization $\gamma \mapsto u(\cdot \,;\gamma)$ of $\mathcal{M}^{\delta}_{\theta}(n)$.  Define $T \in \mathbb{R}$ by 
	\begin{equation*}
		T = \inf \left\{ \gamma \in \mathbb{R} \, \mid \, u(\cdot \,;\gamma) \geq u_{\delta} \, \, \textup{in} \, \, \mathbb{R}^{d} \right\}.
	\end{equation*}
Since $u_{\delta} \leq 1 - c$, $\{u^{\delta} \geq -1\}$ is compact, and $\{u_{\delta} > -1\}$ is non-empty, it follows that $T < \infty$.  

We claim that $u(\cdot \,; T)$ touches $u_{\delta}$ above at some point $\bar{x} \in \mathbb{R}^{d}$.  Indeed, this follows from the fact that the parametrization is continuous and increasing.  Note that $u_{\delta}(\bar{x}) = u(\bar{x};T) > - 1$.  Since $u_{\delta}$ is smooth in a neighborhood of $\{u_{\delta} \geq - 1\}$, the viscosity solution property of $u(\cdot; T)$ yields
	\begin{equation*}
		0 \leq - \delta \Delta u_{\delta}(\bar{x}) + \delta^{-1}\theta(\bar{x}) W'(u_{\delta}(\bar{x}))
	\end{equation*}
This contradicts the strict subsolution property of $u_{\delta}$.  \end{proof}  

\subsection{Dynamics}  As in the sharp interface case, the previous construction also has a dynamical interpretation.  (Below we once again use the half-relaxed limit notation from Definition \ref{D: half-relaxed}.)

\begin{proposition} \label{p.diffuse bubbles}  If for some $\delta > 0$ there is a smooth function $u_{\delta} : \mathbb{R}^{d} \to [-2,1]$ and an $F > 0$ satisfying the hypotheses of \pref{diffuse gaps} and such that $\{u_{\delta} > 0\}$ is non-empty, then, for each $u_{0} \in UC(\mathbb{R}^{d}; [-3,3])$, if $(u^{\ep})_{\ep > 0}$ are the solutions of the Cauchy problem
	\begin{equation*}
		\left\{ \begin{array}{r l}
				\delta(u^{\ep}_{t} - \Delta u^{\ep}) + \ep^{-2} \delta^{-1} \theta(\tfrac{x}{\ep})W'(u^{\ep}) = 0 & \textup{in} \, \, \mathbb{R}^{d} \times (0,\infty), \\
				u^{\ep} = u_{0} & \textup{on} \, \, \mathbb{R}^{d} \times \{0\},
			\end{array} \right.
	\end{equation*}
then 
	\begin{equation*}
		\limsup \nolimits^{*} u^{\ep} = 1 \quad \textup{in} \, \, \{u_{0} > 0\}.
	\end{equation*}
\end{proposition}  

Similarly if there is a smooth $v_{\delta} :  \mathbb{R}^{d} \to [-1,2]$ and an $F > 0$ such that 
	\begin{equation*}
		- \delta \Delta v_{\delta} + \delta^{-1} \theta(x) W'(v_{\delta}) \geq F \quad \textup{in} \, \, \mathbb{R}^{d}
	\end{equation*}
and $\{v_{\delta} \leq 1\}$ is compact and $\{v_{\delta} <0\}$ is non-empty, then there is a symmetrical conclusion for the $\ep$ scaled problem above:
\begin{equation*}
		\liminf \nolimits_{*} u^{\ep} = -1 \quad \textup{in} \, \, \{u_{0} < 0\}.
	\end{equation*}

	\begin{proof}  Note, as in the proof of \pref{diffuse gaps}, that $\max_{\mathbb{R}^{d}} u_{\delta} \leq 1 - c$ for some $c \in (0,1)$.

	Since $u_{\delta} \leq 1 - c$ in $\mathbb{R}^{d}$, $\{u_{\delta} > -1\}$ is compact, and $\{u_{\delta} > 0\}$ is non-empty, we conclude the proof by combining ideas from the proofs of \cref{level set bubbling} and \tref{diffuse interface pinning} (especially \pref{initialization}).  \end{proof}  

\subsection{Proof of Theorem \ref{t.genericdiffuse}}

 In what follows, we let $\mathcal{A}_{\Lambda}$ denote the family of all $a \in C^{1}(\mathbb{T}^{d}; [1,\Lambda])$ such that there is a smooth, bounded open set $\Omega \subset \mathbb{R}^{d}$ such that $\ind_{\overline{\Omega}}$ is a strict subsolution of \eref{sieqn}.  By \lref{adense}, $\mathcal{A}_{\Lambda}$ is a dense subset of $C(\mathbb{T}^{d}; [1,\Lambda])$ and $W^{1,p}(\mathbb{T}^{d};[1,\Lambda])$ for each $p \in (d,\infty)$.

	\begin{proof}[Proof of Theorem \ref{t.genericdiffuse}]  Define $\Theta_{\Lambda} \subset C^{1}(\mathbb{T}^{d}; [1,\Lambda^{2}])$ by 
		\begin{equation*}
			\Theta_{\Lambda} = \left\{\theta \in C^{1}(\mathbb{T}^{d}; [1,\Lambda^{2}]) \, \mid \, \sqrt{\theta} \in \mathcal{A}_{\Lambda} \right\}.
		\end{equation*}
	Notice that the map $\theta \mapsto \sqrt{\theta}$ is a homeomorphism sending $C(\mathbb{T}^{d}; [1,\Lambda^{2}])$ onto $C(\mathbb{T}^{d}; [1,\Lambda])$ and $W^{1,p}(\mathbb{T}^{d};[1,\Lambda])$ onto $W^{1,p}(\mathbb{T}^{d};[1,\Lambda])$.  Thus, by the density of $\mathcal{A}_{\Lambda}$, we know that $\Theta_{\Lambda}$ is dense in both spaces.
	
	If $\theta \in \Theta_{\Lambda}$, then there is a smooth, bounded open set $\Omega$ such that $\ind_{\overline{\Omega}}$ is a strict subsolution and $\ind_{\R^d \setminus \Omega}$ is a strict supersolution of \eref{mcfa}.  Arguing exactly as in \sref{diffusepinning}, this implies we can find an $F_{\theta} > 0$, a $\delta_{\theta} \in (0,1)$, and continuous functions $(u_{\delta}^{(\theta)})_{\delta \in (0,\delta_{\theta})}, (v_{\delta}^{(\theta)})_{\delta \in (0,\delta_{\theta})} \subset UC(\mathbb{R}^{d}; [-2,1])$ satisfying
		\begin{equation*}
			- \delta \Delta u_{\delta}^{(\theta)} + \delta^{-1} \theta(x) W'(u_{\delta}^{(\theta)}) \leq -F_{\theta} \quad \textup{in} \, \, \mathbb{R}^{d}
		\end{equation*} 
		and
		\[ - \delta \Delta v_{\delta}^{(\theta)} + \delta^{-1} \theta(x) W'(v_{\delta}^{(\theta)}) \geq F_{\theta} \quad \textup{in} \, \, \mathbb{R}^{d} \]
	for which the sets $\{\{u_{\delta}^{(\theta)} \geq -1\},\{v_{\delta}^{(\theta)} \leq 1\}\}_{\delta \in (0,\delta_{\theta})}$ are all compact, the sets $\{\{u_{\delta}^{(\theta)} > 0\},\{v_{\delta}^{(\theta)} < 0\}\}_{\delta \in (0,\delta_{\theta})}$ are all non-empty, and $u_{\delta}^{(\theta)}$ and $v_{\delta}^{(\theta)}$ are smooth in $\{u_{\delta}^{(\theta)} > -(1 + \beta \delta)\}$ and $\{v_{\delta}^{(\theta)} < 1 + \beta \delta\}$, respectively.  (The construction shows that it is possible to make $u_{\delta}^{(\theta)}$ and $v_{\delta}^{(\theta)}$ smooth away from these extreme values since $\Omega$ is smooth in this setting.)
	
	Now define open sets $\{\mathcal{G}_{n}\}_{n \in \mathbb{N}} \subset C(\mathbb{T}^{d}; [1,\Lambda^{2}])$ by
		\begin{align*}
			\mathcal{G}_{n} &= \bigcup_{\theta \in \Theta_{\Lambda}} \mathcal{G}_{n}(\theta), \\
			\mathcal{G}_{n}(\theta) &:= \left\{ \tilde{\theta} \in C(\mathbb{T}^{d}; [1,\Lambda^{2}]) \, \mid \, 2 n \|W'\|_{L^{\infty}([-3,3])} \|\tilde{\theta} - \theta\|_{L^{\infty}(\mathbb{T}^{d})} < F_{\theta} \delta_{\theta} \right\}.
		\end{align*}
	Observe that $\mathcal{G}_{n}$ is dense in $C(\mathbb{T}^{d}; [1,\Lambda^{2}])$ since $\Theta_{\Lambda}$ is.  
	
	Next, notice that if $\tilde{\theta} \in \mathcal{G}_{n}$ for some $n \in \mathbb{N}$, then there is a $\theta \in \Theta_{\Lambda}$ such that $\tilde{\theta} \in \mathcal{G}_{n}(\theta)$.  In particular, for each $\delta \in [\frac{\delta_{\theta}}{n},\delta_{\theta})$,
		\begin{equation*}
			- \delta \Delta u_{\delta}^{(\theta)} + \delta^{-1} \tilde{\theta}(x) W'(u_{\delta}^{(\theta)}) \leq -\frac{F_{\theta}}{2} \quad \textup{in} \, \, \mathbb{R}^{d}.
		\end{equation*}
		and
		\[ - \delta \Delta v_{\delta}^{(\theta)} + \delta^{-1} \tilde{\theta}(x) W'(v_{\delta}^{(\theta)}) \geq \frac{F_{\theta}}{2} \quad \textup{in} \, \, \mathbb{R}^{d}. \]
	Accordingly, for such choices of $\delta$, \pref{diffuse gaps} and \pref{diffuse bubbles} (in both subsolution and supersolution form) apply to $\tilde{\theta}$.
	
	Let $\mathcal{G} = \cap_{n \in \mathbb{N}} \mathcal{G}_{n}$.  This is dense in $C(\mathbb{T}^{d}; [1,\Lambda^{2}])$ since $\mathcal{G} \supset \Theta_{\Lambda}$.  If $\theta \in \mathcal{G}$, then there is a sequence $(\theta^{(n)})_{n \in \mathbb{N}}$ such that $\theta \in \mathcal{G}_{n}(\theta^{(n)})$ for each $n$.   Hence ${\theta}$ satisfies the conclusions of the theorem with $I(\theta)$ given by
		\begin{equation*}
			I(\theta) = \bigcup_{n = 1}^{\infty} \left(\frac{\delta_{\theta^{(n)}}}{n}, \delta_{\theta^{(n)}} \right).
		\end{equation*}  
	Since $\sup_{n} \delta_{\theta^{(n)}} < 1$ by construction, we know that $0 \in \overline{I(\theta)}$.  
	
	Notice that if $\theta \in \Theta_{\Lambda}$, then we can take $\theta^{(n)} = \theta$ for all $n$ above.  Thus, $I(\theta) = (0,\delta_{\theta})$ in this case.  Since $\Theta_{\Lambda}$ is dense, this proves the penultimate assertion of the theorem.
	
	Finally, we observe that the same considerations apply to $W^{1,p}(\mathbb{T}^{d};[1,\Lambda])$ since, for each $n$, the set $\mathcal{G}_{n}(\theta) \cap W^{1,p}(\mathbb{T}^{d};[1,\Lambda])$ is open and $\Theta_{\Lambda}$ remains dense in this topology.

	\end{proof}  

\begin{remark} \label{R: other energies generic}  Theorem \ref{t.genericdiffuse} remains true if \eqref{e.diffuseenergy} is replaced by the variants \eqref{E: gradient model} or \eqref{E: weight model}.  As in Remark \ref{R: other energies construction}, smooth diffuse interface subsolutions can be constructed from the sharp interface subsolutions of Lemma \ref{l.adense}.  The only difference in the proof is that since $\theta$ appears multiplied by derivatives of $u_{\delta}$ in the PDE, we need to change the definition of $\mathcal{G}_{n}(\theta)$ accordingly.  This is not a problem since the construction of Section \ref{s.diffusepinning} implies that $u_{\delta}$ has bounded second order derivatives in the set $\{u_{\delta} > -(1 + \beta \delta)\}$. \end{remark}  

\begin{remark} \label{R: pulsating standing waves} Theorem \ref{t.genericdiffuse} provides examples of diffuse interface models in periodic media in which, in every direction $n \in S^{d-1}$, there is no continuous pulsating standing wave.  See \cite{variationaleinstein} for a discussion of the relevance of pulsating standing waves to the analysis of the energy \eqref{e.diffuseenergy} and the homogenization of its gradient flow.  

Given an $n \in S^{d-1}$, a pulsating standing wave of \eqref{e.diffuseenergy} is a function $U_{n} \in L^{\infty}(\mathbb{R} \times \mathbb{T}^{d})$ that is a distributional solution of the PDE
	\begin{equation*}
		\left\{ \begin{array}{c}
			(n \partial_{s} + D_{y})^{*} (n \partial_{s} + D_{y}) U_{n} + a(x) W'(U_{n}) = 0 \quad \text{in} \, \, \mathbb{R} \times \mathbb{T}^{d}, \\
			\lim_{s \to \pm \infty} U_{n}(s,y) = \mp 1, \, \, \|U_{n}\|_{L^{\infty}(\mathbb{R} \times \mathbb{T}^{d})} \leq 1, \, \, \partial_{s} U_{n} \leq 0.
		\end{array} \right.
	\end{equation*}
A pulsating standing wave can be interpreted as a generating function (or hull function) for the plane-like minimizers in the $n$ direction (see \cite[Section 6]{variationaleinstein}).  Such functions always exist, but they can be discontinuous.

Indeed, if $U_{n}$ is a pulsating standing wave and it is a continuous function in $\mathbb{R} \times \mathbb{T}^{d}$, then the plane-like minimizers of \eqref{e.diffuseenergy} in the $n$ direction form a foliation by \cite[Proposition 1]{variationaleinstein}.  Therefore, Theorem \ref{t.genericdiffuse} shows that it is possible that there are no continuous pulsating standing waves in any direction. \end{remark}

\appendix

\section{Perron's Method}

In this appendix, for the sake of completeness, we prove a version of Perron's Method for sharp interfaces:
\begin{equation}\label{e.levelseteqn}
		-a(x) \textup{tr} \left( \left(\textup{Id} - \widehat{Du} \otimes \widehat{Du} \right) D^{2} u \right) - Da(x) \cdot Du - F\|Du\|   = 0.
	\end{equation}
	  It shows that provided there are sufficiently regular (but not necessarily smooth) sets $E_{*} \subset E^{*}$ defining stationary sub- and supersolutions, it is possible to find a stationary solution $E$ between them.

	\begin{proposition} \label{p.perron method}  Let $E^{*}, E_{*} \subset \mathbb{R}^{d}$ be open sets such that $\overline{E_{*}} \subset E^{*}$ and  $\overline{\mathbb{R}^{d} \setminus \overline{E^{*}}} = \mathbb{R}^{d} \setminus E^{*}$.  Define $\overline{v} \in LSC(\mathbb{R}^{d}; \{0,1\})$ and $\underline{v} \in USC(\mathbb{R}^{d}; \{0,1\})$ by
		\begin{equation*}
			\overline{v}(x) = \left\{ \begin{array}{r l}
									1, & \textup{if} \, \, x \in E^{*}, \\
									0, & \textup{otherwise},
							\end{array} \right. \quad 
			\underline{v}(x) = \left\{ \begin{array}{r l}
									1, & \textup{if} \, \, x \in \overline{E_{*}}, \\
									0, & \textup{otherwise}.
							\end{array} \right.
		\end{equation*}
	If $\overline{v}$ is a supersolution of \eref{levelseteqn} and $\underline{v}$, a subsolution, then there is an open set $E \subset \mathbb{R}^{d}$ satisfying $E_{*} \subset E \subset E^{*}$ such that $\ind_{E}$ is a discontinuous viscosity solution of \eref{levelseteqn} (i.e., $\ind_{E}$ is a viscosity supersolution and $\ind_{\overline{E}}$ is a viscosity subsolution).
	   \end{proposition}  
	   
In the proof, we will use semi-continuous envelopes.  In particular, given a locally bounded function $w : \mathbb{R}^{d} \to \mathbb{R}$, we denote by $w^{*}, w_{*} : \mathbb{R}^{d} \to \mathbb{R}$ the upper and lower semi-continuous envelopes defined by 
	\begin{align*}
		w^{*}(x) &= \lim_{\delta \to 0^{+}} \sup \left\{ w(y) \, \mid \, \|x - y\| < \delta \right\}, \\
		w_{*}(x) &= \lim_{\delta \to 0^{+}} \inf \left\{ w(y) \, \mid \, \|x - y\| < \delta \right\}.
	\end{align*}
	   
The proof of this level-set version of Perron's Method rests on the fact that if an open set defines a subsolution but fails to be a supersolution, then it is possible to find a larger subsolution containing it.  More precisely, we have

	\begin{proposition} \label{p.perron construction}  Suppose that $w \in USC(\mathbb{R}^{d}; \{0,1\})$ is a subsolution of \eref{levelseteqn}, $x_{0} \in \mathbb{R}^{d}$, $r > 0$, and there is a smooth function $\psi$ such that $w_{*} - \psi$ has a strict local maximum at $x_{0}$ in $B_{r}(x_{0})$ and $\|D\psi(x_{0})\| > 0$.  If $\psi$ satisfies the following differential inequality at $x_{0}$
	\begin{equation*}
		-a(x_{0}) \textup{tr} \left( \left(\textup{Id} - \widehat{D\psi}(x_{0}) \otimes \widehat{D\psi}(x_{0}) \right) D^{2} \psi(x_{0}) \right) - Da(x_{0}) \cdot D\psi(x_{0}) < F \|D\psi(x_{0})\|,
	\end{equation*}
then there is a $\tilde{w} \in USC(\mathbb{R}^{d}; \{0,1\})$, which is a subsolution of \eref{levelseteqn}, such that $\tilde{w} \geq w$ in $\mathbb{R}^{d}$, $\tilde{w} = w$ in $\mathbb{R}^{d} \setminus B_{r}(x_{0})$, and $\tilde{w} \not \equiv w$.  \end{proposition}

	\begin{proof}  The construction follows along the lines of the usual proof, see, e.g., \cite[Section 4]{user}.  A little care is needed to ensure that the gradient of the smooth subsolution built in the argument never vanishes.  At the end of the argument, we will have a subsolution $\hat{w}$ taking values in $\mathbb{R}$.  The proof is completed upon defining $\tilde{w}$ by 
		\begin{equation*}
			\tilde{w}(x) = \left\{ \begin{array}{r l}
							1, & \textup{if} \, \, \hat{w}(x) \geq \delta, \\
							0, & \textup{otherwise},
						\end{array} \right.
		\end{equation*}
	for some suitable small $\delta > 0$.  \end{proof}

	   	\begin{proof}[Proof of \pref{perron method}]  To start with, observe that the identity $\overline{\mathbb{R}^{d} \setminus \overline{E^{*}}} = \mathbb{R}^{d} \setminus E^{*}$ implies that, at the level of semi-continuous envelopes, we have $(\overline{v}^{*})_{*} = \overline{v}$.  This will be needed later in the argument.
		
		Let $\mathcal{S}$ denote the family of subsolutions $w \in USC(\mathbb{R}^{d}; \{0,1\})$ of \eref{levelseteqn} satisfying $\underline{v} \leq w \leq \overline{v}$ in $\mathbb{R}^{d}$.  Note that $\mathcal{S}$ is nonempty precisely because $\overline{E_{*}} \subset E^{*}$.  Let $v : \mathbb{R}^{d} \to \{0,1\}$ be the pointwise maximum of this family:
			\begin{equation*}
				v(x) = \sup \left\{ w(x) \, \mid \, w \in \mathcal{S} \right\}.
			\end{equation*}
		As the supremum of a family of subsolutions, $v^{*}$ is also a subsolution.
		
		We claim that $v_{*}$ is a supersolution.  To see this, assume that $x_{0} \in \mathbb{R}^{d}$ and $\psi$ is a smooth function such that $v_{*} - \psi$ has a strict local minimum at $x_{0}$ and $\|D \psi(x_{0})\| > 0$.  There are two cases to consider: (i) $v_{*}(x_{0}) = \overline{v}(x_{0})$ and (ii) $v_{*}(x_{0}) < \overline{v}(x_{0})$.
		
		In case (i), observe that $v_{*} \leq (\overline{v}^{*})_{*} = \overline{v}$.  Thus, $\overline{v} - \psi$ has a strict local minimum at $x_{0}$.  This implies that
			\begin{equation*}
				-a(x_{0}) \textup{tr} \left( \left( \textup{Id} - \widehat{D\psi}(x_{0}) \otimes \widehat{D\psi}(x_{0}) \right) D^{2} \psi(x_{0}) \right) - Da(x_{0}) \cdot D\psi(x_{0}) \geq F \|D\psi(x_{0})\|.
			\end{equation*}
		
		In case (ii), it necessarily follows that $\overline{v}(x_{0}) = 1$.  Since $E^{*}$ is open, there is an $r > 0$ such that $\{\overline{v} = 1\} = E^{*} \supset B_{r}(x_{0})$.  With this wiggle room, we can argue by employing a geometric version of the standard Perron argument: if $\psi$ does not satisfy the desired differential inequality at $x_{0}$, then \pref{perron construction} implies that there is a $w \in \mathcal{S}$ such that $w \geq v^{*}$ and $w \not \equiv v^{*}$.  However, this would contradict the definition of $v$.
		
		We proved that $v$ is a $\{0,1\}$-valued (discontinuous) solution of \eref{levelseteqn} with $\underline{v} \leq v \leq \overline{v}$.  Therefore, to conclude, we can set $E = \{v_{*} = 1\}$. \end{proof} 
		
		\section{Proof of \lref{aperturbation}}\label{s.appendixtechnical}

\begin{proof}[Proof of \lref{aperturbation}]  Recall that we are assuming as given a sequence $a_{k} \to a$ uniformly in $\mathbb{T}^{d}$ and sets $S$ and $(S_{k})_{k \in \mathbb{N}}$ of locally finite perimeter such that $S_{k} \to S$ in $L^{1}_{\text{loc}}(\mathbb{R}^{d})$ and $S_{k}$ is a Class A minimizer of $E_{a_{k}}$ for each $k$.  We wish to prove that $S$ is a Class A minimizer of the limiting energy $E_a$.  This is a standard argument: we follow \cite[Chapter 21]{maggibook}.  

As a preliminary step, we will argue that $a_{k} \, \mathcal{H}^{d-1} \restriction_{\partial S_{k}} \overset{*}{\rightharpoonup} a \, \mathcal{H}^{d-1} \restriction_{\partial S}$.  On the one hand, from $S_{k} \to_{L^{1}_{\textup{loc}}} S$, we know that if $\nu$ is any weak-$*$ limit point of $a_{k} \, \mathcal{H}^{d-1} \restriction_{\partial S_{k}}$, then 
 	\begin{equation*}
		a \, \mathcal{H}^{d-1} \restriction_{\partial S} \leq \nu.
	\end{equation*}  
At the same time, if $\nu = \lim_{j \to \infty} a_{k_{j}} \, \mathcal{H}^{d-1} \restriction_{\partial S_{k_{j}}}$ for some subsequence $(k_{j})_{j \in \mathbb{N}} \subset \mathbb{N}$, and if $x \in \mathbb{R}^{d}$, then, by testing the minimality property of $S_{k}$ with the set $E_{k} = (S \cap B(x,s)) \cup (S_{k} \setminus B(x,s))$, we find, for a.e. $s > 0$,
	\begin{equation*}
		\nu(B(x,s)) = \lim_{k \to \infty} E_{a}(S_{k}; B(x,s)) \leq E_{a}(S; B(x,s)).
	\end{equation*}
Hence $\nu \leq a \, \mathcal{H}^{d-1} \restriction_{\partial S}$ also holds.  In particular, $\nu = a \, \mathcal{H}^{d-1} \restriction_{\partial S}$ so this proves $a_{k} \, \mathcal{H}^{d-1} \restriction_{\partial S_{k}} \overset{*}{\rightharpoonup} a \, \mathcal{H}^{d-1} \restriction_{\partial S}$ as claimed.

Now we proceed with the heart of the argument.  Suppose $S'$ is a perturbation of $S$ so that $S \Delta S'$ is compactly contained in some ball $B_R(0)$.  Let us fix $R'<R$ so that $S' \Delta S \subset B_{R'}(0)$.  Making $R > 0$ larger if necessary, we can assume that $\mathcal{H}^{d-1}(\partial S \cap \partial R) = 0$.  By the co-area formula, we can choose $R'' \in (R',R)$ so that
	\begin{equation*}
		\lim_{k \to \infty} \mathcal{H}^{d-1}(\partial B_{R''} \cap (S_{k} \Delta S)) = 0, \quad \mathcal{H}^{d-1}(\partial S \cap \partial B_{R''}) = 0.
	\end{equation*}
Hence testing the minimality property of $S_{k}$ against the set $F_{k} = (S' \cap B_{R''}) \cup (S_{k} \setminus B_{R''})$, we find
	\begin{equation*}
		\lim_{k \to \infty} E_{a_{k}}(S_{k}; B_{R}) \leq E_{a}(S'; B_{R''}) + \lim_{k \to \infty} E_{a_{k}}(S_{k} ; B_{R} \setminus B_{R''}).
	\end{equation*}
Since $a_{k} \, \mathcal{H}^{d-1} \restriction_{\partial S_{k}} \overset{*}{\rightharpoonup} a \, \mathcal{H}^{d-1} \restriction_{\partial S}$ and $\mathcal{H}^{d-1}(\partial S \cap \partial B_{R}) = 0$, the left-most term is $E_{a}(S; B_{R})$.  Similarly, given that $a_{k} \to a$ uniformly and $\mathcal{H}^{d-1}(\partial S_{k} \cap B_{R})$ is uniformly bounded, the right-most term is $E_{a}(S; B_{R} \setminus B_{R''})$.  In particular,
	\begin{equation*}
		E_{a}(S; B_{R}) \leq E_{a}(S';B_{R''}) + E_{a}(S; B_{R} \setminus B_{R''}) = E_{a}(S';B_{R}).
	\end{equation*}
This proves that $S$ is a Class A minimizer of $E_{a}$. \end{proof}	

\section{Proof of Theorem \ref{T: boundary curve}} \label{A: boundary construction}

The proof of Theorem \ref{T: boundary curve} will rely on the following lemma, which demonstrates the utility of the second condition in the definition of regular $\mathbb{Z}^{2*}$-measurable set.
	
	\begin{lemma} \label{L: clockwise} If $A$ is a $\mathbb{Z}^{2*}$-measurable set satisfying condition (ii) in Definition \ref{D: regular cube sets} and $x_{0} \in \partial A \cap \mathbb{Z}^{2}$, then $x_{0}$ has a unique pair of neighbors $x_{1}, x_{-1} \in \mathbb{Z}^{2}$ such that:
		\begin{itemize}
			\item[(a)] $x_{1},x_{-1} \in \partial A$.
			\item[(b)] The edge $[x_{0}, x_{1}]$ is contained in $\partial A$ and it traverses $\partial A$ clockwise.
			\item[(c)] The edge $[x_{0}, x_{-1}]$ is contained in $\partial A$ and it traverses $\partial A$ counterclockwise.  (In particular, $[x_{-1},x_{0}]$ is contained in $\partial A$ and it traverses $\partial A$ clockwise.)
		\end{itemize}
	\end{lemma}
	
		\begin{proof}  The main idea of the proof is (ii) is really a local property.  Hence it is only necessary to consider the geometry of $A$ in the vicinity of $x_{0}$.  
		
		To simplify the notation, observe that, up to replacing $A$ by $A - x_{0}$, we can assume that $x_{0} = 0$.  This is no loss of generality since the translate of a regular $\mathbb{Z}^{2*}$-measurable set by an integer vector remains a regular $\mathbb{Z}^{2*}$-measurable set.
		
		     Since $0 \in \partial A \cap \mathbb{Z}^{2}$, we can fix a $z \in \mathbb{Z}^{2*}$ such that $0 \in z + [-1/2,1/2]^{2} \subset A$.  There is no loss of generality in assuming that $z = (1/2,-1/2)$.  (Indeed, we can reduce to this case by replacing $A$ by $\tau(A)$, where $\tau$ is one of the transformations $\tau(x,y) = (-x,y)$, $\tau(x,y) = (x,-y)$, or $\tau(x,y) = (-x,-y)$.  As in the case of integer translations, the class of regular $\mathbb{Z}^{2*}$-measurable sets is invariant under these transformations.)
		     
		     We complete the proof arguing by cases.  In particular, from the fact that $0 \in \partial A$, we can choose a $z' \in \mathbb{Z}^{2*}$ such that $0 \in z' + (-1/2,1/2)^{2} \subset \mathbb{R}^{2} \setminus A$.  There are three possible cases: (1) $z' = (1/2,1/2)$, (2) $z' = (-1/2,-1/2)$, and (3) $z' = (-1/2,1/2)$.  Since the other cases follow similarly, we will only treat case (1).
		     
We split further into three sub-cases: 
		\begin{itemize}
			\item[(1a)] $A \cap [-1,1]^{2} = (1/2,-1/2) + [-1/2,1/2]^{2}$, 
			\item[(1b)] $A \cap [-1,1]^{2} = [(1/2,-1/2) + [-1/2,1/2]^{2}] \cup [(-1/2,-1/2) + [-1/2,1/2]^{2}]$,
			\item[(1c)] $[-1,1]^{2} \setminus A = (1/2,1/2) + (-1/2,1/2)^{2}$.
		\end{itemize}
	Note that the only possibility that is missing is when $A \cap [-1,1]^{2} = [(1/2,-1/2) + [-1/2,1/2]^{2}] \cup [(-1/2,1/2) + [-1/2,1/2]^{2}]$.  However, that configuration is impossible due to condition (ii) in the definition of regular $\mathbb{Z}^{2}$-measurable set.  Hence the three cases above are exhaustive.
	
	In Case (1A), we take $x_{1} = (1,0)$ and $x_{-1} = (0,-1)$.  Uniqueness is immediate since $[0,x_{1}]$ and $[0,x_{-1}]$ are the only edges emanating from $0$ that are entirely contained in $\partial A$.
	
	In Case (1B), we take $x_{1} = (1,0)$ and $x_{-1} = (-1,0)$.  In Case (1C), we take $x_{1} = (1,0)$ and $x_{-1} = (0,1)$.  In either case, uniqueness follows as in Case (1A).\end{proof}
		
\begin{proof}[Proof of Theorem \ref{T: boundary curve}]  We define $P$ and $\{\gamma^{(j)}\}_{j \in P}$ recursively.  To begin, let $P = \emptyset$.  Suppose, for some $k \in \mathbb{N} \cup \{0\}$, that we are at stage $k$ of the construction.  At this stage, the set $P$ is simply $P = \{1,2,\dots,k\}$, and we have already constructed pairwise disjoint paths $\{\gamma^{(j)}\}_{j \in \{1,2,\dots,k\}}$, each of which traverse $\partial A$ clockwise.  If $\partial A = \bigcup_{j = 1}^{k} \{\gamma^{(j)}\}$, we terminate the construction.  Otherwise, we can choose an edge $e \subset \partial A \setminus \bigcup_{j \in \{1,2,\dots,k\}} \{\gamma^{(j)}\}$.  (To ensure that the algorithm eventually hits all of $\partial A$, we choose $e$ to be as close as possible to the origin, that is, choose $e = [x,x']$ such that $\min\{\|x\|,\|x'\|\}$ is as small as possible.)

Since $A$ is $\mathbb{Z}^{2*}$-measurable and $e \subset \partial A$, there is a $z \in \mathbb{Z}^{2*}$ such that $e \subset z + \partial [-1/2,1/2]^{2}$ and $z + [-1/2,1/2]^{2} \subset A$.  In particular, we can define $x_{0}, x_{1} \in \partial A \cap \mathbb{Z}^{2}$ so that $e = [x_{0}, x_{1}] \subset z + \partial [-1/2,1/2]^{2}$ and the edge $[x_{0}, x_{1}]$ traverses $\partial A$ clockwise.  

By Lemma \ref{L: clockwise}, there is a unique neighbor $x_{2} \in \partial A \cap \mathbb{Z}^{2}$ of $x_{1}$ such that $[x_{1}, x_{2}] \subset \partial A \cap \mathbb{Z}^{2}$ and $[x_{1}, x_{2}]$ traverses $\partial A$ clockwise.  We iterate this process, obtaining a path $\{x_{0},x_{1},x_{2},\dots\} \subset \partial A \cap \mathbb{Z}^{2}$ such that, for any $i \geq 0$, the pair $\{x_{i},x_{i+1}\}$ has the following property: 
	\begin{equation} \label{E: key part of construction 1}
		[x_{i}, x_{i + 1}] \subset \partial A \quad \text{and} \quad [x_{i}, x_{i + 1}] \, \, \text{traverses} \, \,  \partial A \, \, \text{clockwise.}
	\end{equation}

We claim that either $\{x_{0}, x_{1}, x_{2},\dots\}$ is a simple closed curve (in case it is bounded) or else it is an infinite simple path for which $\|x_{i}\| \to \infty$ as $i \to \infty$.  

\textit{Case 1: $\{x_{0},x_{1},\dots\}$ is bounded.}

Suppose that $\{x_{0},x_{1},x_{2},\dots\}$ is a bounded subset of $\mathbb{Z}^{2}$.  Since $\mathbb{Z}^{2}$ is discrete, we deduce that $\{x_{0},x_{1},x_{2},\dots\}$ must be a finite subset.  In particular, we can define $j \in \mathbb{N}$ by 
	\begin{equation*}
		j = \inf \left\{ m \in \mathbb{N} \, \mid \, x_{m} = x_{k} \, \, \text{for some} \, \, k > m \right\}.
	\end{equation*}
In particular, $x_{j} = x_{j + M}$ for some $M \in \mathbb{N}$.  Observe that $M \geq 2$ necessarily holds as $x_{j} \neq x_{j + 1}$ by construction.  

We claim that $j = 0$.  To see this, we assume that $j \geq 1$ and argue by contradiction.  Toward that end, note that $j + M - 1 \geq 1$ since $M \geq 2$.  Thus, $x_{j + M - 1}$ is well-defined.  

The construction implies that the edge $[x_{j + M - 1},x_{j}] = [x_{j + M - 1},x_{j + M}]$ is contained in $\partial A$ and it traverses $\partial A$ clockwise.  In particular, $[x_{j}, x_{j + M - 1}] \subset \partial A$ and it traverses $\partial A$ counterclockwise.  Thus, by Lemma \ref{L: clockwise} and the property \eqref{E: key part of construction 1} with $i = j - 1$, we must have $x_{j - 1} = x_{j + M - 1}$.  Yet $j - 1 < j$ so this contradicts the definition of $j$.

We conclude that $j = 0$ as claimed.  From the identity $x_{0} = x_{M}$, we deduce by construction that $x_{1} = x_{M + 1}$, and then this implies that $x_{k} = x_{M + k}$ for all $k \in \mathbb{N}$ by induction.  In particular, $\{x_{0},x_{1},\dots,x_{M}\} = \{x_{0},x_{1},\dots\}$.  Furthermore, by taking $M$ to be as small as possible, we find that the path $\gamma^{(k + 1)} : \{0,1,\dots,M\} \to \mathbb{Z}^{2}$ given by $\gamma^{(k + 1)}(i) = x_{i}$ is a simple closed path in $\mathbb{Z}^{2}$, as claimed.

\textit{Case: $\{x_{0},x_{1},\dots\}$ is unbounded.}

In this case, we set aside the infinite path $\{x_{0},x_{1},\dots\}$ for the moment.  Let us return to $x_{0}$, but this time we proceed in reverse.  By Lemma \ref{L: clockwise}, we can let $x_{-1} \in \partial A \cap \mathbb{Z}^{2}$ denote the unique neighbor of $x_{0}$ for which $[x_{0},x_{-1}] \subset \partial A$ and such that $[x_{0},x_{-1}]$ traverses $\partial A$ counterclockwise.  We then turn our attention to $x_{-1}$, letting $x_{-2}$ be the unique element of $\partial A \cap \mathbb{Z}^{2}$ with $[x_{-1},x_{-2}] \subset \partial A$ and such that $[x_{-1},x_{-2}]$ traverses the boundary counterclockwise.  Continuing in this way, we obtain a path $\{x_{0},x_{-1},x_{-2},\dots\}$ such that, for any $i \geq 0$, the following property holds:
	\begin{equation} \label{E: key part of construction 2}
		[x_{-i},x_{-(i+1)}] \subset \partial A \quad \text{and} \quad [x_{-i},x_{-(i+1)}] \, \, \text{traverses} \, \, \partial A \, \, \text{counterclockwise.}
	\end{equation}

We claim that $\{x_{0},x_{-1},x_{-2},\dots\}$ must also be unbounded.  Indeed, were this not the case, then we could argue as above to deduce that $\{x_{0},x_{-1},x_{-2},\dots\}$ is a simple closed curve $\{x_{0},x_{-1},x_{-2},\dots\} = \{x_{0},x_{-1},\dots,x_{-L}\}$ for some $L \in \mathbb{N}$, with $x_{-(i + L)} = x_{-i}$ for every $i \geq 0$.  Yet, by \eqref{E: key part of construction 1} and \eqref{E: key part of construction 2}, this would yield the following implications:
	\begin{itemize}
		\item[(A)] $[x_{0},x_{-(L - 1)}] = [x_{-L},x_{-(L-1)}] \subset \partial A$ traverses $\partial A$ clockwise.
		\item[(B)] $[x_{0},x_{1}] \subset \partial A$ traverses $\partial A$ clockwise.
	\end{itemize} 
Therefore, by Lemma \ref{L: clockwise}, $x_{1} = x_{-(L - 1)}$.  Applying the lemma again with $x_{1}$ in place of $x_{0}$, we see that $x_{2} = x_{-(L-2)}$.  By recursion, we conclude that $x_{k} = x_{-(L - k)}$ for every $k \leq L$.  This implies that $x_{0} = x_{L}$, hence $x_{k} = x_{L + k}$ for every $k \in \mathbb{N}$ by construction, but then this would contradict the assumption that $\{x_{0},x_{1},\dots\}$ is infinite.  We conclude that $\{x_{0},x_{-1},x_{-2},\dots\}$ is unbounded.

In sum, we showed that $\{x_{0},x_{1},\dots\}$ and $\{x_{0},x_{-1},\dots\}$ define two infinite paths emanating from $x_{0}$.  By arguments similar to those in the previous paragraph, $\{x_{0},x_{1},\dots\} \cap \{x_{0},x_{-1},\dots\} = \{x_{0}\}$.  Thus, if we define $\gamma^{(k + 1)} : \mathbb{Z} \to \mathbb{Z}^{2}$ by $\gamma^{(k + 1)}(i) = x_{i}$, then $\gamma^{(k + 1)}$ is a simple path in $\partial A \cap \mathbb{Z}^{2}$.

\textit{Conclusion}. 

It only remains to show that the updated collection of curves $\{\gamma^{(j)}\}_{j \in \{1,2,\dots,k+1\}} = \{\gamma^{(j)}\}_{j \in \{1,2,\dots,k\}} \cup \{\gamma^{(k+1)}\}$ is pairwise disjoint.  To see this, let us argue by contradiction: assume it is not pairwise disjoint.  In particular, there is a $j \in \{1,2,\dots,k\}$ such that $\{\gamma^{(j)}\} \cap \{\gamma^{(k+1)}\} \neq \emptyset$.

Since $\{\gamma^{(j)}\} \cap \{\gamma^{(k + 1)}\} \neq \emptyset$, it follows that we can choose $i, \ell \in \mathbb{Z}$ such that $\gamma^{(k + 1)}(i) = \gamma^{(j)}(\ell)$.  By Lemma \ref{L: clockwise} and the construction, it follows that $\gamma^{(k + 1)}(i + m) = \gamma^{(j)}(\ell + m)$ as long as $i + m$ is in the domain of $\gamma^{(k + 1)}$.  In particular, $\gamma^{(j)}$ is a finite, simple closed curve if and only if $\gamma^{(k + 1)}$ is a finite, simple closed curve.  Hence, in the finite case, we readily conclude that $\{\gamma^{(k + 1)}\} = \{\gamma^{(j)}\}$.  Yet we originally chose $e$ in such a way that $e \not\subset \{\gamma^{(j)}\}$, and yet the construction implies that $e = [x_{0},x_{1}] \subset \{\gamma^{(k + 1)}\}$, a contradiction.

If instead $\gamma^{(k + 1)}$ is infinite, then it is not hard to argue that actually $\gamma^{(k + 1)}(i + m) = \gamma^{(j)}(\ell + m)$ also holds for $m \in \mathbb{Z}$; this is a direct consequence of the construction and Lemma \ref{L: clockwise}.  Hence $\{\gamma^{(k + 1)}\} = \{\gamma^{(j)}\}$ follows, leading to a contradiction as in the finite case.

We conclude that our algorithm outputs a pairwise-disjoint collection of paths $\{\gamma^{(j)}\}_{j \in P \cup \{k+1\}}$ contained in $\partial A$.  We update $P$ by setting $P = \{1,2,\dots,k+1\}$, and repeat the process.

Finally, if the algorithm never terminates (i.e., if we can always find an edge $e$ as above), then we conclude with $P = \mathbb{N}$.  To see that $\partial A = \bigcup_{j \in \mathbb{N}} \{\gamma^{(j)}\}$ in this case, recall that the edge $e$ at each stage was chosen to be as close as possible to the origin.  Thus, since $\mathbb{Z}^{2}$ is discrete, the algorithm must eventually hit every boundary curve of $\partial A$. \end{proof}

\section{Proof of Lemma \ref{L: supersolution thing}} \label{A: boundary issue}

\begin{proof}[Proof of Lemma \ref{L: supersolution thing}]  Recall that $S_{\gamma} = \bigcup_{N \in \mathbb{N}} S_{\gamma_{[-N,N]}}$, $U_{\gamma} = \bigcup_{N \in \mathbb{N}} U_{\gamma_{[-N,N]}}$, and, for any compact set $K \subset \mathbb{R}^{2}$, there is an $N$ such that $S_{\gamma_{[-M,M]}} \cap K = S_{\gamma_{[-N,N]}} \cap K$ and $U_{\gamma_{[-M,M]}} \cap K = U_{\gamma_{[-N,N]}}$ for all $M \geq N$.  At the same time, for any $N \in \mathbb{N}$, the local supersolution $(S_{\gamma_{[-N,N]}},U_{\gamma_{[-N,N]}})$ is an edge.  Thus, there are three piecewise smooth curves $\eta^{(N)}_{\text{top}} : [-N,N] \to \mathbb{R}^{2}$, $\eta^{(N)}_{\text{bottom}} : [-N,N] \to \mathbb{R}^{2}$, and $\eta^{(N)}_{S} : [-N,N] \to \mathbb{R}^{2}$ such that
			\begin{align*}
				\partial U_{\gamma_{[-N,N]}} &= \eta^{(N)}_{\text{top}}([-N,N]) \cup \eta^{(N)}_{\text{bottom}}([-N,N]), \\ 
				\partial S_{\gamma_{[-N,N]}} &= \eta^{(N)}_{S}([-N,N]) \cup \eta^{(N)}_{\text{bottom}}([-N,N]), \\
				\eta^{(N)}_{S}((-N,N)) &\subset U_{\gamma_{[-N,N]}}.
			\end{align*} 
		In view of the facts mentioned above, it is possible to parametrize these curves so that, for any $M \geq N$ and $* \in \{\text{top},\text{bottom},S\}$, we have $\eta^{(M)}_{*} = \eta^{(N)}_{*}$ in the interval $[-N,N]$.  In particular, the limit $\eta_{*} = \lim_{N \to \infty} \eta^{(N)}_{*}$ exists for any choice of $*$.  Furthermore, 
			\begin{align*}
				\partial U_{\gamma} = \eta_{\text{top}}(\mathbb{R}) \cup \eta_{\text{bottom}}(\mathbb{R}), \quad \partial S_{\gamma} = \eta_{S}(\mathbb{R}) \cup \eta_{\text{bottom}}(\mathbb{R}), \quad \eta_{S}(\mathbb{R}) \subset U_{\gamma}.
			\end{align*}
		Since $\eta_{S}(\mathbb{R}) \subset U_{\gamma}$, to complete the proof that $\partial S_{\gamma} \setminus A \subset U_{\gamma}$, it only remains to show that $\eta_{\text{bottom}}(\mathbb{R}) \subset A$. 
		
		To begin the proof that $\eta_{\text{bottom}}(\mathbb{R}) \subset A$, we start by observing that either (a) $\eta_{\text{top}}(\mathbb{R}) \subset \mathbb{R}^{2} \setminus A$ and $\eta_{\text{bottom}}(\mathbb{R}) \subset \text{Int}(A)$ or (b) $\eta_{\text{top}}(\mathbb{R}) \subset \text{Int}(A)$ and $\eta_{\text{bottom}}(\mathbb{R}) \subset \mathbb{R}^{2} \setminus A$.  To understand why, first, note that the line segment $[\gamma(i),\gamma(i + 1)] \subset \text{fill}(U_{\gamma(i)}) \cup U_{[\gamma(i),\gamma(i + 1)]} \cup  \text{fill}(U_{\gamma(i + 1)}) \subset U_{\gamma}$ for any $i \in \mathbb{Z}$ by condition (vi) in Definition \ref{D: supersolution network}.  Thus, $\partial A \subset U_{\gamma}$ and hence the boundary curves of $U_{\gamma}$ do not intersect $\partial A$, that is, $\gamma_{*}(\mathbb{R}) \cap \partial A = \emptyset$ for $* \in \{\text{top},\text{bottom}\}$.  From this and the connectedness of $A$, we conclude that one boundary curve must be in $A$ and the other must be in $\mathbb{R}^{2} \setminus A$.  To conclude the proof, we establish that (a) holds.

		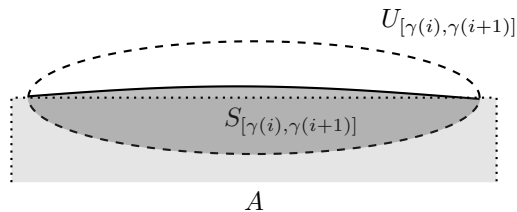
\begin{figure}[t]\label{f.condition5}
\begin{tikzpicture}[scale = .75]

\draw[draw=gray!50!white,fill=gray!50!white] plot[smooth,samples=100,domain=0:8] (\x+.3,{.2*sin(180*(\x+.3)/8)}) --
plot[smooth,samples=100,domain=0:-180] ({4.3+4*cos(\x)},{sin(\x)}); 

\draw[thick] plot[smooth,samples=100,domain=0:8] (\x+.3,{.2*sin(180*(\x+.3)/8)}) node[xshift=-2.5cm, yshift=-.3cm] {$S_{[\gamma(i),\gamma(i+1)]}$}; 
\draw[thick, dashed] plot[smooth,samples=100,domain=0:360] ({4.3+4*cos(\x)},{sin(\x)}) node[xshift=-.4cm, yshift=1cm] {$U_{[\gamma(i),\gamma(i+1)]}$}; 
\draw[thick,dotted,fill=gray, fill opacity = .2] (0,-1.5) -- (0,0) -- (8.6,0) -- (8.6,-1.5);

\draw (4.3, -1.5) node[anchor = north] {$A$};

\end{tikzpicture}
\caption{Condition (vi) of Definition \ref{D: supersolution network}.}
\end{figure}
		
	Finally, to see that (a) holds, pick any $i \in \mathbb{Z}$ and consider the edge (line segment) $[\gamma(i),\gamma(i + 1)] \subset \text{fill}(U_{\gamma(i)}) \cup U_{[\gamma(i), \gamma(i + 1)]} \cup \text{fill}(U_{\gamma(i + 1)}) \subset U_{\gamma}$.  By construction, $S_{[\gamma(i), \gamma(i + 1)]} \setminus [\text{fill}(U_{\gamma(i)}) \cup \text{fill}(U_{\gamma(i + 1)})] \subset S_{\gamma}$.  Further, by condition (vi) in Definition \ref{D: supersolution network}, see \fref{condition5}, we have $U_{[\gamma(i),\gamma(i + 1)]} \setminus S_{[\gamma(i),\gamma(i + 1)]} \subset \mathbb{R}^{2} \setminus A$.  Thus, that part of $\partial U_{\gamma}$ is outside of $A$, and the only remaining possibility is the part of $\partial U_{\gamma}$ contained in $S_{\gamma}$, namely, $\eta_{\text{bottom}}(\mathbb{R})$ is inside $A$.       \end{proof}

  \bibliographystyle{plain}
\bibliography{mobility-articles}
\end{document}